\documentclass[a4paper,abstracton]{scrartcl}

\usepackage[utf8]{inputenc}
\usepackage{geometry}
\usepackage{graphicx}
\usepackage{xcolor}
\usepackage{caption}
\usepackage{subcaption}
\usepackage{authblk}

\usepackage{tikz-cd}

\newcommand{\cU}{\mathcal{U}}

\newcommand{\X}{\mathcal{X}}

\newcommand{\rn}{\mathbb{R}^n}

\newcommand{\norm}[1][\cdot]{\left\| \kern.05em #1 \kern.05em \right\|}

\newcommand{\conic}[1]{\operatorname{conic}\left(#1\right)}

\usepackage{amsmath}
\usepackage{amsthm}
\usepackage{amssymb}

\newtheorem{theorem}{Theorem}[subsection]

\newtheorem{cor}[theorem]{Corollary}
\newtheorem{proposition}[theorem]{Proposition}
\newtheorem{definition}[theorem]{Definition}

\title{Problem-Driven Scenario Reduction and Scenario Approximation for Robust Optimization}
\author[1]{Jamie Fairbrother}
\author[2]{Marc Goerigk\footnote{Corresponding author. Email: marc.goerigk@uni-passau.de}}
\author[2]{Mohammad Khosravi}
\affil[1]{Department of Management Science, Lancaster University, United Kingdom}
\affil[2]{Business Decisions and Data Science, University of Passau, Germany}
\date{}

\begin{document}

\maketitle

\begin{abstract}
In robust optimization, we would like to find a solution that is immunized against all scenarios that are modeled in an uncertainty set. Which scenarios to include in such a set is therefore of central importance for the tractability of the robust model and practical usefulness of the resulting solution. We consider problems with a discrete uncertainty set affecting only the objective function. Our aim is reduce the size of the uncertainty set, while staying as true as possible to the original robust problem, measured by an approximation guarantee. Previous reduction approaches ignored the structure of the set of feasible solutions in this process. We show how to achieve better uncertainty sets by taking into account what solutions are possible, providing a theoretical framework and models to this end. In computational experiments, we note that our new framework achieves better uncertainty sets than previous methods or a simple K-means approach.
\end{abstract}

\noindent\textsf{\textbf{Keywords:}} Robust optimization; scenario reduction; approximation algorithms; scenario clustering

\section{Introduction}
\label{sec:introduction}

We study decision making problems where there is uncertainty in the objective function.
Such a situation occurs whenever we know the alternatives that are available to us, but cannot precisely predict their consequences.
Decision making processes in practice are usually affected by some source of uncertainty; as an example, consider we plan a route in advance, but do not know the precise traffic at the time of travel.

Different decision making paradigms have been studied to include this uncertainty a priori.
Most prominently, we can apply stochastic optimization (see, e.g., \cite{powell2019unified}) if we have an estimate of the probability distribution over the possible outcomes available, or we can apply robust optimization (see, e.g., \cite{ben2009robust,bookbertsimasdenHertog2022}) if this distribution is not known or we are particularly risk-averse.

While this paper focuses on robust optimization, both methods require some description of the set of possible outcomes that we would like to include in our planning.
In robust optimization, structured uncertainty sets can have significant advantages in terms of tractability (frequently used are budgeted uncertainty sets in combinatorial optimization for this reason, see \cite{bertsimas2003robust}).
Such sets often need to be constructed from discrete observed data.
From a modeling perspective therefore, using a discrete set of outcomes is simpler and offers more flexibility.
In the case of stochastic optimization, problems often become intractable when one uses a continuous probability distribution, and so using a discrete set of outcomes is a requirement.
In both stochastic and robust optimization, such a discrete set of outcomes is referred to as a \emph{scenario set}.

Unfortunately, the flexibility of using scenario sets for robust optimizations comes at  a price: robust optimization problems with scenario sets become much harder to solve than their nominal (not uncertain) counterpart. Such hardness results are documented in the survey \cite{kasperski2016robust} and the recent book \cite{robook}, amongst other such overviews.

A more detailed view of the computational complexity of combinatorial problems with discrete uncertainty sets can be obtained by studying the approximability of such problems.
There are two different views that can be applied: on the one hand, if we assume that the uncertainty set is discrete, but the cardinality of the set is a fixed constant and thus not part of the problem input (in the same way as we may define bicriteria optimization as optimizing a fixed amount of two objectives simultaneously), then we can frequently achieve strong approximability results.
For example, robust shortest path, spanning tree, knapsack or assignment problems all have a fully polynomial-time approximation scheme (FPTAS) if the number of scenarios is constant.
Problems that have an FPTAS may be considered the ``easiest'' of all NP-hard optimization problems (see, e.g., \cite{korte2011combinatorial}).
On the other hand, if we allow an arbitrary number of scenarios $N$ as part of the problem input, all of the mentioned problems become strongly NP-hard and not approximable within a constant factor.
Many combinatorial problems are not even approximable within $O(\log^{1-\epsilon} N)$ for any $\epsilon > 0$, while allowing for a simple $O(N)$ approximation method.
Essentially, this means that the more scenarios we have in a robust optimization problem, the more difficult it becomes.

For this reason, a natural question is to ask how far we can reduce the number of scenarios without changing the problem much, or even not at all.
In the following, we refer to the task of removing scenarios from the problem such that the optimal solution and its objective value remains the same as ``scenario reduction'' (a more formal definition is given later in the paper).
If we are looking for a good representation of a high-cardinality uncertainty set by a smaller set of given cardinality, we call the task ``scenario approximation''.
Note that this terminology differs to how it is used in the stochastic optimization literature where the term scenario reduction is used in the general sense of reducing the number of scenarios whether or not the objective value is preserved exactly.

Scenario approximation (and more generally scenario generation) has been thoroughly studied in the area of stochastic optimization, see \cite[Chapter~4]{king2012modeling} for a recent overview.
This presence of a probability distribution can facilitate the analysis of the resulting error when scenarios are removed from the problem through the use of probability metrics, see \cite{DupacovaEA03} for example.
A common distinction made in the scenario reduction literature is to distinguish between ``distribution-driven'' approaches to scenario reduction, and ``problem-driven'' approaches.
The former types of approaches (such as that cited above) only concern themselves with approximating the distribution in some way, while the latter type of approach aims to exploit problem features in order to get a more concise representation of the uncertainty.
For example \cite{henrion2022problem} proposes a probability metric on the stochastic problem which could be used for scenario approximation.
A particularly relevant for this work is the problem-driven approach for scenario approximation proposed in \cite{fairbrother2019problem} for problems involving tail-risk.
These are problems where one aims to minimize extreme losses in the tail of a distribution, and robust optimization problems, where one aims to minimize the worst case, could be viewed as a special case of this.

For robust optimization problems, there seem to be fewer results on principled scenario reduction and approximation.
One long-standing approximation method (which provides the aforementioned $O(N)$ approximation guarantee for all robust problems where the nominal problem can be solved in polynomial time) is to simply take the average of all scenarios, and to solve the nominal problem with respect to this average case.
The survey \cite{aissi2009min} provides an excellent discussion of this approach.
Taking the midpoint of all scenarios is also popular in the related min-max regret problem setting, see \cite{kasperski2006approximation,conde2012constant}, as well as the tighter a fortiori analysis in \cite{chassein2015new}.
It is possible to consider the best possible (in the sense of the resulting approximation guarantee) representation of a scenario set using a single scenario as an optimization problem, see \cite{goerigk2019representative}.
Furthermore, it is possible to estimate the error that occurs when multiple scenarios are merged together in arbitrary order, see \cite{chassein2018scenario} and \cite{chassein2020approximating}.

To construct uncertainty sets from discrete observations, several data-driven methods have been proposed in the literature. As examples, we mention methods to construct uncertainty sets based on hypothesis tests \cite{bertsimas2018data}, support vector clustering \cite{shang2017data}, or neural networks \cite{goerigk2023data}. The recent work \cite{wang2023learning} aims at constructing ellipsoidal uncertainty sets while taking the quality of the resulting robust solution into account.

Most closely related to this paper are the results provided in \cite{goerigk2023optimal}, where optimization models and heuristics for the scenario approximation problem are introduced.
Given a discrete uncertainty set, it is argued that solving an NP-hard problem (heuristically) in a preprocessing step results in a smaller set of given size along with an approximation guarantee that can be used to estimate the quality of the resulting robust solution before solving the corresponding robust problem.
This guarantee holds independently of the set of feasible solutions for the decision maker, which is not taken into account in the optimization.

In this paper, we extend the current toolkit of scenario reduction and approximation for robust optimization problems with discrete uncertainty in the objective to be problem-driven. We formally define the concept of sufficient scenario sets that result in the same objective value as a given uncertainty set and show that the reduction problem is NP-hard. We introduce four sufficient criteria to identify sufficient scenario sets and compare their respective strength. Furthermore, we provide an alternative perspective based on conic duality. We show how to find reduced uncertainty sets using mixed-integer (non-)linear programming. The reduction approaches and their respective optimization models are generalized to approximation approaches as well. In computational experiments, we show that our new reduction and approximation methods outperform previous approaches that do not take the structure of the set of feasible solutions into account.

The remainder of this paper is structured as follows. We study scenario reduction approaches in Section~\ref{sec:reduction} and extend them to the scenario approximation setting in Section~\ref{sec:approximation}. Our computational experiments are presented in Section~\ref{sec:experiments}. We conclude the paper and point out further research questions in Section~\ref{sec:conclusions}. In an appendix, we discuss generalizations to non-linear objective functions (Appendix~\ref{sec:appendix}) and further experimental results (Appendix~\ref{sec:expappendix}).

\section{Scenario Reduction}
\label{sec:reduction}

\subsection{Setting}

We consider robust optimization problems with uncertain linear objective of the form
\begin{equation}
\min_{\pmb{x}\in X} \max_{\pmb{c}\in \cU}\ f(\pmb{x},\pmb{c}),\ \text{ with } f(\pmb{x},\pmb{c}) = \sum_{i\in[n]} c_i x_i, \tag{RO} \label{eq:full-scen-robust-opt}
\end{equation}
where $X\subseteq\rn_+$ is the set of feasible decisions, $\cU\subseteq\rn_+$ is a set of future possible scenarios against we wish to protect our solution (the uncertainty set), and $f: X\times \cU \rightarrow \mathbb{R}_+$ represents some cost or loss we would like to minimize. We use the notation $[n] = \{1,\ldots,n\}$ for brevity. Throughout the paper, we assume that optimization problems are sufficiently well-posed that a maximizer or minimizer is indeed attained, which allows us to write $\min$ and $\max$ instead of $\inf$ and $\sup$.

In the case of discrete uncertainty $\cU=\{\pmb{c}^1,\ldots,\pmb{c}^N\}$, the computational complexity of solving \eqref{eq:full-scen-robust-opt} typically scales badly with the number of scenarios in $\cU$, in particular if $X\subseteq\{0,1\}^n$ represents a combinatorial problem. To improve tractability, we would thus like this set to be as small as possible. Even if $\cU$ is not discrete, it is often beneficial to replace the uncertainty with a smaller set, e.g., we may reduce a continuous set to a finite set or may remove scenarios that are unnecessary.

In this paper, we consider the two problems of scenario reduction and scenario approximation. In scenario reduction, the aim is to find a subset of scenarios $\cU' \subseteq \cU$ that results in an equivalent robust optimization problem, i.e., we would like to remove unnecessary scenarios from the problem. In scenario approximation, the aim is to find a set $\cU'$, which is not necessarily contained in $\cU$, of fixed cardinality, such that the original problem is as closely approximated as possible.

Furthermore, in both settings we would like to find sets $\cU'$ that are not designed for a single robust optimization problem of the form \eqref{eq:full-scen-robust-opt}, but rather for a set of possible problems that may have different constraints. For example, we would like to apply the scenario reduction or approximation approach not for a single $s$-$t$-pair in a graph with uncertain edge costs, but in a way that robust problems with respect to multiple $s$-$t$-pairs are considered simultaneously. The advantage of this approach is that both scenario reduction or approximation only need to be performed once for an uncertainty set $\cU$, and then can be applied to multiple robust optimization problems. To include this perspective, let us assume that the uncertainty set $\cU$ applies to a family $\{X_1,X_2,\ldots\}$ of different sets of feasible solutions with the same dimension. We further assume that a set $\X\subseteq\mathbb{R}^n_+$ is known such that $X_i \subseteq \X$ for all $i$, that is, $\X$ reflects properties that all sets of feasible solutions have in common (such as a cardinality constraint, for example). If only a single instance $X$ needs to be considered, we can simply set $\X=X$.
In what follows, we apply scenario reduction and approximation to all solutions in this set $\X$.

\subsection{Problem Definitions and Hardness}

We study the scenario reduction problem in this section and first consider different approaches to define ''equivalent'' problems more formally.

\begin{definition}
A set $\cU'\subseteq\rn$ is said to be
\begin{itemize}
\item an \emph{optimality-sufficient} set, if
\begin{equation}
\min_{\pmb{x}\in\X} \max_{\pmb{c}\in\cU} f(\pmb{x},\pmb{c}) = \min_{\pmb{x}\in\X} \max_{\pmb{c}\in\cU'} f(\pmb{x},\pmb{c})
\end{equation}
and the set of optimal solutions is the same.
\item a \emph{sufficient} set, if
\begin{equation}
    \label{eq:sufficiency}
    \max_{\pmb{c}\in\cU'} f(\pmb{x},\pmb{c}) = \max_{\pmb{c}\in\cU} f(\pmb{x},\pmb{c}) \qquad \text{ for all } \pmb{x}\in\mathcal{X}.
  \end{equation}
\end{itemize}
Furthermore, if additionally $\cU'\subseteq \cU$ holds, then $\cU'$ is an (optimality-)sufficient \emph{subset}. $\cU'$ is minimally (optimality-)sufficient, if any proper subset $\tilde{\cU}\subsetneq\cU'$ is not (optimality-)\-sufficient.
\end{definition}

If a set $\cU'$ is optimality-sufficient, then solving \eqref{eq:full-scen-robust-opt} with $\cU'$ instead of $\cU$ will lead to a solution which is also optimal for the original problem.
A minimally optimality-sufficient subset of $\cU$ is the best we can achieve in scenario reduction.

We construct a small example to highlight difficulties arising from the definition of optimality-sufficient sets.
Let us assume that there is a unique solution $\pmb{x}^*\in\X$ to the robust problem \eqref{eq:full-scen-robust-opt} with a linear function $f$, $\X\subseteq\{0,1\}^n$ and $\cU\subseteq\rn_{>0}$. Let $v=\max_{\pmb{c}\in\cU} f(\pmb{x}^*,\pmb{c})$ be its value and let us further assume that $\pmb{x}^*\neq \pmb{0}$. Then we can construct a scenario $\pmb{c}'\in\rn$ so that $c'_i = v/\|\pmb{x}^{*}\|_1$ for all $i\in[n]$ with $x_i=1$, and $c'_i = v + 1$ otherwise. This means that $\cU'=\{\pmb{c}'\}$ is an optimality-sufficient set, and indeed the smallest possible such set. This example illustrates that the concept of optimality-sufficient sets is too general to be useful in practice, as we would like to avoid solving \eqref{eq:full-scen-robust-opt} in the first place. However, it is easy to verify that any sufficient set is also an optimality-sufficient set. For this reason, we focus on finding sufficient (sub)sets.

The following complexity result was presented in the context of iterative constraint generation to solve robust problems.

\begin{theorem}[\cite{goerigk2022data}]\label{th:hardness}
Let $X$ be a polyhedron given by an outer description, and let $\cU=\{\pmb{c}^1,\ldots,\pmb{c}^N\}$ be a discrete set. Then the problem
\[ \max_{\substack{\mathcal I \subseteq [N]\\ |\mathcal I |\le k}} \min_{\pmb{x}\in X} \max_{j\in \mathcal{I}} (\pmb{c}^j)^\top \pmb{x} \]
is NP-hard, if $k$ is part of the input.
\end{theorem}

This result indicates that it is hard to identify a subset of scenarios of given size that maximize the objective value. This is closely related to scenario reduction, in which case we want the robust objective value of the reduced problem to be the same as the original robust objective value. More formally, we can consider the following decision problem (Reduce): Given a finite set $\X$, a discrete uncertainty set $\cU=\{\pmb{c}^1,\ldots,\pmb{c}^N\}$, and an integer $k\le N$. Does there exist a sufficient subset $\cU'\subseteq \cU$ of size at most $k$?

\begin{theorem}
The decision problem (Reduce) is NP-complete.
\end{theorem}
\begin{proof}
Let an instance of the Set Cover problem be given, consisting of a set of items $E=\{e_1,\ldots,e_n\}$, a collection $\mathcal{S}=\{S_1,\ldots,S_N\}$ of sets $S_j\subseteq[n]$, and an integer $k$. We assume that each $e_i$ is contained in at least one $S_j$. The question we ask is if there exists a set $\mathcal{I}\subseteq[N]$ with cardinality at most $k$ such that each $e_i$ is contained in at least one $S_j$ with $j\in \mathcal{I}$.

Set $\X= \{\pmb{x}\in\{0,1\}^n : \sum_{i\in[n]} x_i = 1\}$ and $\cU=\{\pmb{c}^1,\ldots,\pmb{c}^N\}$ with $c^j_i =1$ if $e_i \in S_j$, and $c^j_i=0$ otherwise. We have $f(\pmb{x},\pmb{c}) = \sum_{i\in[n]} c_ix_i = 1$ for all $\pmb{c}\in\cU$ and $\pmb{x}\in \X$.
Observe that any subset $\mathcal{I}\subseteq[N]$ is a set cover if and only if $\max_{j\in \mathcal{I}} f(\pmb{x},\pmb{c}^j) =  1$ for all $\pmb{x}\in \X$. Hence, the Set Cover problem is a Yes-Instance, if and only if the exists a sufficient subset $\cU'$ of size at most $k$.
\end{proof}

\subsection{Sufficiency Criteria}

We now introduce criteria that allow us to verify if a set $\cU'\subseteq \cU$ is indeed a sufficient subset, that is, using $\cU'$ instead of $\cU$ leads to the same robust objective value for any $\pmb{x}\in\X$.

\begin{definition}\label{def:types}
We define the following properties for a set $\cU'\subseteq \cU=\{\pmb{c}^1,\ldots,\pmb{c}^N\} \subseteq\rn_+$ and with respect to a set of feasible solutions $\X\subseteq\rn_+$.
\begin{itemize}
\item Sufficiency type (i):
\[ \forall \pmb{c}\in \cU \ \exists \pmb{c}'\in\cU' : \pmb{c}' \ge \pmb{c} \]

\item Sufficiency type (ii):
\[ \forall \pmb{c}\in\cU, \pmb{\alpha}\in\mathbb{R}^n_+\ \exists \pmb{c}'\in\cU' : \pmb{\alpha}^\top \pmb{c}' \ge \pmb{\alpha}^\top \pmb{c} \]

\item Sufficiency type (iii):
\[ \forall \pmb{c}\in\cU\ \exists \pmb{c}'\in\cU' \ \forall \pmb{x}\in\X : \pmb{x}^\top \pmb{c}' \ge \pmb{x}^\top\pmb{c} \]

\item Sufficiency type (iv):
\[ \forall \pmb{c}\in\cU, \pmb{x}\in\X\ \exists \pmb{c}'\in\cU' : \pmb{x}^\top \pmb{c}' \ge \pmb{x}^\top\pmb{c} \]

\end{itemize}
\end{definition}

Before discussing properties of sets that are of types (i-iv), we first consider the intuition behind them. Type (i) requires set $\cU'$ to contain elements with which we can dominate any of the original scenarios in each component. That is, for each $\pmb{c}\in\cU$, we can provide (at least) one scenario $\pmb{c}'\in\cU'$ that is at least as large. In type (ii), we do not require $\pmb{c}'$ to be larger than $\pmb{c}$ in every component. Instead, we only require that a direction $\pmb{\alpha}$ exists, so that $\pmb{c}'$ is farther than $\pmb{c}$ in this direction.

Both types (i) and (ii) do not consider the set of feasible solutions $\X$.
Types (iii) and (iv) on the other hand can be thought of as generalizations of type (i) and (ii) which depend on $\X$.
For type (iii), we want that some $\pmb{c}'\in\cU'$ exists for each $\pmb{c}\in\cU$ such that the objective value in $\pmb{c}'$ is at least as large as the objective value in $\pmb{c}$ for every feasible solution. In comparison, type (iv) allows us to choose the scenario $\pmb{c}'$ after the solution $\pmb{x}\in\X$ is given; that is, for any combination of original scenario $\pmb{c}$ and solution $\pmb{x}$, there is some scenario in $\cU'$ that gives an objective value that is at least as large for this particular $\pmb{x}$.

\begin{theorem}\label{th:sufficient}
Let $\cU'\subseteq \cU$ be any set that fulfills one of the properties (i)-(iv). Then, $\cU'$ is a sufficient subset.
\end{theorem}
\begin{proof}
We need to show that
\[ \max_{\pmb{c}\in\cU'} \sum_{i\in[n]} c_i x_i = \max_{\pmb{c}\in\cU} \sum_{i\in[n]} c_i x_i \]
for all $\pmb{x}\in\X$. As $\cU'\subseteq\cU$, it already holds that $\max_{\pmb{c}\in\cU'} \sum_{i\in[n]} c_i x_i \le \max_{\pmb{c}\in\cU} \sum_{i\in[n]} c_i x_i$ for all $\pmb{x}\in\X$. So let us assume that some $\pmb{x}\in\X$ is given. Let $\pmb{c}(\pmb{x})\in\cU$ be a maximizer of $\max_{\pmb{c}\in\cU} \sum_{i\in[n]} c_i x_i$.
We consider each type separately.
\begin{itemize}
\item Type (i): By definition, there is some $\pmb{c}'\in\cU'$ such that $\pmb{c}' \ge \pmb{c}(\pmb{x})$. Hence, \[ \max_{\pmb{c}\in\cU'} \sum_{i\in[n]} c_i x_i \ge \sum_{i\in[n]} c'_ix_i \ge \sum_{i\in[n]} c_i(\pmb{x}) x_i = \max_{\pmb{c}\in\cU} \sum_{i\in[n]} c_i x_i. \]

\item Type (ii): By definition, there is $\pmb{c}'\in\cU'$ such that $\pmb{\alpha}^\top \pmb{c}' \ge \pmb{\alpha}^\top\pmb{c}$ for any $\pmb{\alpha}\in\rn_+$. Set $\pmb{\alpha} = \pmb{x}$, and the claim follows as before.

\item Type (iii): By definition, there is $\pmb{c}'\in\cU'$ such that $\hat{\pmb{x}}^\top\pmb{c}' \ge \hat{\pmb{x}}^\top \pmb{c}$ for all $\hat{\pmb{x}}\in\X$. Set $\hat{\pmb{x}} = \pmb{x}$, and the claim follows as before.

\item Type (iv). By definition, for all $\hat{\pmb{x}}\in\X$ and any $\hat{\pmb{c}}$ there is
$\pmb{c}'\in\cU'$ such that $\hat{\pmb{x}}^\top\pmb{c}' \ge \hat{\pmb{x}}^\top \hat{\pmb{c}}$. Set $\hat{\pmb{x}}=\pmb{x}$ and $\hat{\pmb{c}} = \pmb{c}(\pmb{x})$, and the claim follows as before.

\end{itemize}
\end{proof}

We consider a two-dimensional example for these definitions in Figure~\ref{fig:example}. Every point corresponds to one scenario. Here, let us assume that $\X=\{(1,0),(0,1)\}$. In this case, it can easily be verified that the definitions of type (i) and (iii) are equivalent.
Due to the structure of $\X$, we only need two scenarios to fulfill property (iv), while we need five scenarios for type (i).

\begin{figure}[htb]
\begin{center}
\includegraphics[width=0.6\textwidth]{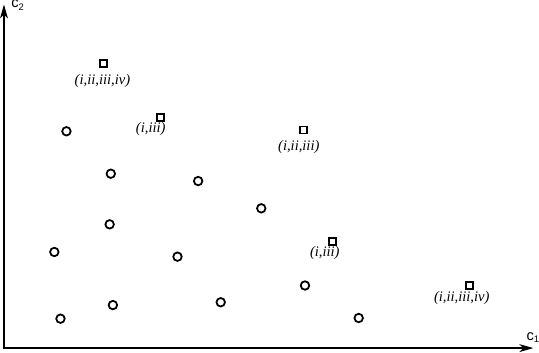}
\caption{Example for dominated scenarios.}\label{fig:example}
\end{center}
\end{figure}

While each of the four properties (i-iv) is sufficient to prove that $\cU'$ is a sufficient subset by Theorem~\ref{th:sufficient}, they may not be necessary. We compare their strength in the following result.

\begin{theorem}\label{th:implications}
A set of type (i) is also a set of type (ii) and (iii). Sets of type (ii) or (iii) are also sets of type (iv). The following diagram shows their relationships:
\begin{center}
\begin{tikzcd}[arrows=Rightarrow]
& (ii) \arrow[dr] \\
(i) \arrow[ur] \arrow[dr] & & (iv) \\
& (iii) \arrow[ur]
\end{tikzcd}
\end{center}
\end{theorem}
\begin{proof}
Let any $\pmb{c}\in\cU$ be given.
\begin{itemize}
\item (i) $\Rightarrow$ (ii): By definition of (i), there is $\pmb{c}'\in\cU'$ such that $c'_i \ge c_i$ for all $i\in[n]$. This means that for any $\pmb{\alpha}\in\rn_+$, we have $\sum_{i\in[n]} \alpha_i c'_i \ge \sum_{i\in[n]} \alpha_i c_i$ as claimed.

\item (i) $\Rightarrow$ (iii): By definition of (i), there is $\pmb{c}'\in\cU'$ such that $c'_i \ge c_i$ for all $i\in[n]$. This means that for any $\pmb{x}\in\X\subseteq\rn_+$, we have $\sum_{i\in[n]} x_i c'_i \ge \sum_{i\in[n]} x_i c_i$ as claimed.

\item (ii) $\Rightarrow$ (iv): As $\X\subseteq\mathbb{R}^n_+$, it follows directly that the condition of type (ii) is stronger than the condition of type (iv).

\item (iii) $\Rightarrow$ (iv): Let some $\pmb{c}\in\cU$ be given. By definition of (iii), there is some $\pmb{c}'\in\cU'$ such that $\pmb{x}^\top\pmb{c}' \ge \pmb{x}^\top\pmb{c}$ for all $\pmb{x}\in\X$. This means that For any given $\pmb{x}\in\X$, we can use $\pmb{c}'$ to fulfill condition (iv).
\end{itemize}
\end{proof}

The implications (ii) $\Rightarrow$ (iii), (iii) $\Rightarrow$ (ii), and (iv) $\Rightarrow$ (i) do not hold in general.

\begin{cor}
Let $\cU_{(i)}$, $\cU_{(ii)}$, $\cU_{(iii)}$, and $\cU_{(iv)}$ be sets of type $(i)$, $(ii)$, $(iii)$, and $(iv)$, respectively, with minimum cardinality. Then it holds that $|\cU_{(i)}| \ge |\cU_{(ii)}| \ge |\cU_{(iv)}|$ and $|\cU_{(i)}| \ge |\cU_{(iii)}| \ge |\cU_{(iv)}|$.
\end{cor}

Another property of these criteria is that as the feasible set becomes smaller, the dominance types become stronger. That is, $\cU'$ is type (iii) (type (iv)) with respect to $\X_{1}$ implies $\cU'$ is type (iii) (type (iv)) with respect to $\X_{2}$, when $\X_{2} \subseteq \X_{1}$.

\subsection{Alternative Characterization of Sufficiency Criteria}
\label{sec:altern-defin-type-1}

We now provide alternative characterizations for some of the sufficiency type (iii) from Definition~\ref{def:types}.
This will provide further insights into this sufficiency type and will be useful for the optimization formulations given in the following sections.

We start with type (iii) which can be expressed in terms of a conic ordering.
Recall that for a pointed convex cone $\mathcal{K}$, a partial ordering and strict partial ordering can be defined as follows:
\begin{align*}
  \pmb{c}_1 &\preceq_K \pmb{c}_2 \Leftrightarrow \pmb{c}_2 - \pmb{c}_1 \in \mathcal{K}.
\end{align*}
A special case of this is when $\mathcal{K} = \rn_{+}$ for which the conic ordering is equivalent to the standard component-wise ordering, that is:
\begin{equation}
  \label{eq:conic-standard}
  \pmb{c}_{1} \preceq_{\rn_{+}} \pmb{c}_{2} \Leftrightarrow \pmb{c}_{1} \leq \pmb{c}_{2}.
\end{equation}
Also, for a given cone $\mathcal{K}\subseteq\rn$, its dual is defined as follows:
\begin{equation*}
\mathcal{K}^* = \{ \pmb{u}\in\rn : \forall \pmb{c}\in \mathcal{K}\ \pmb{c}^\top \pmb{u} \ge 0\},
\end{equation*}
For example, if $\mathcal{K}=\mathbb{R}^n_+$, then $\mathcal{K}^*=\mathcal{K}$.

The following proposition now gives the alternative characterization for type (iii):

\begin{proposition}
  \label{prop:conic-type-iii}
  Suppose $\cU' \subseteq \cU$, and let $\mathcal{K} = \conic{\X}$. Then $\cU'$ is type (iii) sufficient if and only if for all $\pmb{c}\in\cU$ there exists $\pmb{c}'\in\cU'$ such that $\pmb{c}' \succeq_{\mathcal{K}^*} \pmb{c}$.
\end{proposition}

\begin{proof}
  It suffices to show that $\pmb{c}_{1}^\top \pmb{x} \leq \pmb{c}_{2}^\top \pmb{x}\ \forall \pmb{x}\in\X \Longleftrightarrow \pmb{c}_{1} \preceq_{\mathcal{K}^{*}} \pmb{c}_{2}$.

Suppose first that $\pmb{c}_{1}^\top \pmb{x} \leq \pmb{c}_{2}^\top \pmb{x}\ \forall \pmb{x}\in\X$.
Then,
\begin{equation}
  \label{eq:ineq-conic-prop}
  (\pmb{c}_{2} - \pmb{c}_{1})^\top \pmb{x} \geq 0 \text{ for all } \pmb{x}\in\X.
\end{equation}
For any $\pmb{y}\in\conic{\X}$, we have $\pmb{y} = \lambda_{1} \pmb{x}_{1} + \ldots + \lambda_{k}\pmb{x}_{k}$ where $\pmb{x}_{i}\in\X$ for $i=1,\ldots,k$ and $\lambda \geq 0$. Now,
\begin{equation*}
  (\pmb{c}_{2} - \pmb{c}_{1})^\top \pmb{y} = \sum_{i=1}^{k} \lambda_{i} \underbrace{(\pmb{c}_{2} - \pmb{c}_{1})^\top \pmb{x}_{i}}_{\geq 0 \text{ by } \eqref{eq:ineq-conic-prop}} \geq 0.
\end{equation*}
Hence $\pmb{c}_1 \preceq_{\mathcal{K}^{*}} \pmb{c}_{2}$.

On the other hand, suppose that $\pmb{c}_{1}^\top \pmb{x} > \pmb{c}_{2}^\top \pmb{x}$ for some $\pmb{x}\in\X$.
Then, $(\pmb{c}_{2}-\pmb{c}_{1})^\top \pmb{x} < 0$ which implies that $\pmb{c}_{1}\npreceq_{\mathcal{K}^{*}} \pmb{c}_{2}$ as required.

\end{proof}

We illustrate in Figure~\ref{fig:ill-example} minimal sufficient sets of type (i), (ii), (iii) and (iv) for a randomly sampled two-dimensional set of 30 scenarios $\mathcal{U}$ and where the feasible set of solutions has conic hull $\mathcal{K}$.
We see that the type (i) sufficient set, which contains 8 scenarios, contains all other minimal sufficient sets, and is the weakest type of sufficient set.
The type (ii) set contains 6 points, excluding the two points in the type (i) sufficient set which are not supported efficient.
For type (iii), note that a scenario $\pmb{x}$ only belongs to the minimal sufficient set if there is no other scenario $\pmb{y}$ for which we have $\pmb{y} - \pmb{x} \in \mathcal{K}^{*}$.
The type (iii) minimally sufficient set contains 4 scenarios, but note that although it is smaller than the type (ii) set, it is not contained in it, as it contains one non-supported efficient point which is not in the type (ii) set.
Finally, we can see that the type (iv) minimally sufficient set contains only 2 scenarios, and is subset of the type (iii) set excluding the two scenarios which are not supported efficient with respect to $\mathcal{K}$.
This figure illustrates the hierarchy of sufficient set types stated in Theoerem~\ref{th:implications}.

\begin{figure}[h]
  \centering
  \includegraphics[width=0.8\textwidth]{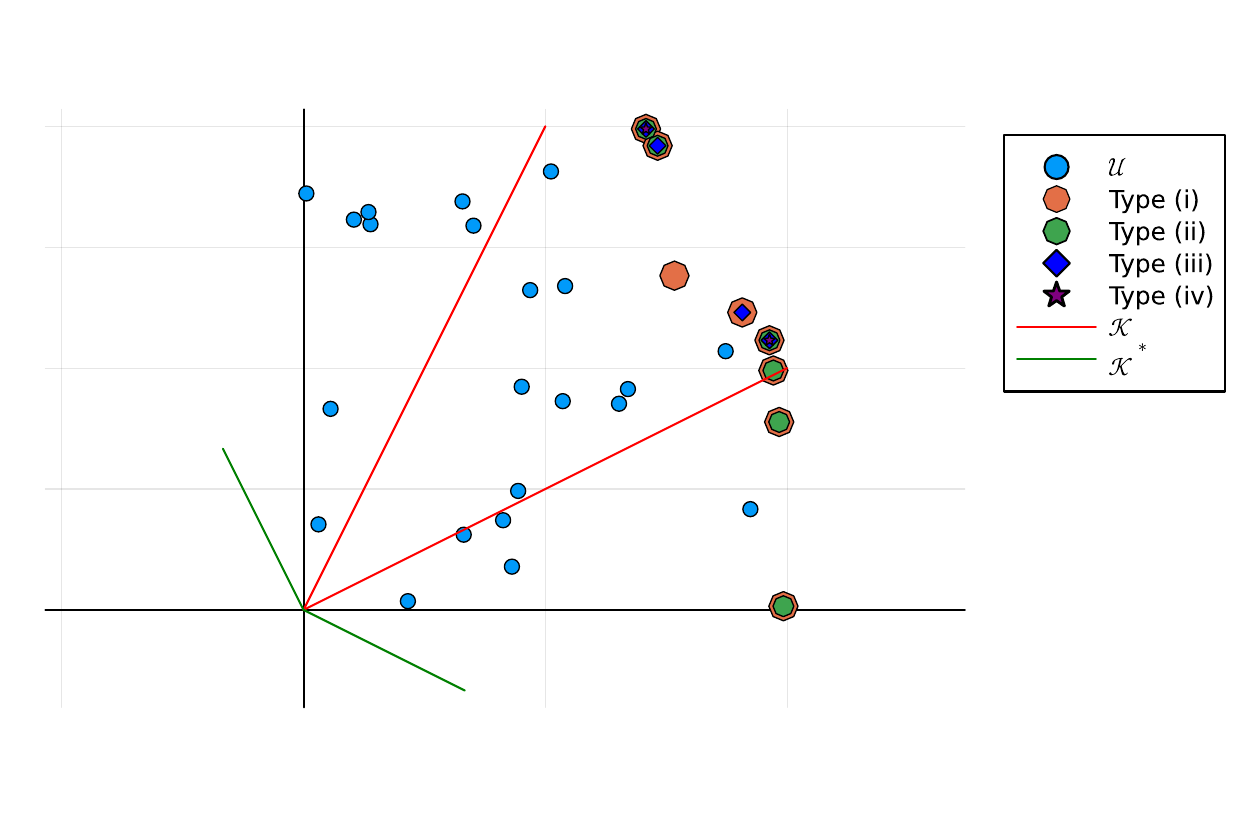}
  \caption{Illustration of different minimal sufficient sets for a two-dimension uncertainty set.}
  \label{fig:ill-example}
\end{figure}

For many problems, the feasible region $\X$ can be expressed as follows:
\begin{equation*}
  \label{eq:lp-x}
\X = \{\pmb{x} \in \rn_+: \exists \pmb{y}\in\mathbb{R}^p \ \text{s.t. } A\pmb{x} + B\pmb{y} \geq \pmb{b}\}.
\end{equation*}
where $A\in\mathbb{R}^{m\times n}$ and $B\in\mathbb{R}^{m \times p}$.
That is, $\X$ is the feasible region of a linear program.
The conic hull of this set is then given by:
\begin{equation*}
  \conic{\X} = \{\pmb{x} \in \rn_+ : \exists \pmb{y}\in\mathbb{R}^p, \lambda \ge 0 \ \text{s.t. } A\pmb{x} + B \pmb{y} = \lambda \pmb{b} \}
\end{equation*}
The following result gives the dual of this cone.
\begin{proposition}\label{prop:coneLP}
  Suppose we have the convex cone $\mathcal{K} = \{\pmb{x} \in \rn_+: \exists \pmb{y}\in\mathbb{R}^p, \lambda \ge 0 \ \text{s.t. } A\pmb{x} + B \pmb{y} = \lambda \pmb{b} \}$.
  Then, its dual cone $\mathcal{K}^{*}$ is defined as follows:
  \begin{equation*}
   \mathcal{K}^{*} = \{\pmb{\pi}\in \rn: \exists \pmb{z} \ge \pmb{0}\ \text{s.t. } \pmb{\pi} \geq  A^\top \pmb{z},\ B^\top \pmb{z} = \pmb{0},\  \pmb{b}^\top \pmb{z} \geq 0 \}
  \end{equation*}
\end{proposition}

\begin{proof}
\begin{align*}
\mathcal{K}^{*} &= \left\{\pmb{\pi} \in \rn: \pmb{x}^\top \pmb{\pi} \geq 0\ \forall \pmb{x} \in \mathcal{K}\right\} \\
&=\left\{\pmb{\pi} \in \rn: \min_{\pmb{x},\pmb{y},\lambda}\{\pmb{x}^\top \pmb{\pi} : A\pmb{x} + B\pmb{y} - \lambda \pmb{b} \leq \pmb{0},\ \lambda\geq 0,\ \pmb{x} \geq \pmb{0}\} \geq 0 \right\}\\
&=\left\{\pmb{\pi} \in \rn : \max_{\pmb{z}}\{0 : A^\top\pmb{z} \leq \pmb{\pi},\ B^\top \pmb{z} = \pmb{0},\ \pmb{b}^\top \pmb{z} \geq 0, \pmb{z}\geq \pmb{0}\} \geq 0 \right\}\\
&= \left\{\pmb{\pi} \in \rn: \pmb{\pi} \geq A^\top\pmb{z},\ B^\top \pmb{z} = \pmb{0},\ \pmb{b}^\top \pmb{z} \geq 0,\ \pmb{z}\geq \pmb{0} \right\}
\end{align*}
\end{proof}

As a consequence of this result, for $\X$ as defined in \eqref{eq:lp-x}, we have the following equivalence:
\begin{align*}
  \pmb{c}_{1}^\top \pmb{x} \geq \pmb{c}_{2}^\top \pmb{x} \qquad \forall \pmb{x}\in \X &\Longleftrightarrow \pmb{c}_{1} - \pmb{c}_{2} \succeq_{\mathcal{K}^{*}} \pmb{0}\\
  &\Longleftrightarrow \pmb{c}_{1} - \pmb{c}_{2} \geq  A^\top \pmb{z},\ \text{where } B^\top \pmb{z} = \pmb{0},\ \pmb{b}^\top \pmb{z} \geq 0,\ \pmb{z}\geq \pmb{0}.
\end{align*}

\subsection{Finding Reduced Scenario Sets}
\label{sec:exact-reduction}
In the following, we derive optimization models to find smallest possible sets of type (i) to (iv). We begin with type (i) and show that it is possible to find such a set using an integer linear program that is similar to what has been proposed in \cite{goerigk2023optimal}.

\begin{theorem}\label{th:typei-subset}
Let $(\pmb{\lambda},\pmb{\mu})$ be an optimal solution of the following integer linear program:
\begin{align}
\min\ & \sum_{k\in[N]} \lambda_k \label{tired-1}\\
\text{s.t. } & c^j_i \le \sum_{j\in[N]} \mu_{jk} c^k_i & \forall j\in[N], i\in[n] \label{tired-2}\\
& \mu_{jk} \le \lambda_k & \forall j,k\in[N] \label{tired-3}\\
& \sum_{k\in[N]} \mu_{jk} = 1 & \forall j\in[N] \label{tired-4}\\
& \mu_{jk} \in \{0,1\} & \forall j,k\in[N] \label{tired-5}\\
& \lambda_k \in \{0,1\} & \forall k\in[N]\label{tired-6}
\end{align}
Then, $\cU' = \{ k\in[N] : \lambda_k = 1\}\subseteq \cU$ is a smallest sufficient subset of type (i).
\end{theorem}
\begin{proof}
A set $\cU'\subseteq \cU$ is a sufficient subset of type (i) if and only if for all $\pmb{c}\in\cU$ there is a $\pmb{c}'\in\cU'$ such that $\pmb{c}'\ge\pmb{c}$. Variables $\pmb{\lambda}$ model if a scenario is contained in $\cU'$ or not. Variables $\pmb{\mu}$ model if scenario $\pmb{c}^j\in\cU$ is dominated by scenario $\pmb{c}^k\in\cU'$. By constraints~\eqref{tired-4}, we need to assign exactly one such scenario $\pmb{c}^k$ per scenario $\pmb{c}^j$. Constraints~\eqref{tired-3} ensure that a scenario can only be used for this purpose if it is contained in $\cU'$. Finally, constraints~\eqref{tired-2} model that scenario $\pmb{c}^j$ is indeed dominated by scenario $\pmb{c}^k$ in each component $i\in[n]$, if $\mu_{jk}=1$.
\end{proof}

While Theorem~\ref{th:typei-subset} provides a way to find an sufficient subset that is the smallest amongst all sets of type (i) through an integer program, there is a natural polynomial-time algorithm for this problem. A scenario $\pmb{c}\in\cU$ Pareto-dominates a scenario $\pmb{c}'\in\cU$ if $c_i \ge c'_i$ for all $i\in[n]$ and $c_i > c'_i$ for at least one $i\in[n]$. A scenario that is not Pareto-dominated by another scenario is called  Pareto efficient, see, e.g., \cite{ehrgott2005multicriteria}. We can determine if a given scenario is Pareto efficient by comparing it against the remaining $N-1$ scenarios in $\cU$. By performing such pairwise comparisons, we identify the set of all Pareto efficient solutions in polynomial time. By definition, this is the smallest subset of type (i), which means that the problem is polynomial-time solvable. However, the integer program provided in Theorem~\ref{th:typei-subset} provides a starting point to model the further types of dominance.

\begin{theorem}\label{th:typeii-subset}
Let $(\pmb{\lambda},\pmb{\mu})$ be an optimal solution of the following mixed-integer linear program:
\begin{align}
\min\ & \sum_{k\in[N]} \lambda_k \label{tiired-1}\\
\text{s.t. } & \eqref{tired-2}, \eqref{tired-3}, \eqref{tired-4}, \eqref{tired-6} \nonumber\\
& \mu_{jk} \in [0,1] & \forall j,k\in[N] \label{tiired-5}
\end{align}
Then, $\cU' = \{ k\in[N] : \lambda_k = 1\}\subseteq \cU$ is a smallest sufficient subset of type (ii).
\end{theorem}
\begin{proof}
As before, we use variables $\lambda_k\in\{0,1\}$ to indicate if scenario $\pmb{c}^k$ is contained in $\cU'$. The set $\cU'\subseteq \cU$ is a sufficient subset of type (ii) if and only if for all $\pmb{c}\in\cU$ and $\pmb{\alpha}\in\mathbb{R}^n_+$ there is $\pmb{c}'\in\cU'$ such that $\sum_{i\in[n]} \alpha_i c_i \le \sum_{i\in[n]} \alpha_i c'_i$.
Let us fix any $\pmb{c}\in\cU$. Then this condition is equivalent to requiring that
\[ \max_{\pmb{\alpha}\in\mathbb{R}^n_+} \min_{\pmb{c}'\in\cU'} \sum_{i\in[n]} \alpha_i (c_i - c'_i) \le 0 \]
Note that
\begin{align*}
& \max_{\pmb{\alpha}\in\mathbb{R}^n_+} \min_{\pmb{c}'\in\cU'} \sum_{i\in[n]} \alpha_i (c_i - c'_i) \\
=& \max_{\pmb{\alpha}\in\mathbb{R}^n_+} \min \left\{ \sum_{i\in[n]} \alpha_i (c_i - \sum_{k\in[N]} \mu_k c^k_i) : \sum_{k\in[N]} \mu_k = 1,\ \mu_k \le \lambda_k,\ \mu_k \ge 0\ \forall k\in[N] \right\} \\
=&\min \left\{ \max_{\pmb{\alpha}\in\mathbb{R}^n_+} \sum_{i\in[n]} \alpha_i (c_i - \sum_{k\in[N]} \mu_k c^k_i) :  \sum_{k\in[N]} \mu_k = 1,\ \mu_k \le \lambda_k,\ \mu_k \ge 0\ \forall k\in[N] \right\}
\end{align*}
where the last equality holds because of von Neumann's minimax theorem. Finally, we have
\[ \max_{\pmb{\alpha}\in\mathbb{R}^n_+} \sum_{i\in[n]} \alpha_i (c_i - \sum_{k\in[N]} \mu_k c^k_i) \le 0 \ \Leftrightarrow\ c_i - \sum_{k\in[N]} \mu_k c^k_i \le 0\ \forall i\in[n]. \]
By using a distinct set of variables $\pmb{\mu}$ for each $j\in[N]$, we obtain the claimed mixed-integer linear program.
\end{proof}

As in the case of type (i) sufficient subsets, there is an alternative to solving problem~(\ref{tiired-1}-\ref{tiired-5}) that is based on multi-criteria efficiency. As the variables $\mu_{jk}$ are continuous, it means that a scenario set $\cU'$ of type (ii) has the property that for each of the original scenarios $\pmb{c}\in\cU$, we need to be able to find a convex combination of scenarios in $\cU'$ that dominates $\pmb{c}$. In other words, a smallest set of this type contains exactly the supported efficient solutions (see, e.g., \cite{ehrgott2005multicriteria}), which we can again identify in polynomial time (for example, by finding the set of Pareto solutions, and removing those that can be dominated by a convex combination within this set).

We now turn to those types of sufficient subsets that make use of the set $\X$ of feasible solutions.

\begin{theorem}
\label{th:typeiii-subset-generic}
Let $(\pmb{\lambda},\pmb{\mu})$ be an optimal solution of the following integer linear program with potentially infinitely many constraints:
\begin{align}
\min\ &\sum_{k\in[N]} \lambda_k \\
\text{s.t. } & \sum_{i\in[n]} (\sum_{k\in[N]} \mu_{jk} c^k_i - c^j_i) x_i \ge 0 & \forall j\in [N], \pmb{x}\in \X \label{tiiired-2}\\
& \eqref{tired-3}, \eqref{tired-4}, \eqref{tired-5}, \eqref{tired-6} \nonumber
\end{align}
Then, $\cU'=\{ k\in[N] : \lambda_k =1 \}\subseteq \cU$ is a smallest sufficient subset of type (iii).
\end{theorem}
\begin{proof}
As before, $\pmb{\lambda}$ is used as a binary variable to model set $\cU'$. For each $\pmb{c}^j\in\cU$, the binary variable $\mu_{jk}$ models a choice of some $\pmb{c}^k\in\cU'$. Constraints~\eqref{tiiired-2} then imply that for all $\pmb{c}^j$ and for all $\pmb{x}\in\X$, we have $\pmb{x}^\top \pmb{c}^k \ge \pmb{x}^\top \pmb{c}^j$. As variables $\pmb{\mu}$ do not depend on $\pmb{x}$, the order of quantifiers is as required by sufficiency type (iii).
\end{proof}

The existence of potentially infinitely many constraints makes the proposed optimization model intractable in practice. If the set of feasible solutions is sufficiently ''nice'', this issue can be avoided by using duality. We therefore formalize this property in the following.

\begin{definition}
We say $\X$ is (LP), if there is a known polyhedron in outer description $\X'=\{ (\pmb{x},\pmb{y})\in \mathbb{R}^n_+ \times \mathbb{R}^p : A\pmb{x}+B\pmb{y} \ge \pmb{b}\}$, where $\min_{\pmb{x}\in\X} \sum_{i\in[n]} c_i x_i = \min_{(\pmb{x},\pmb{y})\in\X'} \sum_{i\in[n]} c_i x_i$ for all $\pmb{c}\in\mathbb{R}^n$.
\end{definition}

In particular, note that $\X$ is (LP) is $\X$ if a polyhedron itself. Furthermore, many combinatorial problems allow for such a reformulation as a linear program, for example, this is the case for shortest path, spanning tree, and assignment problems. Recall that we already used this property in Proposition~\ref{prop:coneLP} to derive the dual cone.

\begin{cor}\label{cor:rediii}
Let $\X$ be (LP) and let $(\pmb{\lambda},\pmb{\mu},\pmb{\pi})$ be an optimal solution of the following integer linear program:
\begin{align}
\min\ &\sum_{k\in[N]} \lambda_k \\
\text{s.t. } & \pmb{b}^\top \pmb{\pi}^j \ge \pmb{0} & \forall j\in [N] \\
& A^\top\pmb{\pi}^j \le  \sum_{k\in[N]} \mu_{jk} \pmb{c}^k - \pmb{c}^j  & \forall j\in[N] \\
& B^\top\pmb{\pi}^j =  \pmb{0}  & \forall j\in[N] \\
& \eqref{tired-3}, \eqref{tired-4}, \eqref{tired-5}, \eqref{tired-6} \nonumber \\
& \pi^j_\ell \ge 0 & \forall j\in[N], \ell\in [q]
\end{align}
where $q$ is the number of rows in matrices $A$ and $B$. Then, $\cU'=\{ k\in[N] : \lambda_k =1 \}\subseteq \cU$ is a smallest sufficient subset of type (iii).
\end{cor}
\begin{proof}
We reformulate Constraints~\eqref{tiiired-2} under the assumption that $\X$ is (LP) using Proposition~\ref{prop:coneLP}. For any $j\in[N]$, we have
\begin{align*}
 &\sum_{i\in[n]} (\sum_{k\in[N]} \mu_{jk} c^k_i - c^j_i) x_i \ge 0 & \forall \pmb{x}\in\X \\
 \Leftrightarrow\ & \min_{\pmb{x}\in\X} \sum_{i\in[n]} (\sum_{k\in[N]} \mu_{jk} c^k_i - c^j_i) x_i \ge 0 \\
\Leftrightarrow\ & \min_{(\pmb{x},\pmb{y}) \in\X'} \sum_{i\in[n]} (\sum_{k\in[N]} \mu_{jk} c^k_i - c^j_i) x_i \ge 0 \\
\Leftrightarrow\ & \max \left\{ \pmb{b}^\top \pmb{\pi}^j : A^\top \pmb{\pi}^j \ge \sum_{k\in[N]} \mu_{jk} \pmb{c}^k - \pmb{c}^j,\ B^\top \pi^j = \pmb{0},\ \pmb{\pi}^j\in\mathbb{R}^{q}_+ \right\}
\end{align*}
The rest of the proof is the same as in Theorem~\ref{th:typeiii-subset-generic}.
\end{proof}

Note that even if $\X$ is not (LP), we may use any linear programming relaxation of $\X$ in the model of Corollary~\ref{cor:rediii} to find a sufficient subset, albeit not necessarily the smallest such set. This is because if we have the property that for all $\pmb{c}\in\cU$ there is some $\pmb{c}'\in\cU'$ such that for all $\pmb{x}'\in\X'$ we have $\pmb{x}^\top\pmb{c}' \ge \pmb{x}^\top\pmb{c}$ for some superset $\X'$ of $\X$, then this is in particular true for $\X$ and therefore property (iii) is fulfilled.

We now consider our strongest condition for finding sufficient subsets.

\begin{theorem}\label{th:typeiv-subset}
Let $(\pmb{\lambda},\pmb{\mu})$ be an optimal solution of the following integer linear program with potentially infinitely many variables and constraints:
\begin{align}
\min\ &\sum_{k\in[N]} \lambda_k \label{tivred-1}\\
\text{s.t. } & \sum_{i\in[n]} (\sum_{k\in[N]} \mu^{\pmb{x}}_{jk} c^k_i - c^j_i) x_i \ge 0 & \forall j\in [N], \pmb{x}\in \X \label{tivred-21}\\
& \mu^{\pmb{x}}_{jk} \le \lambda_k & \forall j,k\in[N], \pmb{x}\in \X \\
& \sum_{k\in[N]} \mu^{\pmb{x}}_{jk} = 1 & \forall j\in[N], \pmb{x}\in\X \\
& \mu^{\pmb{x}}_{jk} \in \{0,1\} & \forall j,k\in[N], \pmb{x}\in\X \\
& \lambda_k \in \{0,1\} & \forall k\in[N] \label{tivred-2}
\end{align}
Then, $\cU'=\{ k\in[N] : \lambda_k =1 \}\subseteq \cU$ is a smallest sufficient subset of type (iv).
\end{theorem}
\begin{proof}
This model is derived in the same way as type (iii) in Theorem~\ref{th:typeiii-subset-generic}, except that an index is added to variables $\pmb{\mu}$ so that they may depend on $\pmb{x}\in\X$.
\end{proof}

Unless $\X$ is a finite set of small cardinality, model~(\ref{tivred-1}-\ref{tivred-2}) is unlikely to be solvable in practice. If we assume that $\X$ is (LP), then we can use the following approach to find a sufficient subset, even though it may not be minimal.

\begin{cor}\label{th:typeiv-subset-approximate}
Let $\X$ be (LP) and let $(\pmb{\lambda},\pmb{\mu},\pmb{\pi})$ be a feasible solution of the following integer linear program:
\begin{align*}
\min\ &\sum_{k\in[N]} \lambda_k \\
\text{s.t. } & \pmb{b}^\top \pmb{\pi}^j \ge \pmb{0} & \forall j\in [N] \\
& A^\top\pmb{\pi}^j \le  \sum_{k\in[N]} \mu_{jk} \pmb{c}^k - \pmb{c}^j  & \forall j\in[N] \\
& B^\top\pmb{\pi}^j =  \pmb{0}  & \forall j\in[N] \\
& \eqref{tired-3}, \eqref{tired-4}, \eqref{tired-6} \nonumber \\
& \mu_{jk} \in [0,1] & \forall j,k\in[N] \\
& \pi^j_\ell \ge 0 & \forall j\in[N], \ell\in [q]
\end{align*}
Then, $\cU'=\{ k\in[N] : \lambda_k =1 \}\subseteq \cU$ is a sufficient subset of type (iv).
\end{cor}
\begin{proof}
For fixed $\pmb{\lambda}$ and $j\in[N]$, let
\[ M = \{\pmb{\mu}\in[0,1]^{N} : \mu_k \le \lambda_k \ \forall k\in[N],\ \sum_{k\in[N]} \mu_k = 1\}.\]
We reformulate Constraints~\eqref{tivred-21} using the (LP) assumption for fixed $j\in [N]$:
\begin{align*}
&  \forall \pmb{x}\in \X \ \exists \pmb{\mu}\in M:\  \sum_{i\in[n]} (\sum_{k\in[N]} \mu_k c^k_i - c^j_i) x_i \ge 0 \\
\Leftrightarrow\ & \min_{\pmb{x}\in\X} \max_{\pmb{\mu}\in M} \sum_{i\in[n]} (\sum_{k\in[N]} \mu_k c^k_i - c^j_i) x_i \ge 0 \\
\Leftarrow\ & \max_{\pmb{\mu}\in M} \min_{\pmb{x}\in\X} \sum_{i\in[n]} (\sum_{k\in[N]} \mu_k c^k_i - c^j_i) x_i \ge 0\\
\Leftrightarrow\  & \max_{\pmb{\mu}\in M} \min_{(\pmb{x},\pmb{y})\in\X'} \sum_{i\in[n]} (\sum_{k\in[N]} \mu_k c^k_i - c^j_i) x_i \ge 0 \\
\Leftrightarrow\ & \max \left\{ \pmb{b}^\top \pmb{\pi}^j : A^\top \pmb{\pi}^j \ge \sum_{k\in[N]} \mu_{jk} \pmb{c}^k - \pmb{c}^j,\ B^\top \pi^j = \pmb{0},\ \pmb{\pi}^j\in\mathbb{R}^{q}_+,\ \pmb{\mu}\in M \right\}
\end{align*}
Note that the implications mean that the final line is a sufficient but not necessary criterion to fulfill type (iv). In general, $\min_{\pmb{x}\in\mathcal{X}} \max_{\pmb{y}\in\mathcal{Y}} f(\pmb{x},\pmb{y}) \ge \max_{\pmb{y}\in\mathcal{Y}} \min_{\pmb{x}\in\mathcal{X}} f(\pmb{x},\pmb{y})$, but equality does not necessarily hold if $\X$ is discrete.
\end{proof}

\subsection{Example Problems}
\label{sec:example-problems}

\subsubsection{The Selection Problem}
\label{sec:selection-problem}

Take as an example the \emph{selection problem} where the feasible region is $\X_{n,p} = \{ \pmb{x}\in\{0,1\}^n : \sum_{i\in[n]} x_i = p\}$ for an integer $p$.
For small values of $n$ we can enumerate $\X_{n,p}$ as above; in particular there are ${n \choose p}$ solutions.
For example, for $n=3$ and $p=2$, we have
\[ \X_{3,2} = \{ (1,1,0), (0,1,1), (1,0,1)\} \]

Recall that type (iii) dominance can be expressed in terms of the dual of the conic hull of the feasible set.
In this case, the conic hull of $\X_{n,p}$ can then simply be expressed as a finitely generated cone:

\begin{equation}
  \label{eq:conic-hull-selection}
  \conic{\X_{n,p}} = \{\sum\lambda_{i} \pmb{x}_{i} : \pmb{x}_{i}\in\X,\ \lambda \geq 0 \}.
\end{equation}

Note that these conic hulls are different for different values of $p$. For example, $(1,2,3) = 1\cdot(1,0,1) + 2\cdot(0,1,1) \in \conic{\X_{3,2}}$, while $(1,2,4)\notin \conic{\X_{3,2}}$. The following result shows that this cone becomes more restricted as $p$ increases.

\begin{proposition}
  \label{prop:conic-selection-monotonic}
  \begin{equation*}
    \conic{\X_{n,p+1}} \subseteq \conic{\X_{n,p}}
  \end{equation*}
\end{proposition}

\begin{proof}
  It suffices to note that we have $\pmb{x}\in\conic{\X_{n,p}}$ for any $\pmb{x}\in\X_{n,p+1}$.
  Fix $\pmb{x}\in\X_{n,p+1}$ and let $I = \{ 1\leq i \leq n : x_{i} = 1\}$.
  Then, let $\mathcal{J} = \{J \subset I : \sum_{j\in J} x_{j} = p\}$.
  For $J \in \mathcal{J}$ let $\pmb{x}^{J}\in\rn$ be such that $x^{J}_{i} =
  \begin{cases}
    1 & \text{ if } i\in J,\\
    0 & \text{ otherwise.}
  \end{cases}$
  Now, we have
  \begin{align*}
    \pmb{x} &= \frac{1}{p}\left( \sum_{J\in\mathcal{J}} \pmb{x}^{J} \right)
  \end{align*}
  and so $\pmb{x}\in\conic{\X_{n,p}}$ as required.
\end{proof}

Note that the converse of this result cannot be true as any element in $\conic{\X_{n,p+1}}$ must have at least $p+1$ non-zero entries, whereas any element of $\X_{n,p}$ has only $p$ non-zero entries.
The significance of this result is that if a set $\cU' \subset \cU$ is type (iii) sufficient with respect to $\X_{n,p}$, then it is also type (iii) sufficient with respect to $\X_{n,p+1}$.
This follows from the conic definition of sets of type (iii) in Proposition~\ref{prop:conic-type-iii}.

\subsubsection{Relaxation for Other Combinatorial Problems}

Consider another robust combinatorial optimization problem with feasible decision set $\X\subseteq \{0,1\}^{n}$.
Like with the selection problem, the solutions $\pmb{x}\in\mathcal{X}$ of such problems can be viewed as selecting from a subset from larger collection of items.
Other examples include the shortest path problem and spanning tree problems (where one selects a collection of edges in a graph), and the vertex cover problem (where one selects a subset of nodes in a graph).
Such problems will require the selection of a minimal number of elements:
\begin{equation*}
  \alpha_{\X} = \min_{\pmb{x}} \left\{\sum_{i\in[n]} x_{i}: \pmb{x}\in\mathcal{X} \right\}
\end{equation*}
By Proposition~\ref{prop:conic-selection-monotonic}, this means that $\mathcal{X} \subseteq \X_{n,\alpha_{\X}}$ and so for such a problem, we could use $\X_{n,\alpha_{\X}}$ for the purposes of scenario reduction in the case that $\X$ cannot be used directly.

\subsubsection{Shortest Path Problem}
\label{sec:short-path-probl}

Given a network $G=(N,A)$, and a pair of nodes $s, e \in N$ in this, the feasible set of decisions for a shortest path problem is defined as follows:
\begin{equation*}
  \X = \left\{ (x_{a})_{a\in\mathcal{A}} \in \{0,1\}^{n} : \sum_{a\in \delta_{+}(i)}x_{a} - \sum_{a\in \delta_{-}(i)}x_{a} =
  \begin{cases}
   -1 & \text{ if } i = s \\
    1 & \text{ if } i = e \\
    0 & \text{ otherwise}
  \end{cases}
  \text{ for each } i \in \mathcal{N} \right\}
\end{equation*}
where $d = |\mathcal{A}|$ and $\delta_{+}(i)$ and $\delta_{-}(i)$ represent the set of incoming and out-going arcs for a given node $i\in\mathcal{N}$.

For the purposes of this paper we will work with the shortest path problem for a particular type of network called a \emph{layered network}.
This network is defined by two parameters $L\in\mathbb{N}$, the number of layers, and $W\in\mathbb{N}$ the width of each layer.
The nodes of the network consist of a start and end node, and then $L$ layers of nodes, where each layer has $W$ nodes.
There are arcs connecting the start node to each node in the first layer, arcs connecting each node of a layer to all the nodes in the next layer, and finally arcs connecting each node in the final layer to the end node.
Such a graph will thus consist of $LW + 2$ nodes and $(L-1)W^{2} + 2W$ arcs.

For a layered graph with $L$ layers of width $W$ we have $\alpha_{\X} = L + 1$ so in this case we have $\X_{n, L+1} \subseteq \X$ where $n=(L-1)W^{2} + 2W$.

\section{Scenario Approximation}
\label{sec:approximation}

\subsection{Definitions and Properties}

We now study the scenario approximation problem, which relaxes the requirement of scenario reduction by allowing the robust objective value to change. We define the scenario approximation problem as follows.

\begin{definition}
A discrete set $\cU'\subseteq\mathbb{R}^n$ is said to be $\beta$-approximate with respect to $\cU$, if for all minimizers $\pmb{x}\in\X'$ of $\max_{\pmb{c}\in\cU'} f(\pmb{x},\pmb{c})$, we have
\[ \min_{\pmb{x}\in\X} \max_{\pmb{c}\in\cU} f(\pmb{x},\pmb{c}) \le \beta \max_{\pmb{c}\in\cU} f(\pmb{x}',\pmb{c}) \]
\end{definition}

By constructing an $\beta$-approximate set, we can solve a robust problem with respect to $\cU'$ to obtain a solution that is an $\beta$-approximation for the original robust problem. In analogy to Definition~\ref{def:types}, we now define four criteria that allow us to identify such sets.

\begin{definition}\label{def:approxcriteria}
We define the following properties for a discrete set $\cU'\subseteq \textup{conv}(\cU)$ and given $\beta\ge 1$:
\begin{itemize}
\item $\beta$-approximate type (i'):
\[ \forall \pmb{c}\in \cU \ \exists \pmb{c}'\in\textup{conv}(\cU') : \beta \pmb{c}' \ge \pmb{c} \]

\item $\beta$-approximate type (ii'):
\[ \forall \pmb{c}\in\cU, \pmb{\alpha}\in\mathbb{R}^n_+\ \exists \pmb{c}'\in\textup{conv}(\cU') : \beta  \pmb{\alpha}^\top \pmb{c}' \ge \pmb{\alpha}^\top \pmb{c} \]

\item $\beta$-approximate type (iii'):
\[ \forall \pmb{c}\in\cU\ \exists \pmb{c}'\in\textup{conv}(\cU') \ \forall \pmb{x}\in\X : \beta  \pmb{x}^\top \pmb{c}' \ge \pmb{x}^\top\pmb{c} \]

\item $\beta$-approximate type (iv'):
\[ \forall \pmb{c}\in\cU, \pmb{x}\in\X\ \exists \pmb{c}'\in\textup{conv}(\cU') : \beta \pmb{x}^\top \pmb{c}' \ge \pmb{x}^\top\pmb{c} \]

\end{itemize}
\end{definition}

Note that we allow scenarios $\cU'$ to come form the convex hull of $\cU$, which further relaxes the conditions of Definition~\ref{def:types} and allows us to construct few scenarios that represent $\cU$ well. We show that the difference between properties (i') and (ii') disappears in this case. In the following, we use the notation $\Delta_n=\{\pmb{\lambda}\in\mathbb{R}^n_+ : \sum_{i\in[n]} \lambda_i = 1\}$ for the $n$-dimensional simplex.

\begin{theorem}\label{th:equiv}
It holds that (i') $\Leftrightarrow$ (ii').
\end{theorem}
\begin{proof}
Let $\cU'$ be of type (i'), and let any $\pmb{c}\in\cU$ and $\pmb{\alpha}\in\mathbb{R}^n_+$ be given. As there is $\pmb{c}'\in\textup{conv}(\cU')$ with $\beta\pmb{c}'\ge\pmb{c}$, it follows that $\beta\pmb{\alpha}^\top\pmb{c}' \ge \pmb{\alpha}^\top\pmb{c}$ holds as well.

Let $\cU'=\{\hat{\pmb{c}}^1,\ldots,\hat{\pmb{c}}^K\}$ be of type (ii'), and let any $\pmb{c}\in\cU$ be given. We thus have that
\begin{align*}
& \forall \pmb{\alpha}\in\mathbb{R}^n_+ \exists \pmb{c}' \in\textup{conv}(\cU') : \beta \pmb{\alpha}^\top\pmb{c}' \ge \pmb{\alpha}^\top\pmb{c} \\
\Leftrightarrow\ & \forall \pmb{\alpha}\in\mathbb{R}^n_+ : \min_{\pmb{\lambda}\in\Delta_K} \left( \pmb{\alpha}^\top\pmb{c} - \beta \pmb{\alpha}^\top(\sum_{k\in[K]} \lambda_k \hat{\pmb{c}}^k) \right) \le 0 \\
\Leftrightarrow\ & \max_{\pmb{\alpha}\in\mathbb{R}^n_+} \min_{\pmb{\lambda}\in\Delta_K} \left( \pmb{\alpha}^\top\pmb{c} - \beta \pmb{\alpha}^\top(\sum_{k\in[K]} \lambda_k \hat{\pmb{c}}^k) \right) \le 0 \\
\Leftrightarrow\ & \min_{\pmb{\lambda}\in\Delta_K} \max_{\pmb{\alpha}\in\mathbb{R}^n_+}  \left( \pmb{\alpha}^\top\pmb{c} - \beta \pmb{\alpha}^\top(\sum_{k\in[K]} \lambda_k \hat{\pmb{c}}^k) \right) \le 0
\end{align*}
where the last reformulation holds because of von Neumann's minimax theorem.
Given $\pmb{\lambda}\in\Delta_K$, note that
\[ \max_{\pmb{\alpha}\in\mathbb{R}^n_+}  \left( \pmb{\alpha}^\top\pmb{c} - \beta \pmb{\alpha}^\top(\sum_{k\in[K]} \lambda_k \hat{\pmb{c}}^k) \right) \le 0 \ \Leftrightarrow\ c_i - \beta(\sum_{k\in[K]} \lambda_k \hat{c}^k_i ) \le 0\ \forall i\in[n] \]
We therefore obtain that type (ii') implies there is $\lambda\in\Delta_K$ such that $\pmb{c} \le \beta \sum_{k\in[K]} \lambda_k\hat{\pmb{c}}^k$, which is the claimed property of type (i').
\end{proof}

The following result was given in \cite{goerigk2023optimal}.

\begin{theorem}\label{th:approxold}
Let $\cU = \{\pmb{c}^1,\ldots,\pmb{c}^N\}$ and let $\cU' = \{\hat{\pmb{c}}^1,\ldots,\hat{\pmb{c}}^K\}\subseteq \textup{conv}(\cU)$. Let $\beta \ge 1$ be such that
\[ \forall \pmb{c}\in\cU\ \exists \pmb{c}'\in\textup{conv}(\cU') : \beta \pmb{c}' \ge \pmb{c} \]
Then, $\cU'$ is a $\beta$-approximate set with respect to $\cU$.
\end{theorem}

We extend this result towards the criteria from Definition~\ref{def:approxcriteria}.

\begin{theorem}\label{th:approxnew}
Let $\cU'\subseteq \textup{conv}(\cU)$ be a discrete set that fulfills one of the properties (i')-(iv'). Then, $\cU'$ is a $\beta$-approximate set with respect to $\cU$.
\end{theorem}
\begin{proof}
The case of property (i') is already covered in Theorem~\ref{th:approxold}. As (ii') and (i') are equivalent according to Theorem~\ref{th:equiv}, the claim also holds for (ii').

Let some optimal solution $\pmb{x}^*$ of the robust problem with respect to $\cU$ be given, and let $\hat{\pmb{x}}$ be an optimal solution of the robust problem with respect to $\cU'$. Let $\pmb{c}^*\in\cU$ be any worst-case scenario for $\pmb{x}^*$.

Assume that (iii') holds. Then there exists $\pmb{c}'\in\textup{conv}(\cU')$ such that for all $\pmb{x}\in\X$, we have $\beta \pmb{x}^\top \pmb{c}' \ge \pmb{x}^\top\pmb{c}^*$. Hence,
\[ \min_{\pmb{x}\in\X} \max_{\pmb{c}\in\cU} \pmb{c}^\top\pmb{x} = (\pmb{c}^*)^\top\pmb{x}^* \le (\pmb{c}^*)^\top\hat{\pmb{x}}
\le \beta (\pmb{c}')^\top\hat{\pmb{x}} \le \beta \max_{\pmb{c}\in\textup{conv}(\cU')} \pmb{c}^\top\pmb{x}^* = \beta \min_{\pmb{x}\in\X} \max_{\pmb{c}\in\cU'} \pmb{c}\pmb{x}. \]
For case (iv'), we choose $\pmb{c}'\in\textup{conv}(\cU')$ such that $(\pmb{c}^*)^\top\hat{\pmb{x}} \le \beta (\pmb{c}')^\top \hat{\pmb{x}}$ and use the same reasoning as above.

\end{proof}

\subsection{Optimization Models for Scenario Approximation}

We now use Theorem~\ref{th:approxnew} to derive optimization models that construct sets of each type with best possible approximation guarantees. We first consider type (i'). In \cite{goerigk2023optimal}, the following model for such sets has already been proposed:
\begin{align}
\max\ &t \\
\text{s.t. } & t \pmb{c}^j \le \sum_{k\in[K]} \mu_{ik} \hat{\pmb{c}}^k & \forall j\in [N] \\
& \hat{\pmb{c}}^k = \sum_{j\in[N]} \lambda_{kj} \pmb{c} ^j& \forall k\in[K] \\
& \sum_{k\in[K]} \mu_{jk} = 1 & \forall j\in[N] \\
& \sum_{j\in[N]} \lambda_{kj} = 1 & \forall k\in[K] \\
& \mu_{jk} \in [0,1] & \forall j\in[N], k\in[K] \\
& \lambda_{kj} \in [0,1] & \forall k\in[K], j\in[N]
\end{align}
A separate model for type (ii') is not necessary due to Theorem~\ref{th:equiv}.

For type (iii'), we need to consider constraints
\[ t \sum_{\ell\in[n]} c^j_\ell x_\ell \le \sum_{\ell\in[n]} \left( \sum_{k\in[K]} \mu_{ik} \hat{c}_\ell^k\right) x_\ell \ \forall j\in [N], \pmb{x}\in\X. \]
This is equivalent to
\[ \min_{\pmb{x}\in\X}  \sum_{\ell\in[n]} \left(  \sum_{k\in[K]} \mu_{ik} \hat{c}_\ell^k - t c^j_\ell \right) x_\ell \ge 0 \ \forall j\in[N] \]
This is a nominal-type problem with potentially negative cost coefficients (i.e., a problem of the form $\min_{\pmb{x}\in\X} \pmb{c}^\top\pmb{x}$ for a specific $\pmb{c}\in\mathbb{R}^n$). Let us assume that a compact linear programming formulation for the nominal problem is available (recall that this is the case, e.g., for the minimum spanning tree problem with negative costs, as well as the selection problem and the shortest path problem on cycle-free graphs, while the shortest path problem on general graphs becomes NP-hard when negative costs are allowed). This means we can dualize this problem, with the dual being $\max \{ \pmb{b}^\top \pmb{y} : A^\top \pmb{y} \le   \sum_{k\in[K]} \mu_{ik} \hat{c}_\ell^k - t c^i_\ell\}$. Then we can solve (iii') using the following model:
\begin{align}
\max\ &t \label{app3-start}\\
\text{s.t. } & b^\top \pmb{\pi}^j \ge \pmb{0} & \forall j\in [N] \\
& A^\top \pmb{\pi}^j \le  \sum_{k\in[K]} \mu_{jk} \hat{c}_\ell^k - t c^j_\ell  & \forall j\in[N] \\
& B^\top \pmb{\pi}^j = \pmb{0} & \forall j\in[N] \\
& \hat{\pmb{c}}^k = \sum_{j\in[N]} \lambda_{kj} \pmb{c}^j& \forall k\in[K] \\
& \sum_{k\in[K]} \mu_{jk} = 1 & \forall j\in[N] \\
& \sum_{j\in[N]} \lambda_{kj} = 1 & \forall k\in[K] \\
& \pi^j_\ell \ge 0 & \forall j\in[N], \ell\in[q] \\
& \mu_{jk} \in [0,1] & \forall j\in[N], k\in[K] \\
& \lambda_{kj} \in [0,1] & \forall k\in[K], j\in[N] \label{app3-end}
\end{align}

Finally, for (iv') we need to fulfill constraints of the form
\[ \min_{\pmb{x}\in\X} \max_{\pmb{\mu}\in\Lambda}  \sum_{\ell\in[n]} \left(  \sum_{k\in[K]} \mu_{k} \hat{c}_\ell^k -t c^i_\ell \right) x_\ell \ge 0 \ \forall i\in[N] \]
with $\Lambda=\{\pmb{\mu}\in[0,1]^K : \sum_{k\in[K]} \mu_k = 1\}$. The left-hand side is equivalent to the following standard robust optimization problem with $K$ scenarios:
\begin{align*}
\min\ & z \\
\text{s.t. } & z \ge \sum_{\ell\in[n]} \left( \hat{c}_\ell^k - t c^i_\ell\right) x_\ell & \forall k\in[K] \\
& \pmb{x} \in \X
\end{align*}
This motivates the following iterative column-and-constraint generation solution procedure, whose principle is frequently applied to solve robust two-stage problems, see \cite{zeng2013solving,robook}. We first solve
\begin{align*}
\max\ & t \\
\text{s.t. } &  \sum_{\ell\in[n]} \left(   \sum_{k\in[K]} \mu^{\pmb{x}}_{jk} \hat{c}_\ell^k - t c^j_\ell\right) x_\ell \ge 0 & \forall \pmb{x}\in \X', j\in[N] \\
& \hat{\pmb{c}}^k = \sum_{j\in[N]} \lambda_{kj} \pmb{c}^j & \forall k\in[K] \\
& \sum_{k\in[K]} \mu^{\pmb{x}}_{jk} = 1 & \forall \pmb{x}\in \X', j\in[N] \\
& \sum_{j\in[N]} \lambda_{kj} = 1 & \forall k\in[K] \\
& \pmb{\mu},\pmb{\lambda} \ge \pmb{0}
\end{align*}
for some subset of candidate solutions $\X' \subseteq \X$ (it is possible to use $\X'=\emptyset$ in the first iteration). This gives $(t,\hat{\pmb{c}}^1,\ldots,\hat{\pmb{c}}^K)$. We then solve the $N$ robust problems
\[ \min_{\pmb{x}\in \X} \max_{k\in[K]} \sum_{\ell\in[n]} \left(  \hat{c}_\ell^k - t c^j_\ell \right) x_\ell \]
If the objective value is non-negative for all $j\in[N]$, we have found an optimal solution. Otherwise, we add the corresponding $\pmb{x}$-solutions to $\X'$ and repeat the process.

In the following, we also use this approach to evaluate a given scenario set $\cU'$, i.e., we follow the same principle for fixed variables $\hat{\pmb{c}}^k$ to their corresponding approximation guarantee.

\section{Computational Experiments}
\label{sec:experiments}

In the following experiments, we study the performance of both scenario reduction (Experiment~1) and scenario approximation (Experiment~2). We apply these approaches to the two types of problem already discussed: the selection problem, with feasible set of solutions $\X_{n,p}$, and the shortest path problem for layered graphs of length $L$ and width $W$.

\subsection{Experiment 1: Scenario Reduction}

\subsubsection{Setting}

We first test the potential for reduction using types (i)-(iv) reduction for the selection problem and the shortest path problem with a layered graph.

For both problem types we perform the different types of reduction across a range of problem parameters.
In both cases, for each set of problems, we randomly sample 10 scenario sets of size $N=100$ uniformly on $[0,1]^{n}$ where $n$ is the dimension of the problem.
For each sampled set we reduce the set using the different types of reduction and record the proportion of scenarios reduced.
All mixed integer linear programs were implemented using the JuMP package \cite{Lubin2023} in Julia and solved using Gurobi version 9 \cite{gurobi} on a laptop with Intel(R) Core(TM) i5-8265U CPU @ 1.60GHz processor with 8 CPUs.

In the results below, the ``Red-1'' and ``Red-2'' methods find reduced scenario sets by solving the optimization problems in Theorems~\Ref{th:typei-subset} and \ref{th:typeii-subset}.
Minimal subsets of type (iii) are found using the problem specific formulations in Corollary~\ref{cor:rediii} and are labeled ``Red-3''.
Minimal subsets of type (iv) are found via the formulation in Theorem~\ref{th:typeiv-subset} and are labeled ``Red-4''
Finally, the method ``Red-4S'' refers to sets found using the optimization problem in Corollary~\ref{th:typeiv-subset-approximate}.

\subsubsection{Selection Problem Results}

For the selection problem, for each dimension, we perform reduction across a range of $p$.
For better comparability, for each $n$ we use the same set of scenarios for reduction across all $p$. Note also that the computational burden of some of the methods is much greater than others.
Minimal type (iv) reduction in particular can only done for only relatively small dimensions and numbers of scenarios.
We therefore run this experiment over different ranges of dimensions for the different methods.

\begin{figure}[htbp]
  \centering
  \begin{subfigure}{0.49\textwidth}
    \includegraphics[width=\textwidth]{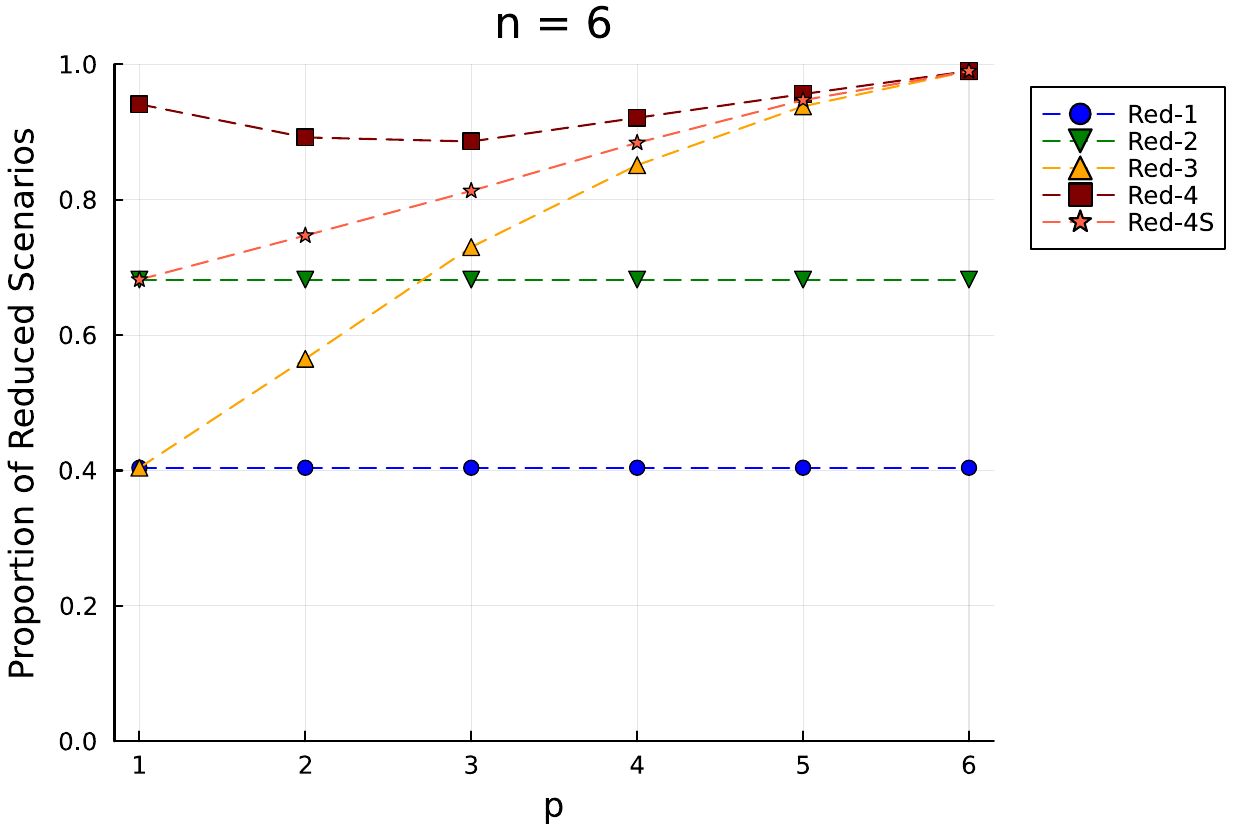}
  \end{subfigure}
  \begin{subfigure}{0.49\textwidth}
    \includegraphics[width=\textwidth]{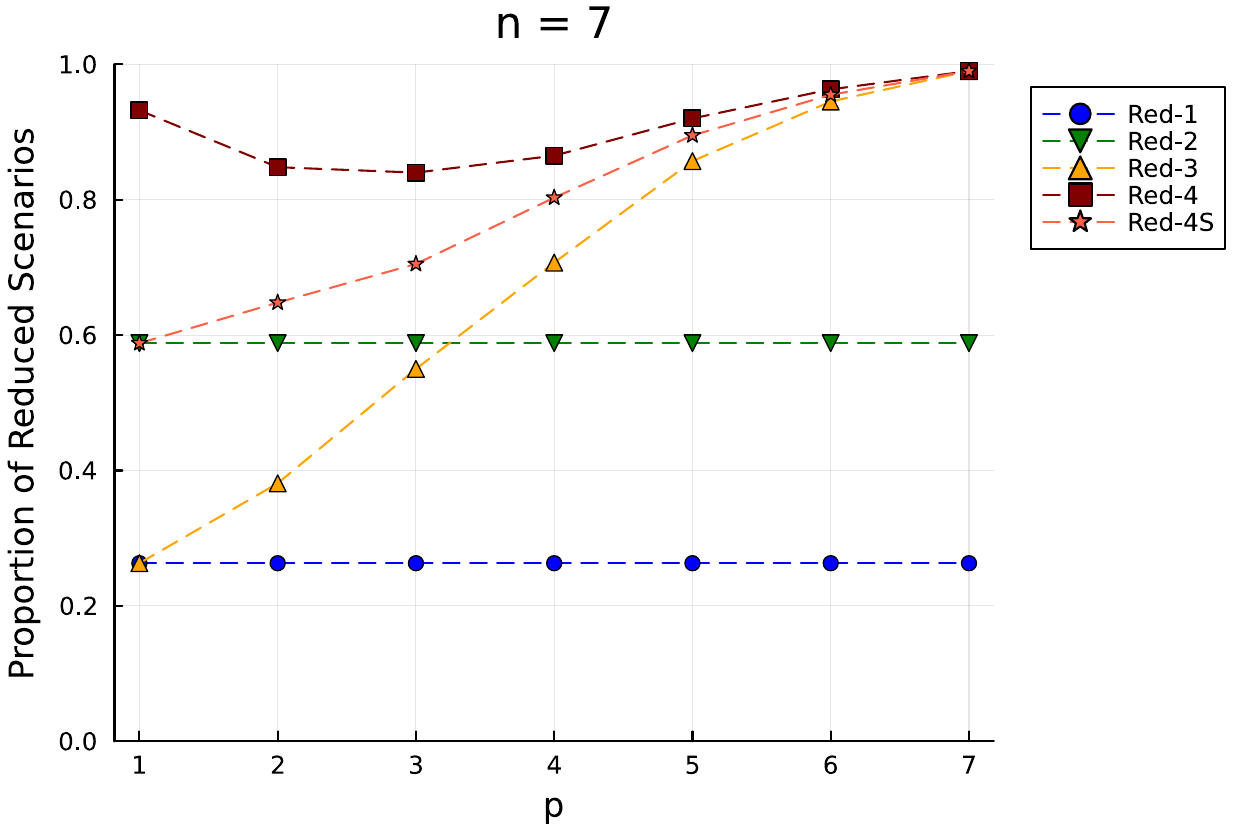}
  \end{subfigure}
  \begin{subfigure}{0.49\textwidth}
    \includegraphics[width=\textwidth]{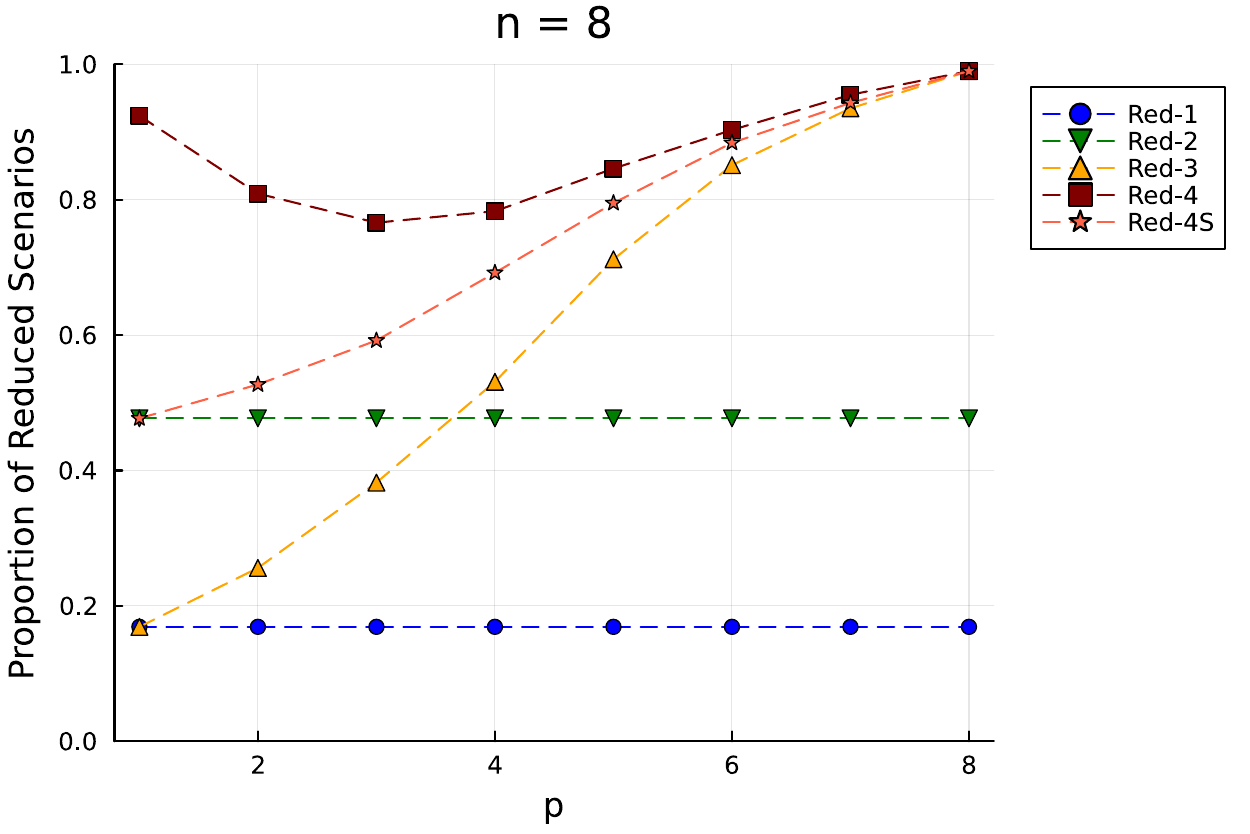}
  \end{subfigure}
  \begin{subfigure}{0.49\textwidth}
    \includegraphics[width=\textwidth]{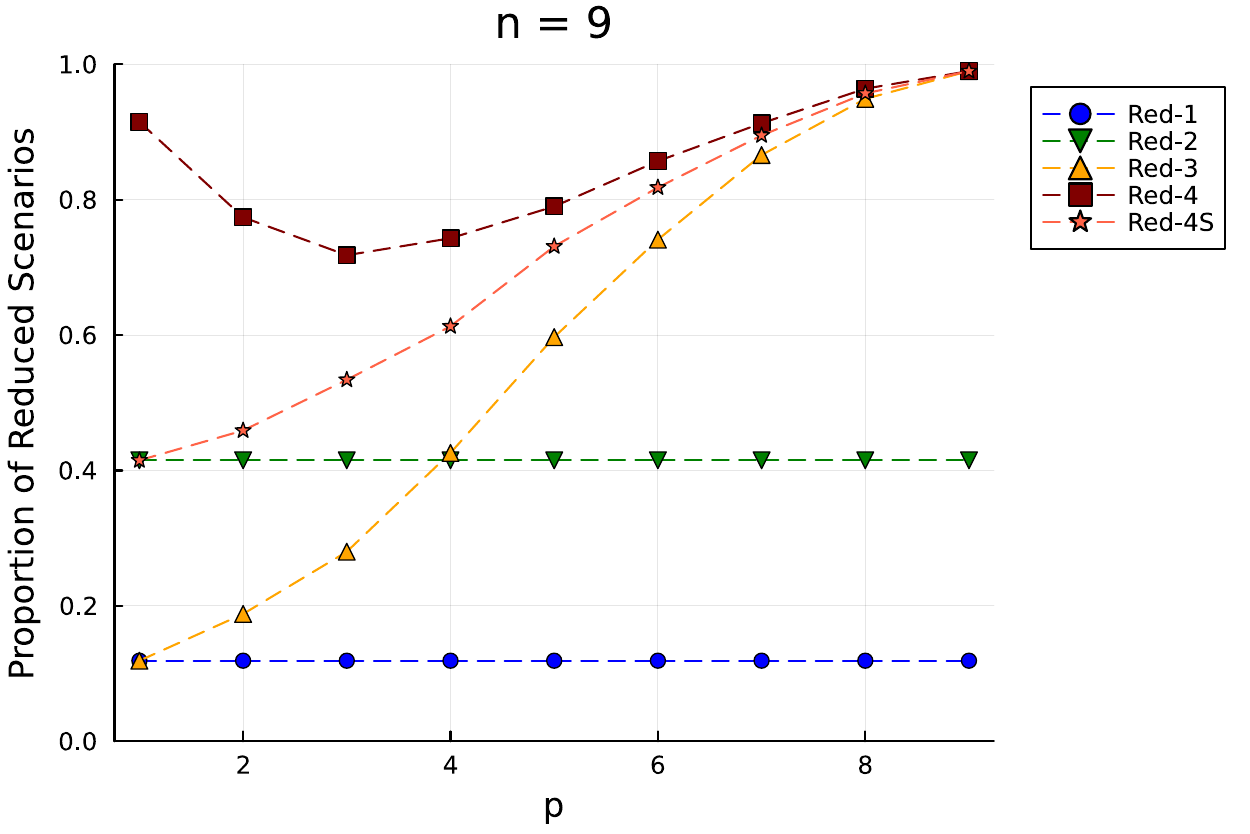}
  \end{subfigure}
  \caption{Reduction using for selection problem for dimensions $n=6,7,8,9$}
  \label{fig:reduction-selection-exact-small-dims}
\end{figure}

Figure~\ref{fig:reduction-selection-exact-small-dims} shows the average reduction for dimensions $n=6,\ldots,9$.
This plot demonstrates again the relationships between the different types of set stated in Theorem~\ref{th:implications}.
Type (i) reduction is always the weakest type of reduction and it becomes less effective as we increase the dimension $n$ of the problem.
Neither type (ii) nor type (iii) dominate the other strictly, but type (iii) becomes stronger as we increase the parameter $p$ and eventually beats type (ii).
The fact that the effectiveness of type (iii) improves can be explained by Proposition~\ref{prop:conic-selection-monotonic} which says that the conic hull of the feasible region becomes smaller of we increase $p$.
Note that both type (i) and type (ii) reduction are flat lines in the plot below since they do not depend on the problem, and are thus unaffected by different values of $p$.

The fact that approximate type (iv) reduction is more effective than exact type (iii) can be explained by the fact that the minimization problem for finding an approximate type (iv) set in Corollary~\ref{th:typeiv-subset-approximate} is a relaxation of the problem used to find the minimal type (iii) set in Corollary~\ref{cor:rediii}.
Type (iv) reduction is always by far the most effective type of reduction but interestingly its effectiveness does not observe the same monotonic relationship with $p$ as the other problem-based approaches.
The fact that reduction is very high for small values of $p$ can be explained by the relatively small number of solutions in the feasible set $\X_{n,p}$.
For example, when $p=1$ then $\X_{n,p}$ consists of only $n$ solutions, which means that we need at most $n$ scenarios in order to preserve the worst-case value for each solution.
As $p$ increases, we have a decrease in the effectiveness as the number of feasible solutions sharply increase, but the effectiveness then increases again as the conic hull of the feasible solution set gets smaller.

\begin{figure}[htbp]
  \centering
  \begin{subfigure}{0.49\textwidth}
    \includegraphics[width=\textwidth]{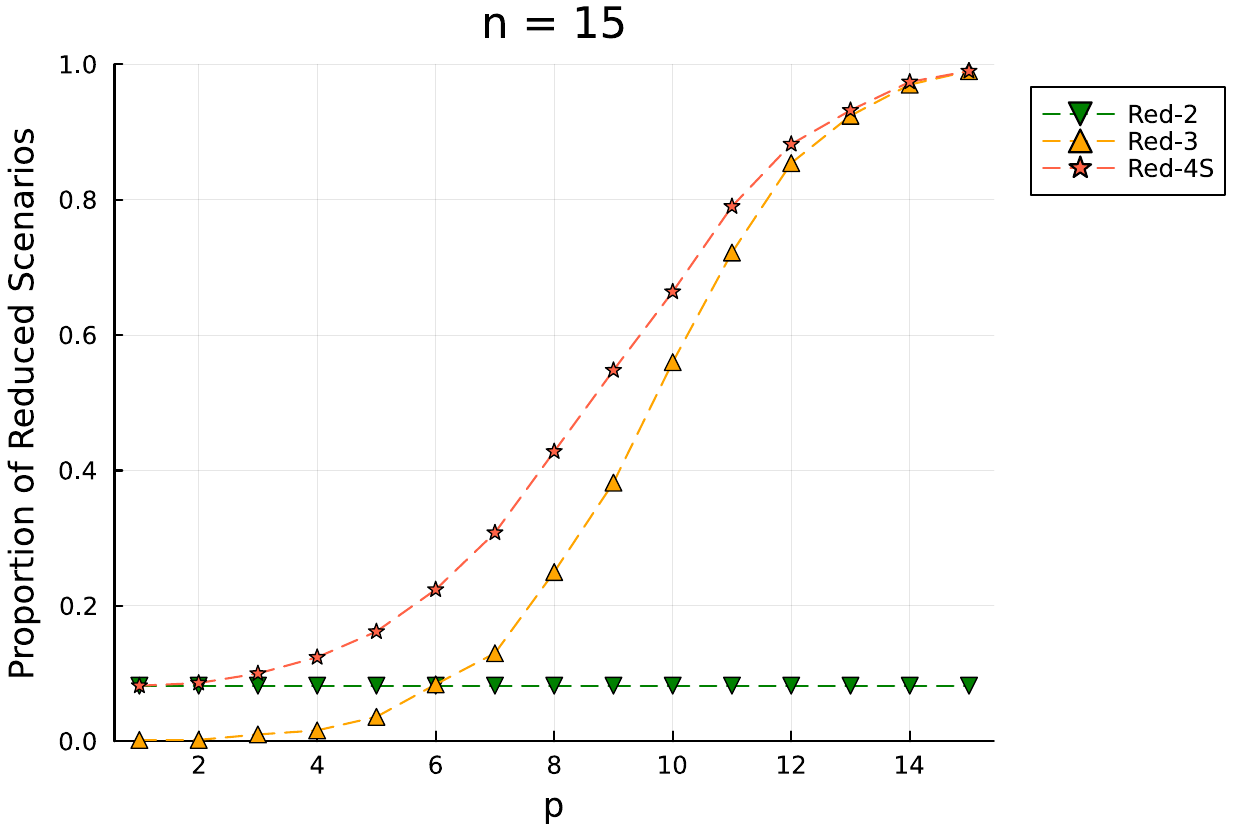}
  \end{subfigure}
  \begin{subfigure}{0.49\textwidth}
    \includegraphics[width=\textwidth]{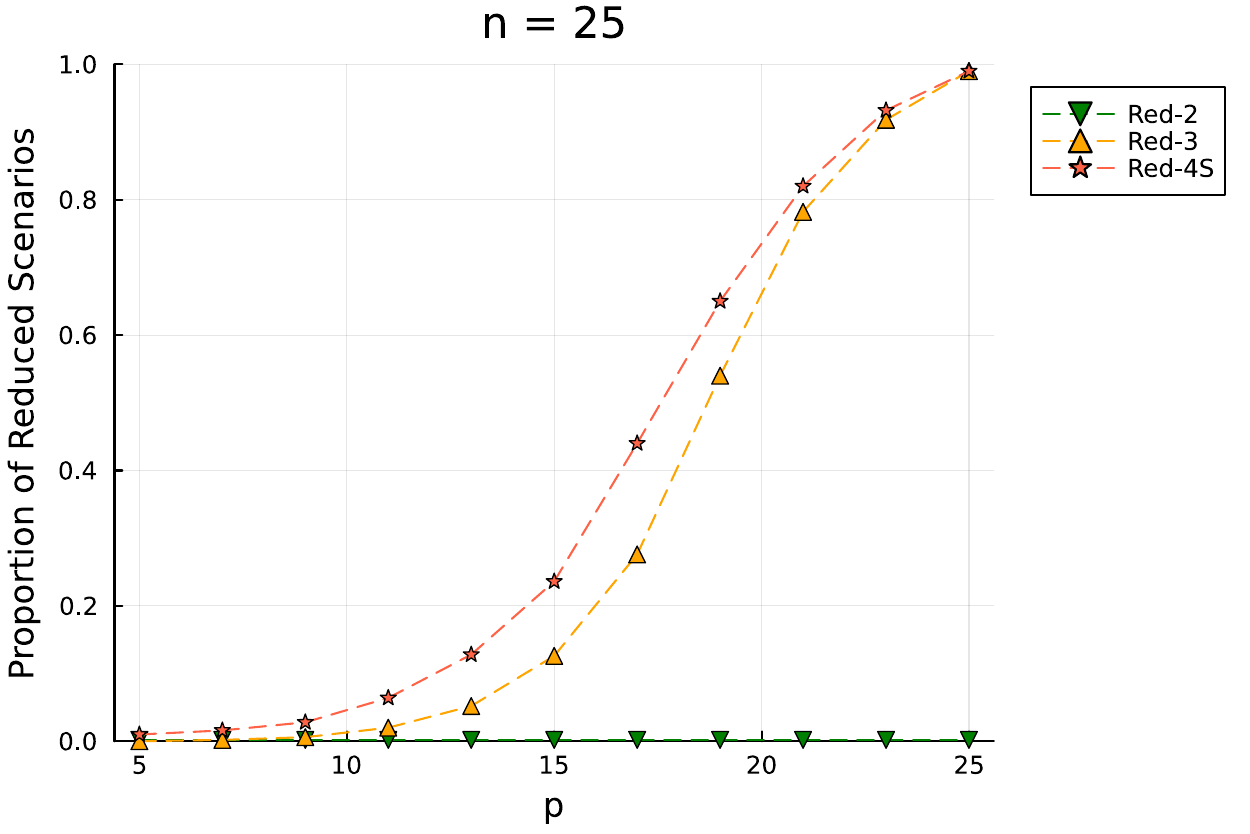}
  \end{subfigure}
  \begin{subfigure}{0.49\textwidth}
    \includegraphics[width=\textwidth]{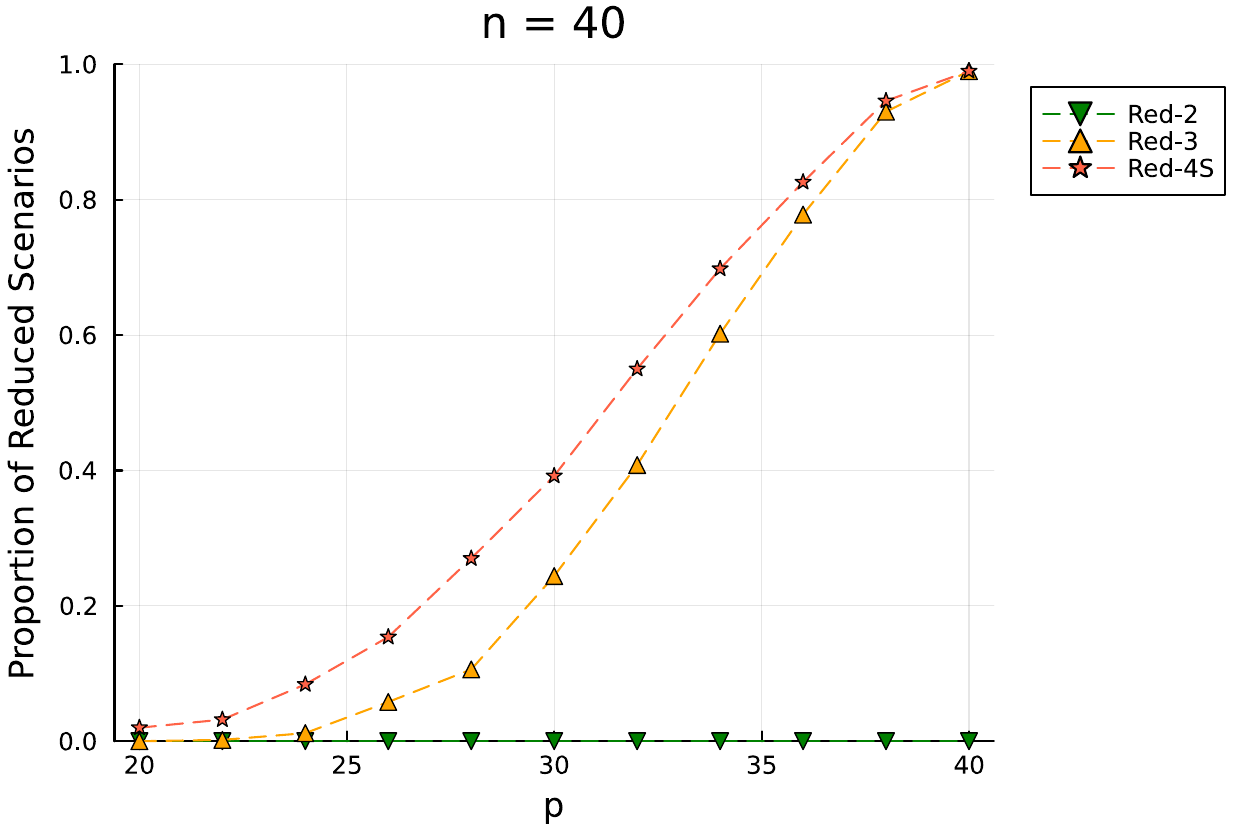}
  \end{subfigure}
  \caption{Reduction using for selection problem for dimensions $n=15$, $25$ and $40$}
  \label{fig:reduction-selection-exact-medium-dims}
\end{figure}

Figure~\ref{fig:reduction-selection-exact-medium-dims} shows the average reduction from 100 scenarios for dimensions 15, 20 and 25.
For these dimensions we cannot perform type (iv) reduction and we do not perform any reduction of type (i) since this method results in no reduction for higher dimensions.
Again, we observe that the effectiveness of type (iii) and approximate type (iv) reduction increase with $p$, but we need much higher values of $p$ for it to be effective at all.
For example for $n=40$, we need $n > 28$ for type (iii) reduction to have much effect at all.
We also observe in this case that for $n > 15$, type (ii) reduction has essentially no effect.

\subsubsection{Shortest Path Results}

For the shortest path problem we solve the problem for layered directed graphs with layers ($L$) and widths ($W$) between 2 and 6. Again, we can only do type (iv) reduction for problems with smaller dimensions.
The results showing average reductions achieved for the different types of reduction are shown in Table~\ref{tab:sp-red}.
To aid the interpretation of these results we also include a column for the dimension $n$ of the problem, namely $n=2W + (L-1)W^{2}$.
Although we did perform reduction experiments for problems with larger dimensions without type (iv), they nearly all achieved zero reductions so we do not show detailed tables of results for these.

\begin{table}[htb]
  \centering
\begin{tabular}{c|c|c|c|c|c|c|c}
L & W & dim & Red-1 & Red-2 & Red-3 & Red-4 & Red-4S\\
\hline
2 & 2 & 8 & 0.162 & 0.452 & 0.808 & 0.960 & 0.892\\
2 & 3 & 15 & 0.004 & 0.102 & 0.432 & 0.916 & 0.708\\
2 & 4 & 24 & 0.000 & 0.004 & 0.102 & 0.866 & 0.414\\
2 & 5 & 35 & 0.000 & 0.000 & 0.014 & 0.796 & 0.200\\
2 & 6 & 48 & 0.000 & 0.000 & 0.004 & 0.740 & 0.104\\
3 & 2 & 12 & 0.018 & 0.200 & 0.692 & 0.936 & 0.850\\
3 & 3 & 24 & 0.000 & 0.002 & 0.112 & 0.806 & 0.412\\
3 & 4 & 40 & 0.000 & 0.000 & 0.006 & 0.644 & 0.132\\
4 & 2 & 16 & 0.000 & 0.082 & 0.504 & 0.890 & 0.722\\
4 & 3 & 33 & 0.000 & 0.000 & 0.028 & 0.656 & 0.196\\
5 & 2 & 20 & 0.000 & 0.026 & 0.344 & 0.804 & 0.610\\
6 & 2 & 24 & 0.000 & 0.004 & 0.268 & 0.774 & 0.532\\
\end{tabular}
  \caption{Shortest Path Exact Reduction Results}
  \label{tab:sp-red}
\end{table}

These results again demonstrate the relative effectiveness of the different types of reduction (according to Theorem~\ref{th:implications}), and that the effectiveness of all methods diminishes as the dimension of the problem grows, and we see that only exact type (iv) reduction is able to produce large reductions across all combinations of $L$ and $W$.
What is particularly interesting here is to compare the reductions here to those which would be achieved using the selection problem as a relaxation (see Section~\Ref{sec:short-path-probl}).
That is, for a shortest path problem with parameters $L$ and $W$, it is valid to use types (iii) and (iv) reduction for the selection problem with $n=2W +(L-1)W^{2}$ and $p=l+1$.
For example, for $L=2$ and $W=3$ type (iii) and approximate type (iv) reduction give us average reductions of 0.43 and 0.71.
On the other hand, using the selection problem relaxation with $n=15$ and $p=4$ gives respective reductions of only 0.016 and 0.124.
This demonstrates that it is important to make use of as much problem structure as possible when it comes to doing reduction.

\subsection{Experiment 2: Scenario Approximation}

\subsubsection{Setting}

To assess the scenario approximation approach, we again consider the selection problem as well as the shortest path problem with layered graphs. For selection problems, we generate a set of small-scale instances, and a set of large-scale instances. We consider instances with $n \in \{8, 12,15\}$ and $p \in \{1,2,\ldots,n\}$ to be small, and instances with $n \in \{20,30,50,100\}$ and $p \in \{2,4,\ldots,n\}$ to be large. In both cases, we set $N=50$ and $K=5$ (i.e., we approximate 50 original scenarios using 5 scenarios). We chose a smaller value of $N$ in comparison to the previous experiment due to the increased computational difficulty of solving the approximation problem. For each combination of $n$, $N$ and $K$, we generate 50 instances and always show the average results. We use a 300-second time limit for each optimization problem.

In the experiment of the shortest path problem we first set $W=3$ with $L \in \{2,\ldots,15\}$, and then set $L=4$ with $W \in \{2,\ldots,8\}$. In addition, we set the size of the original uncertainty set to be $N=20$ and the reduced uncertainty set to be $K=5$ for both cases.

All instances of the selection and shortest path problems are randomly sampled uniformly from the set $\{1,2, \ldots,100\}$. Each generated instance for these problems is solved using all scenarios to determine the optimal objective value, or an estimate if the 300-second time limit is reached.

Optimal solution methods are used for all instances except for the large-scale instances of the selection problem, where heuristics are applied. In each case, the results are compared with the k-means solution method.
For large-scale instances of the selection problem, we apply a heuristic pre-processing. We first use the K-means algorithm to cluster the scenario set into $K$ clusters. We then use our scenario approximation methods to represent each of these cluster by a single scenario. Depending on the method applied for this second step, we refer to these heuristic methods as App-1/2 (App-3) as well.

For all scenario approximation experiments, we use CPLEX version 22.11 on a GenuineIntel pc-i440fx-7.2 CPU computer server running at 2.00 GHz with 15 GB of RAM. Additionally, only one core is used for each instance evaluation. Code was implemented in C++.

\subsubsection{Selection Problem Results}

For small-scale instances, we discuss $n=12$ as a representative set of results. Corresponding plots for $n=8$ and $n=15$ are presented in Appendix~\ref{sec:expappendix}. In Figure~\ref{fig:guarantee-n12}, we compare the guarantee found by the respective optimization model to the guarantee obtained by fixing the reduced uncertainty set of each method, including k-means, in the iterative approximation (iv), which we refer to as Guarantee-T4. The Guarantee-T4 value of App-3 is nearly identical to the value found by App-3 itself, while the lines for App-4 and its Guarantee-T4 overlap. The best guarantee for $p=1$ is achieved by App-4, but this method performs worse than App-3 due to the time limit we imposed. In Figure~\ref{fig:optgap-n12}, we assess the quality of the solutions found by each method by plotting the actual ratio to the optimal solution. This means that we take an optimal solution to the robust problem with approximate uncertainty set and calculate its robust objective value with respect to the original uncertainty set. We compare this value to the optimal objective value of the robust problem with the original uncertainty set. This is denoted as UB ratio. Here, App-3 provides the best solution in most cases. The relatively poor performance of App-4 in comparison to App-3 in this experiment is likely due to the computational challenge that these models pose. We note a similarity to the behavior in $p$ that has been observed in the previous case of scenario reduction: The larger the value of $p$, the better the performance of our methods, when we exclude small values of $p$. Again, this improvement eludes the previous aggregation methods. We also note that scenario approximation can achieve very strong bounds, even when exact reduction is ineffective.

\begin{figure}[htbp]
	\centering
		\begin{subfigure}{0.45\textwidth}
         		\centering
         			\includegraphics[width=\textwidth]{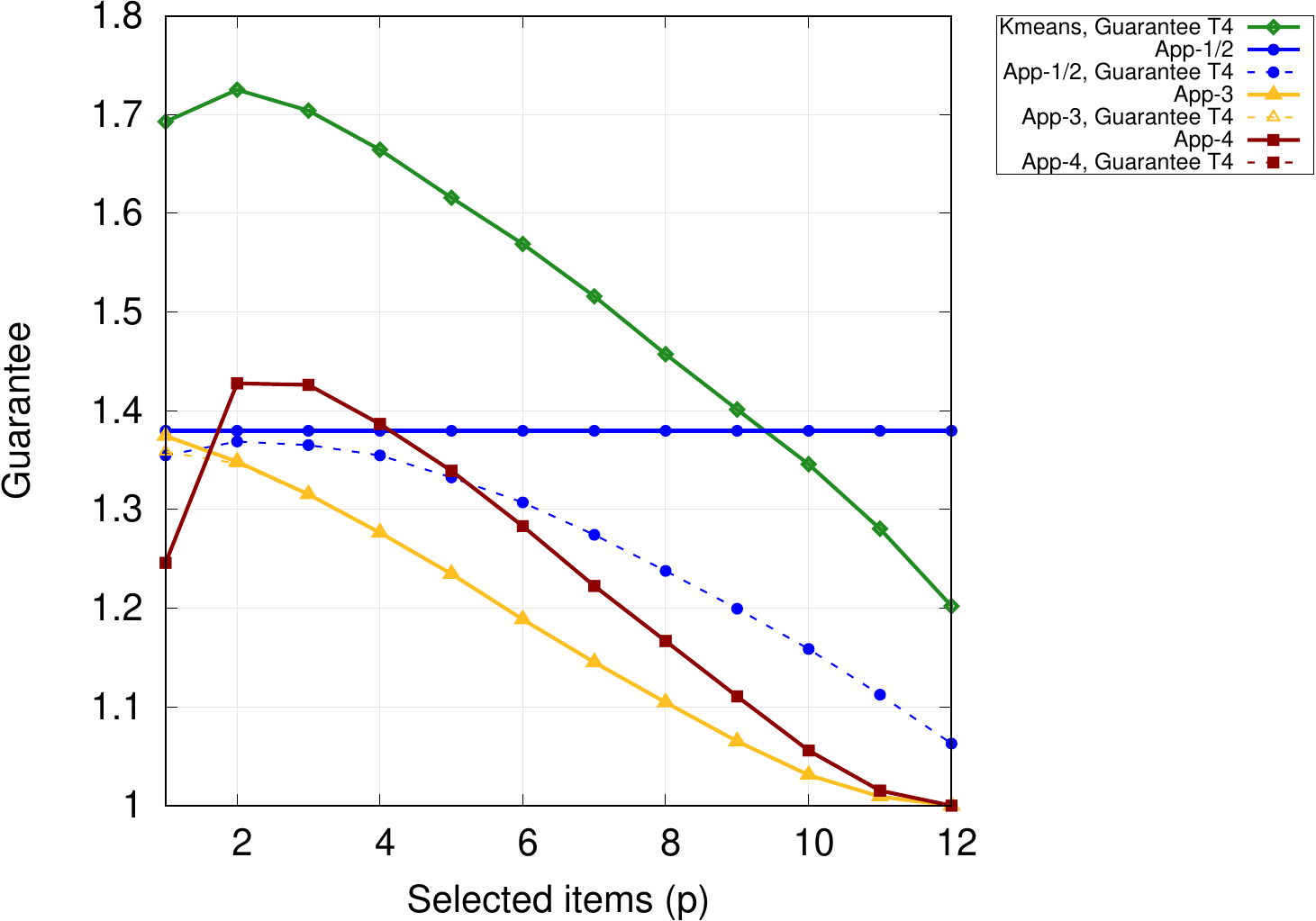}
         		\caption{Guarantee ($n=12$)}\label{fig:guarantee-n12}
		\end{subfigure}
     		\begin{subfigure}{0.45\textwidth}
         		\centering
         			\includegraphics[width=\textwidth]{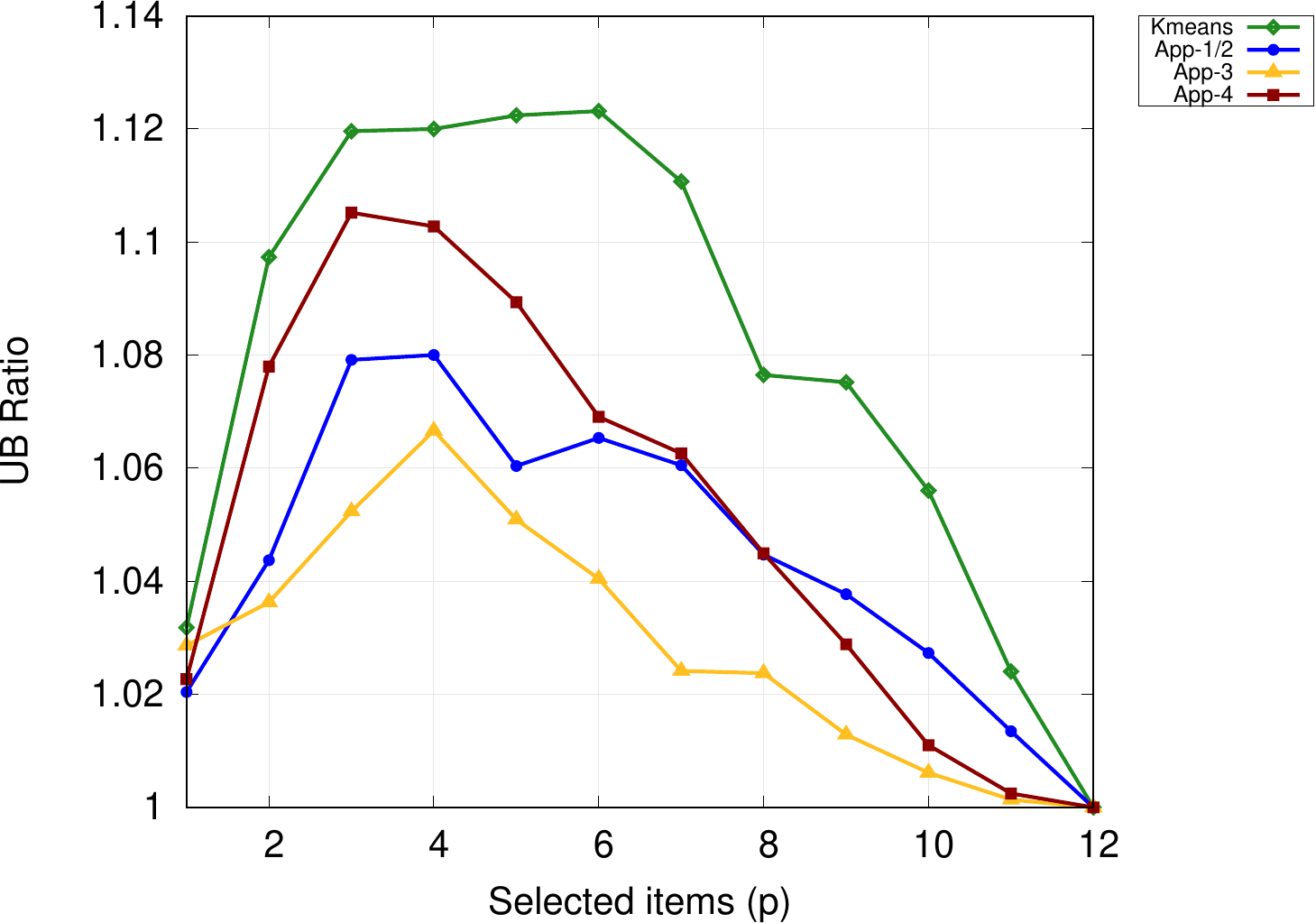}
         		\caption{UB Ratio ($n=12$)}\label{fig:optgap-n12}
     		\end{subfigure}
        		\caption{Selection - performance of small instances}\label{fig:sel-performance-small}
\end{figure}

In Figure~\ref{fig:sel-time-performance-small}, we present the computation times, divided into three categories: the time to find the corresponding approximate scenario set (labeled as aggregation time in Figure~\ref{fig:aggtime-n12}), the time to solve the resulting reduced robust optimization problem (labeled as solution time in Figure~\ref{fig:soltime-n12}), and the process time, which is the sum of both, compared to the time to solve the original robust optimization problem (Figure~\ref{fig:protime,n12}). The solution time for all methods is very fast, since the size of the reduced problem is small, making the aggregation time and process time appear almost the same. App-4 is the slowest approach and hits the time limit for $p=6$. Given the small size of the problem in this case, CPLEX outperforms all methods.

\begin{figure}[htbp]
\centering
\begin{subfigure}{0.32\textwidth}
\centering
\includegraphics[width=\textwidth]{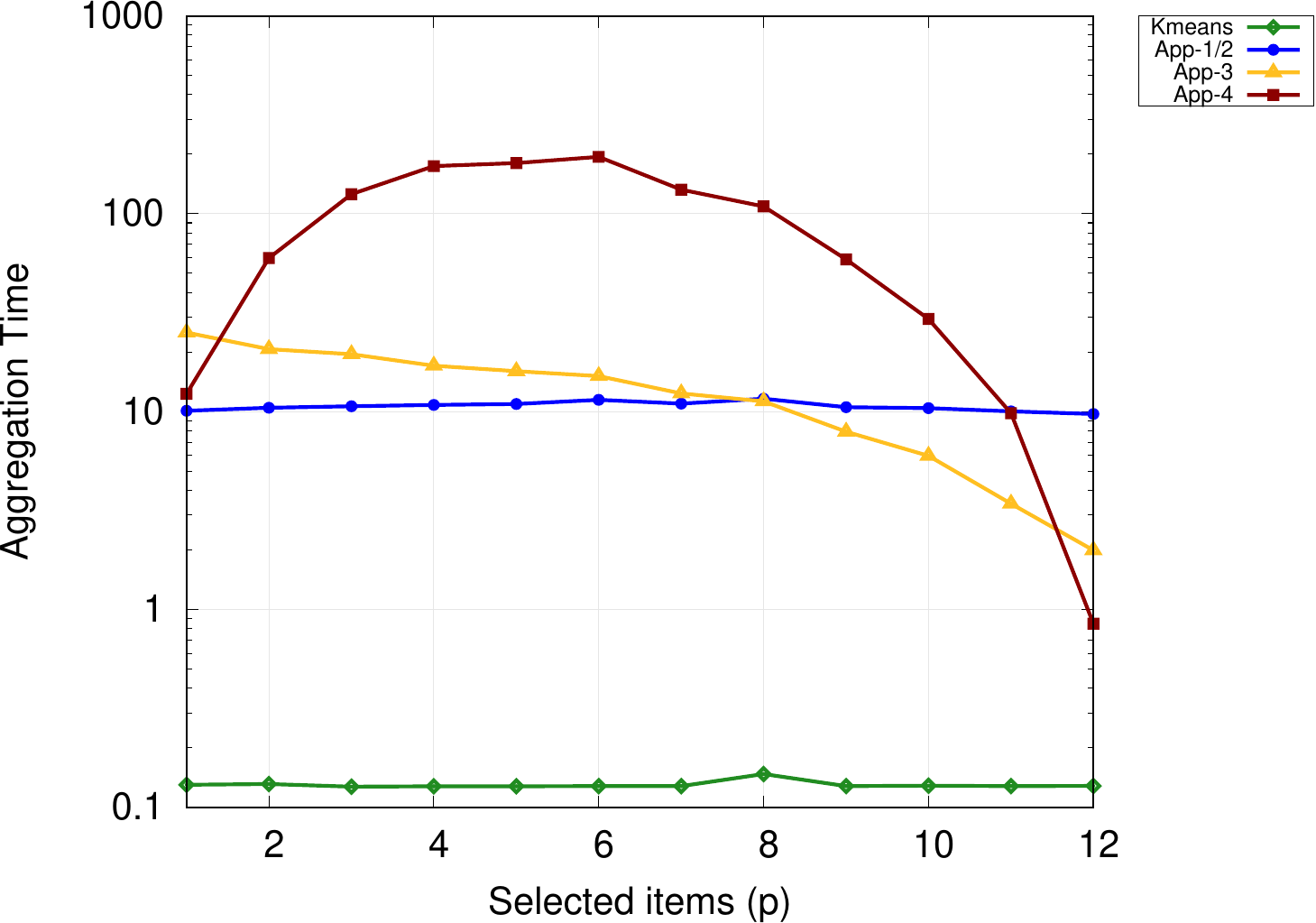}
\caption{Aggregation Time}\label{fig:aggtime-n12}
\end{subfigure}
\begin{subfigure}{0.32\textwidth}
\centering
\includegraphics[width=\textwidth]{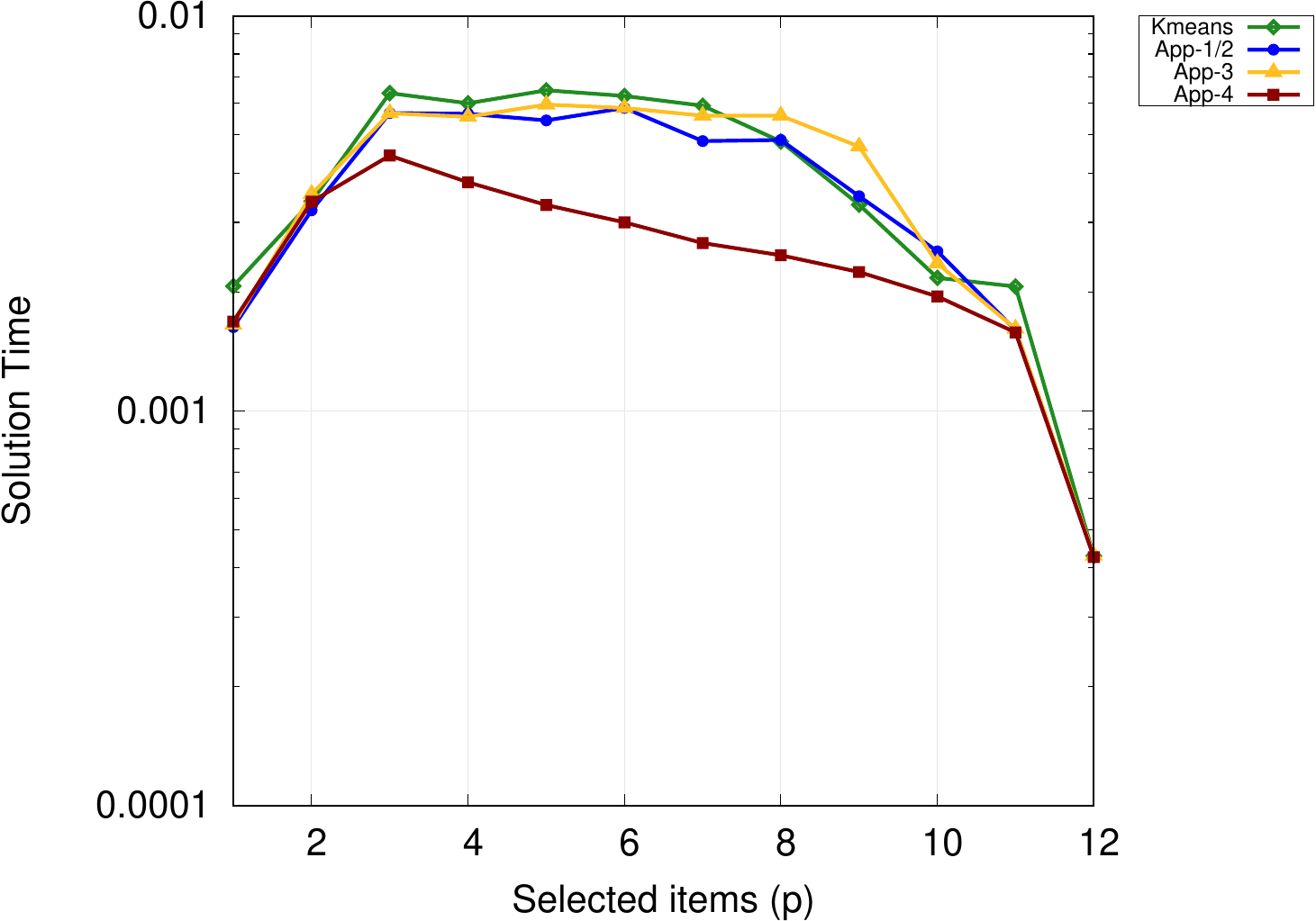}
\caption{Solution Time}\label{fig:soltime-n12}
\end{subfigure}
\begin{subfigure}{0.32\textwidth}
\centering
\includegraphics[width=\textwidth]{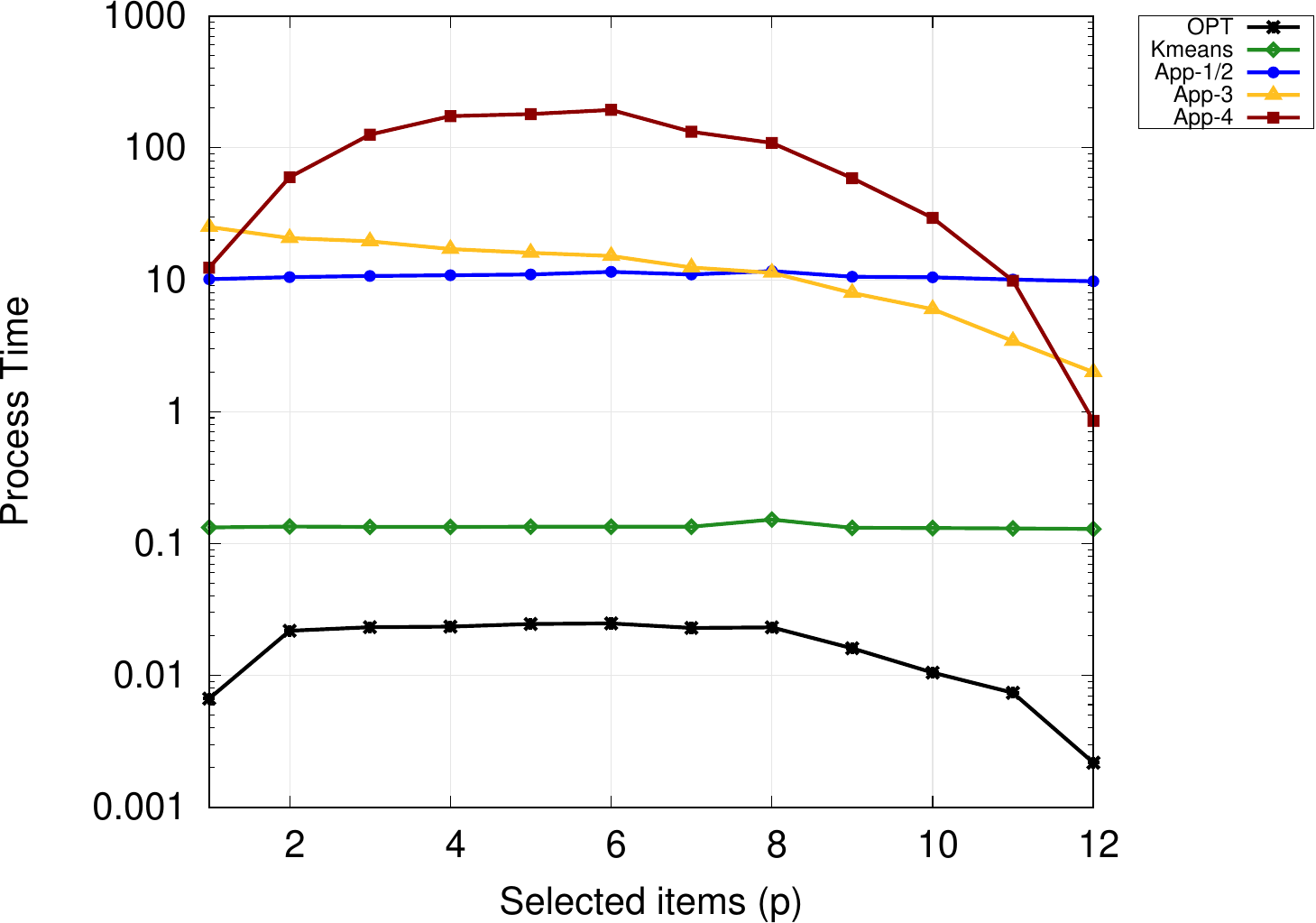}
\caption{Process}\label{fig:protime,n12}
\end{subfigure}
\caption{Selection - time performance of small instances}\label{fig:sel-time-performance-small}
\end{figure}

We now turn to large-scale instances, where we discuss $n=50$ as a representative case, and present additional results on $n=20$, $n=30$, and $n=100$ in the appendix. We do not use App-4 here, due to its large computation times. Similarly to the small instances, in Figure~\ref{fig:sel-performance-larg}, we illustrate the quality (guarantee and actual ratio) of the approximated uncertainty sets. In this scenario, App-3 consistently offers the best guarantee; however, the quality of its solution sometimes exceeds that of App-1, while at other times it is dominated.

\begin{figure}
	\centering
		\begin{subfigure}{0.45\textwidth}
         		\centering
         			\includegraphics[width=\textwidth]{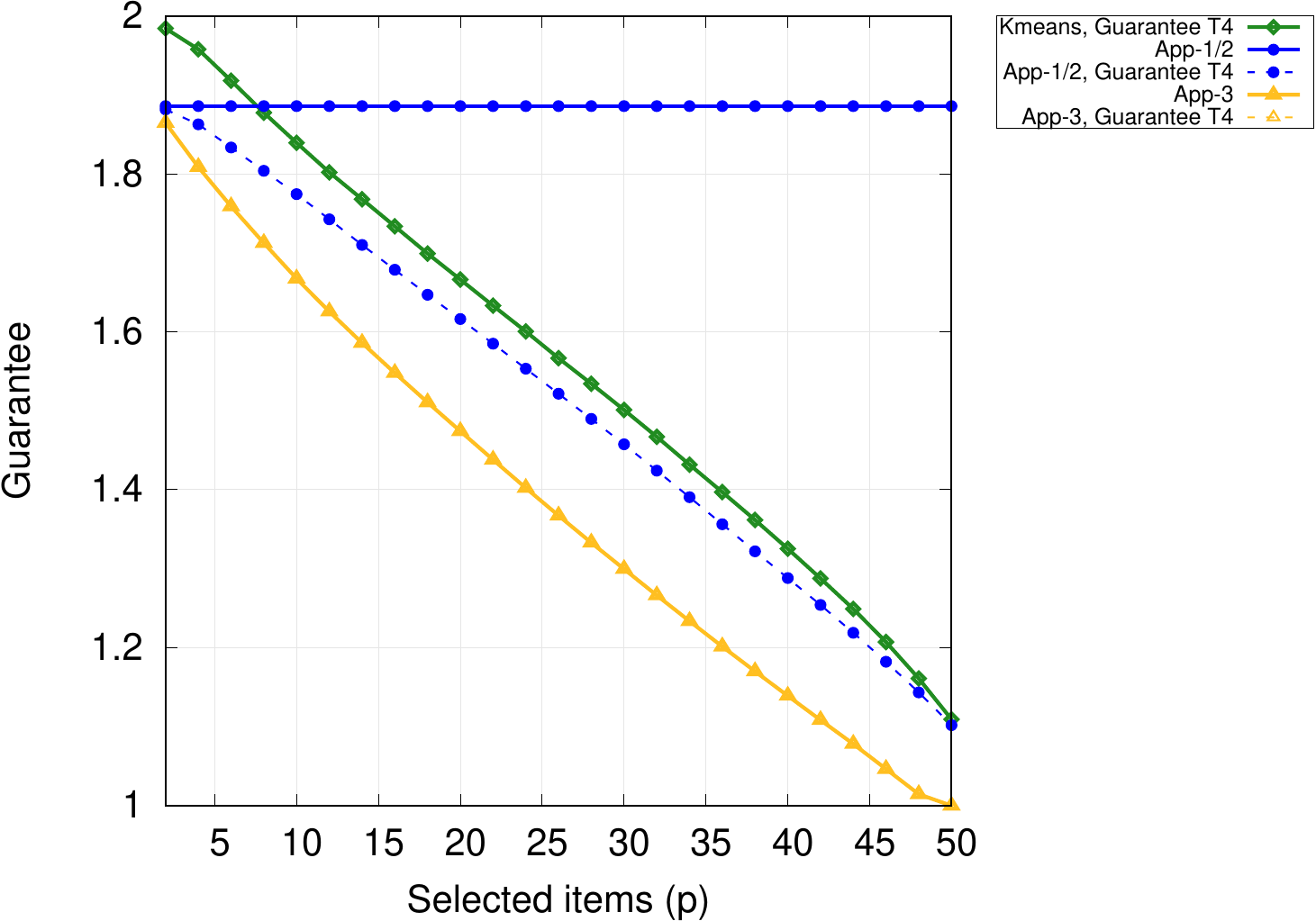}
         		\caption{Guarantee ($n=50$)}\label{fig:guarantee-t4-50}
		\end{subfigure}
     		\begin{subfigure}{0.45\textwidth}
         		\centering
         			\includegraphics[width=\textwidth]{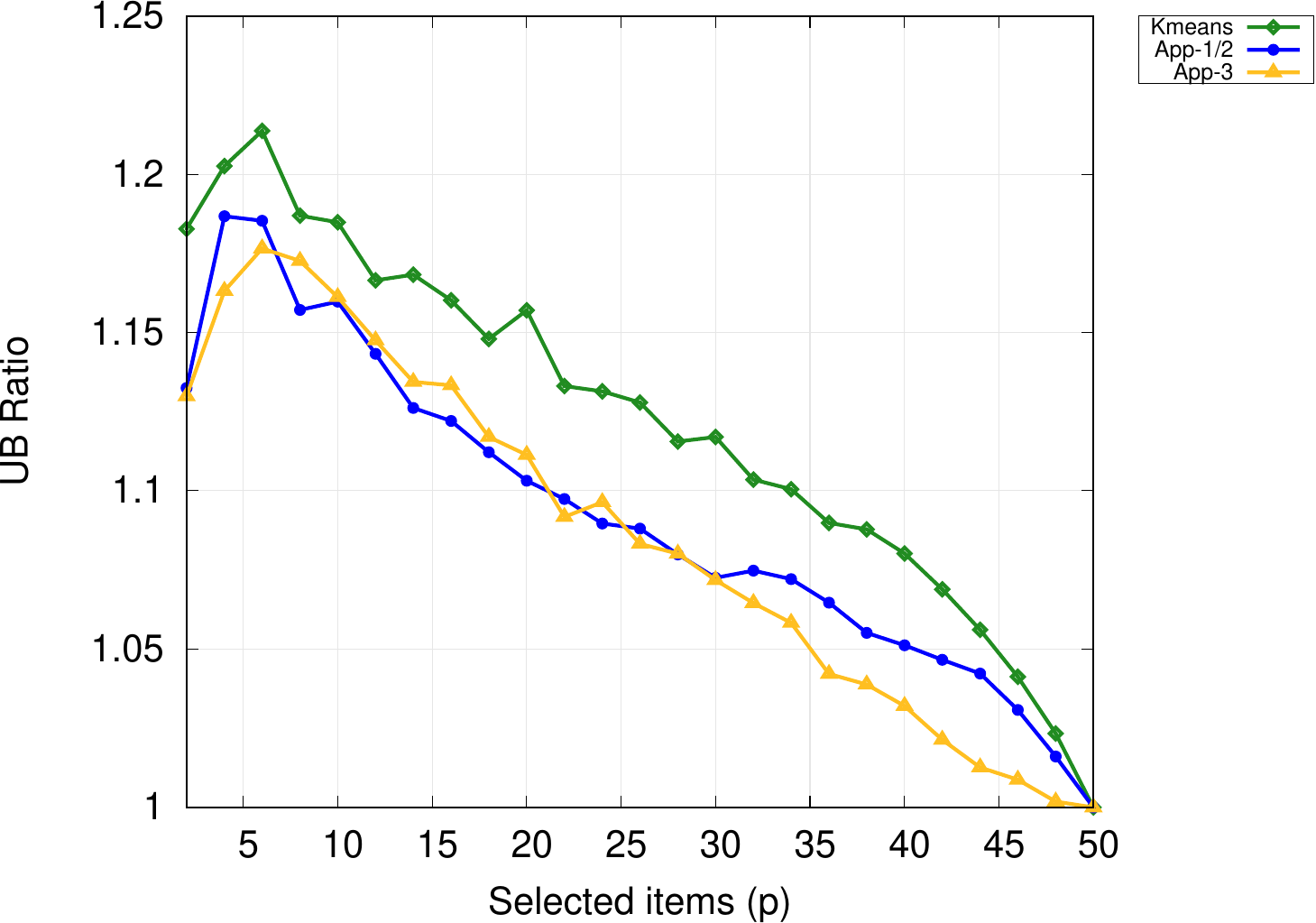}
         		\caption{UB Ratio ($n=50$)}\label{fig:ubratio-50}
     		\end{subfigure}
        		\caption{Selection - performance of large instances}\label{fig:sel-performance-larg}
\end{figure}

The impact of our methods can be emphasized by examining the computation times in Figure~\ref{fig:sel-time-performance-large}, where our methods demonstrate significant efficiency compared to solving the problem over the original uncertainty set.

\begin{figure}
	\centering
		\begin{subfigure}{0.32\textwidth}
         		\centering
         			\includegraphics[width=\textwidth]{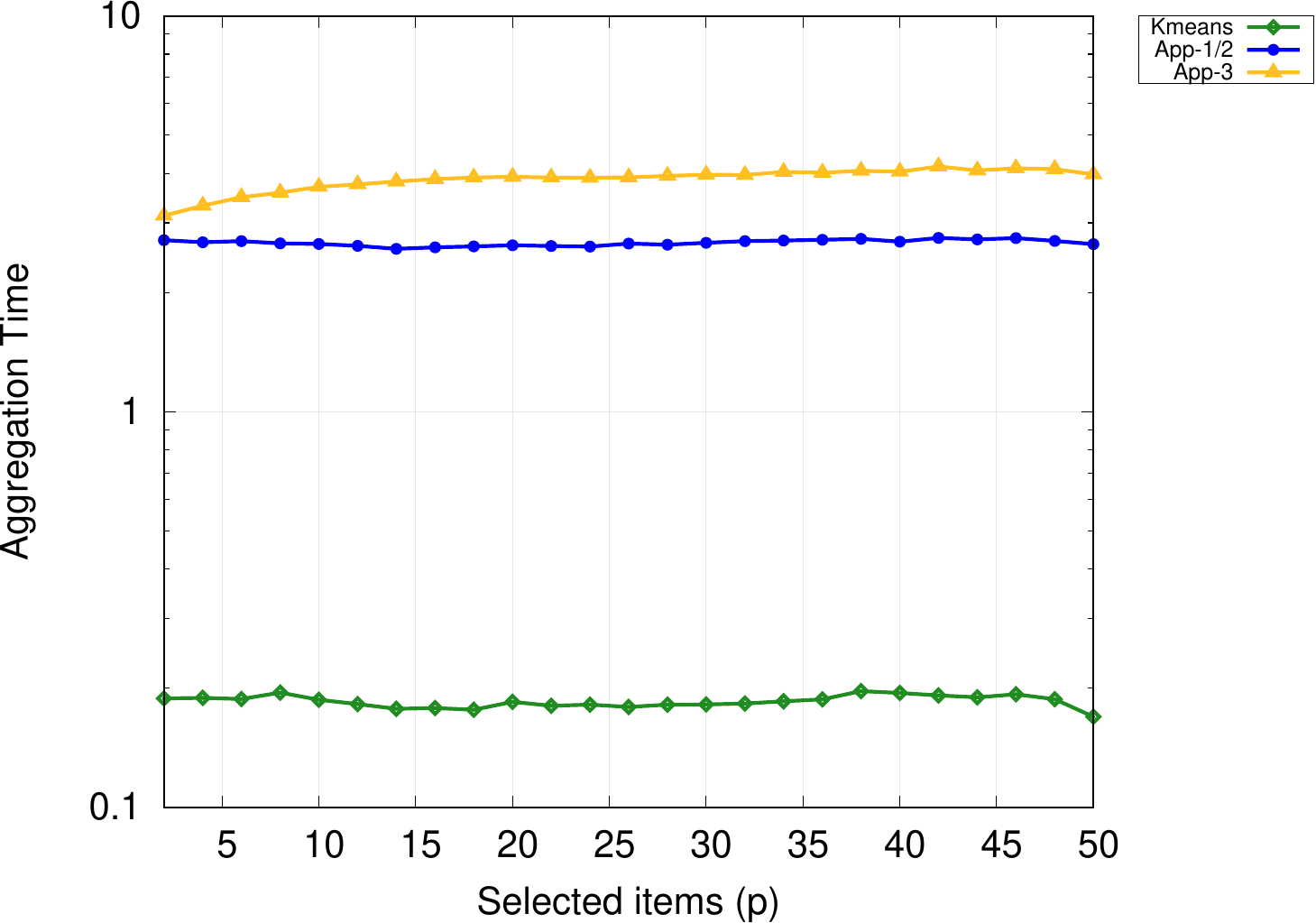}
         		\caption{Aggregation ($n=50$)}\label{fig:aggtime-n50}
     		\end{subfigure}
     		\begin{subfigure}{0.32\textwidth}
         		\centering
         			\includegraphics[width=\textwidth]{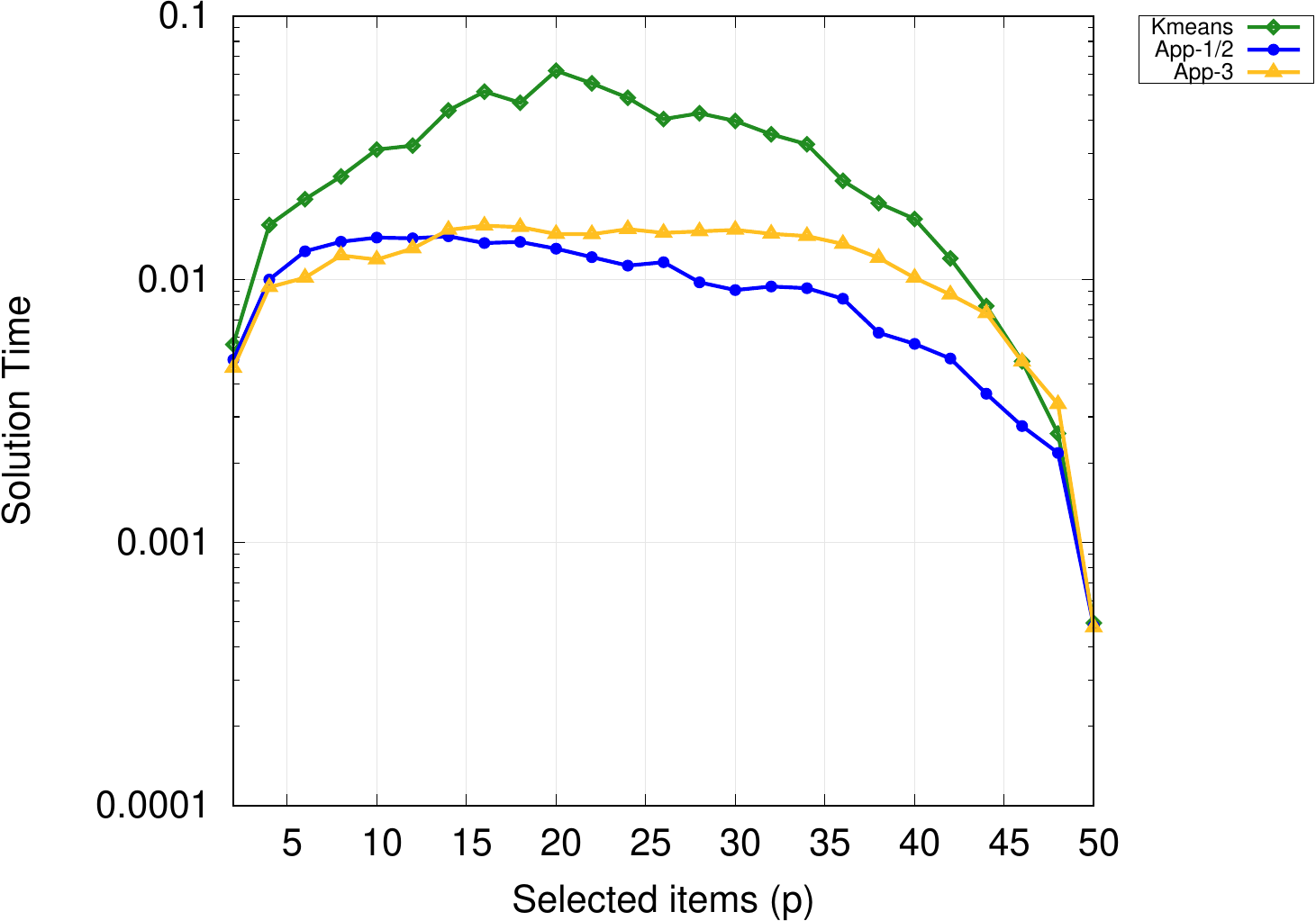}
         		\caption{Solution ($n=50$)}\label{fig:soltime-n50}
     		\end{subfigure}
     		\begin{subfigure}{0.32\textwidth}
         		\centering
         			\includegraphics[width=\textwidth]{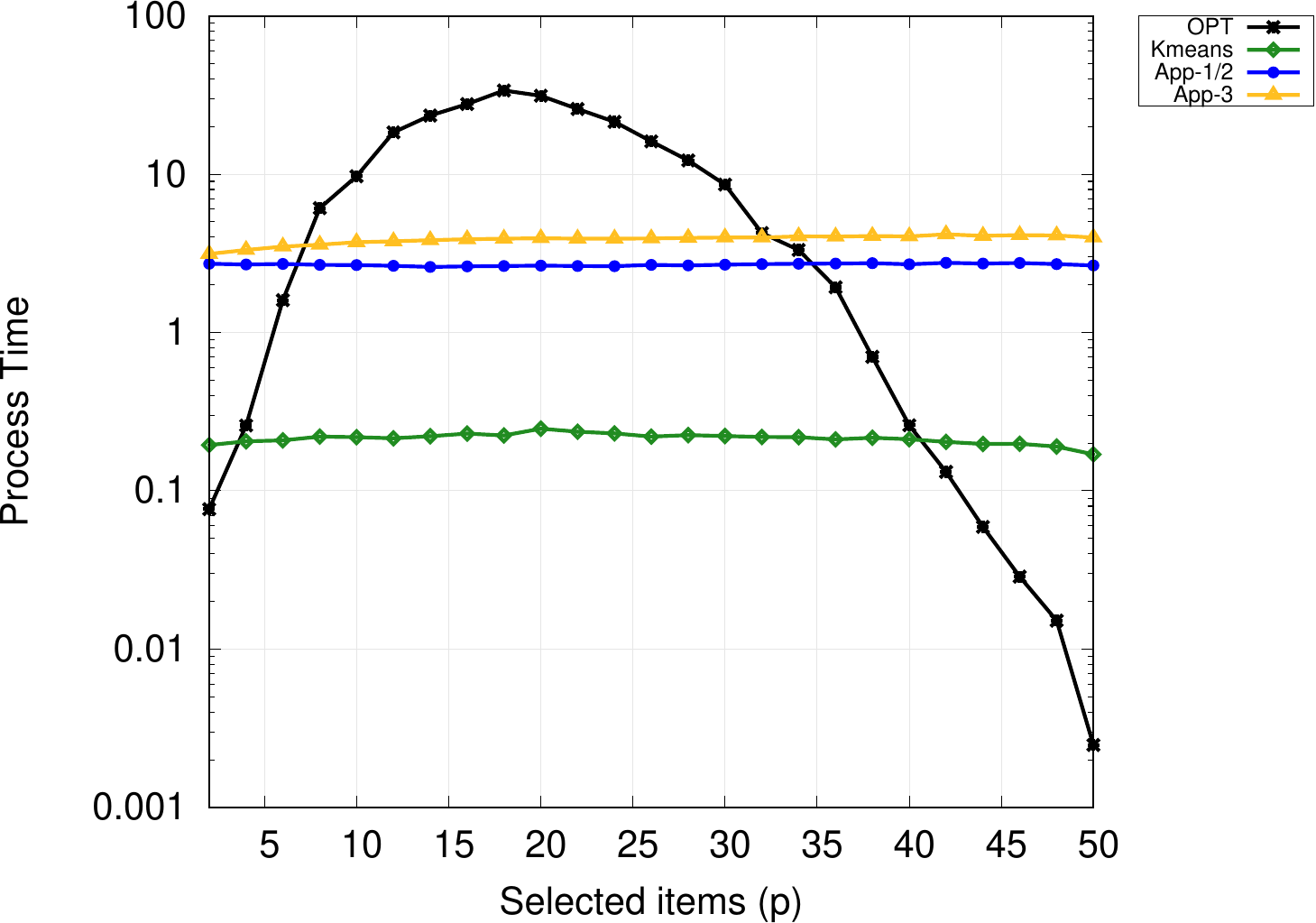}
         		\caption{Process ($n=50$)}\label{fig:protime,n50}
     		\end{subfigure}
        		\caption{Selection - time performance of large instances}\label{fig:sel-time-performance-large}
\end{figure}

\subsubsection{Shortest Path Results}

We now present the results for the shortest path problems. The quality of our methods for both settings with fixed $W=3$ is shown in Figure~\ref{fig:app-shp-performance-l}, and for fixed $L=4$ it is presented in Figure~\ref{fig:app-shp-performance-w}. In this experiment, besides Guarantee-T4, which works as in the selection experiments, we also include Guarantee-T3. Guarantee-T3 is the guarantee obtained by fixing the corresponding reduced uncertainty set into the type-(iii) model~(\ref{app3-start}-\ref{app3-end}) and finding its $t$.

In Figure~\ref{fig:app-shp-performance-l}, we observe that the iterative approximation (iv) yields better guarantees for larger cases, and Guarantee-T4 outperform Guarantee-T3 for all methods (Figures~\ref{fig:app-shp-gua3-l} and \ref{fig:app-shp-gua4-l}). Regarding solution quality, App-1/2 and App-3 perform better than the other two methods.

The results for the setting with fixed $L$, shown in Figure~\ref{fig:app-shp-performance-w}, illustrate different trends, likely due to the increasing value of $n$. Here, App-1/2 and App-3 provide better guarantees (Figures~\ref{fig:app-shp-gua3-w} and \ref{fig:app-shp-gua4-w}), while other behaviors remain similar to the case with fixed $W$.

\begin{figure}[htbp]
	\centering
		\begin{subfigure}{0.32\textwidth}
         		\centering
         			\includegraphics[width=\textwidth]{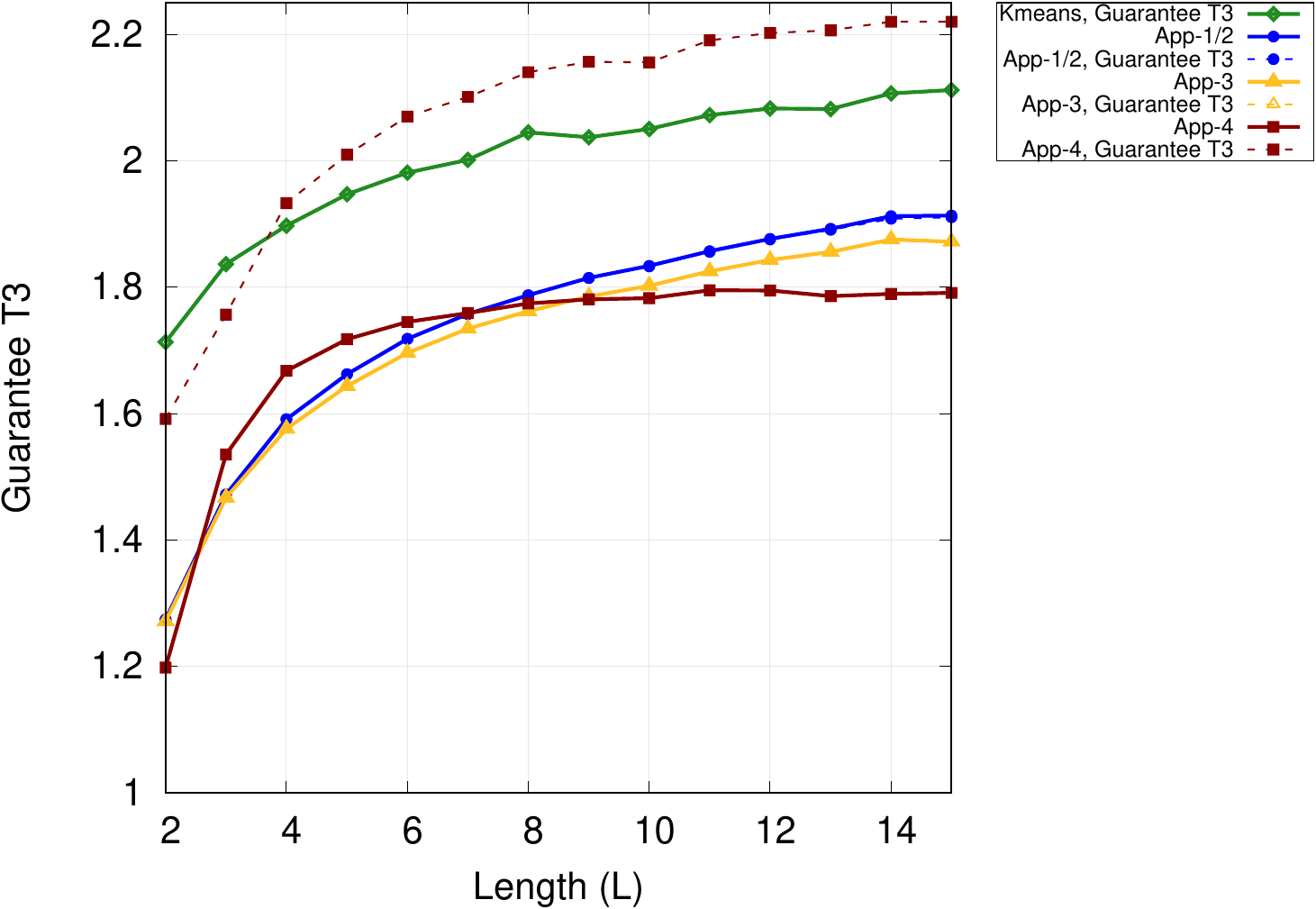}
         		\caption{Guarantee T3}\label{fig:app-shp-gua3-l}
     		\end{subfigure}
     		\begin{subfigure}{0.32\textwidth}
         		\centering
         			\includegraphics[width=\textwidth]{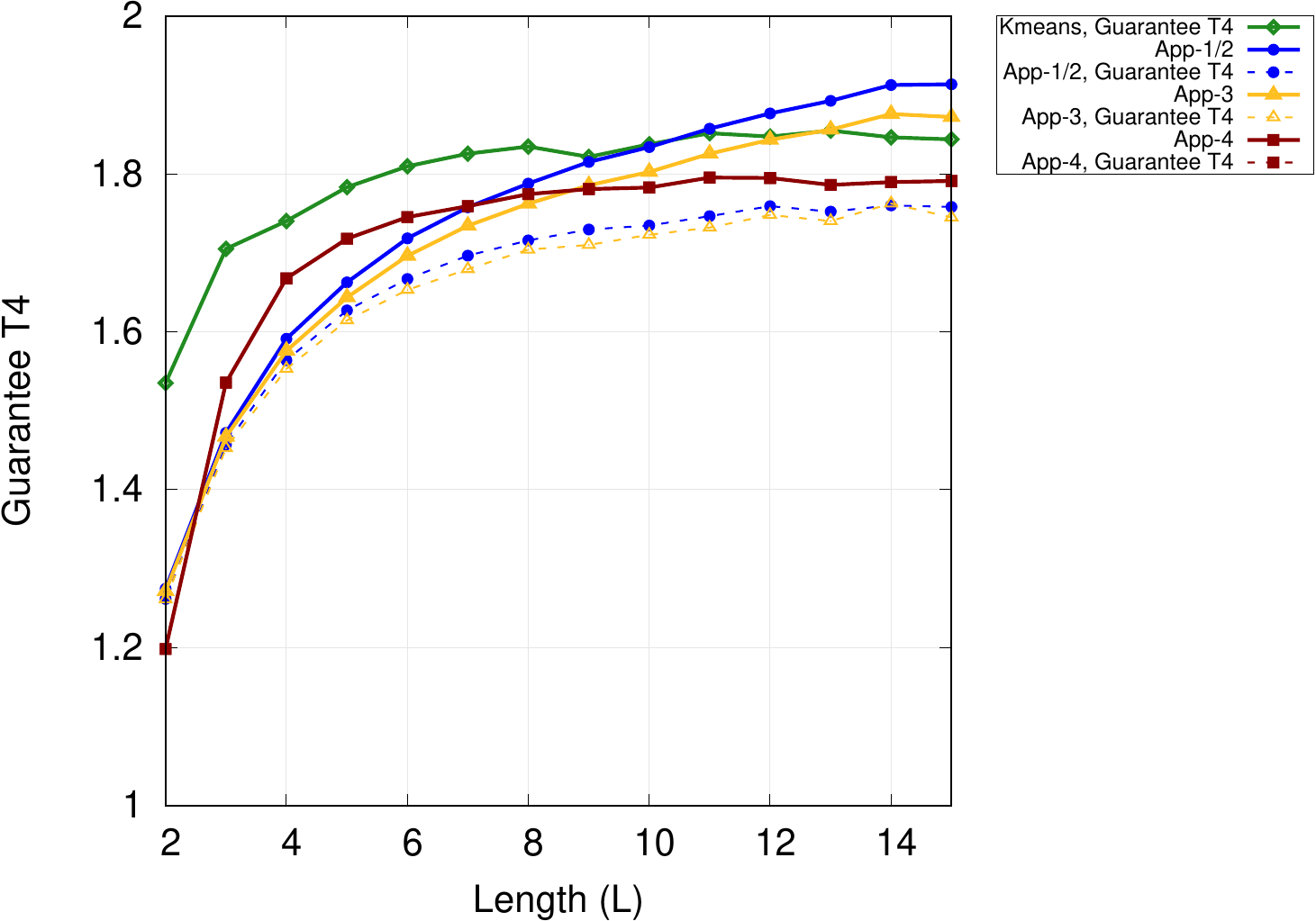}
         		\caption{Guarantee T4}\label{fig:app-shp-gua4-l}
     		\end{subfigure}
		\begin{subfigure}{0.32\textwidth}
         		\centering
         			\includegraphics[width=\textwidth]{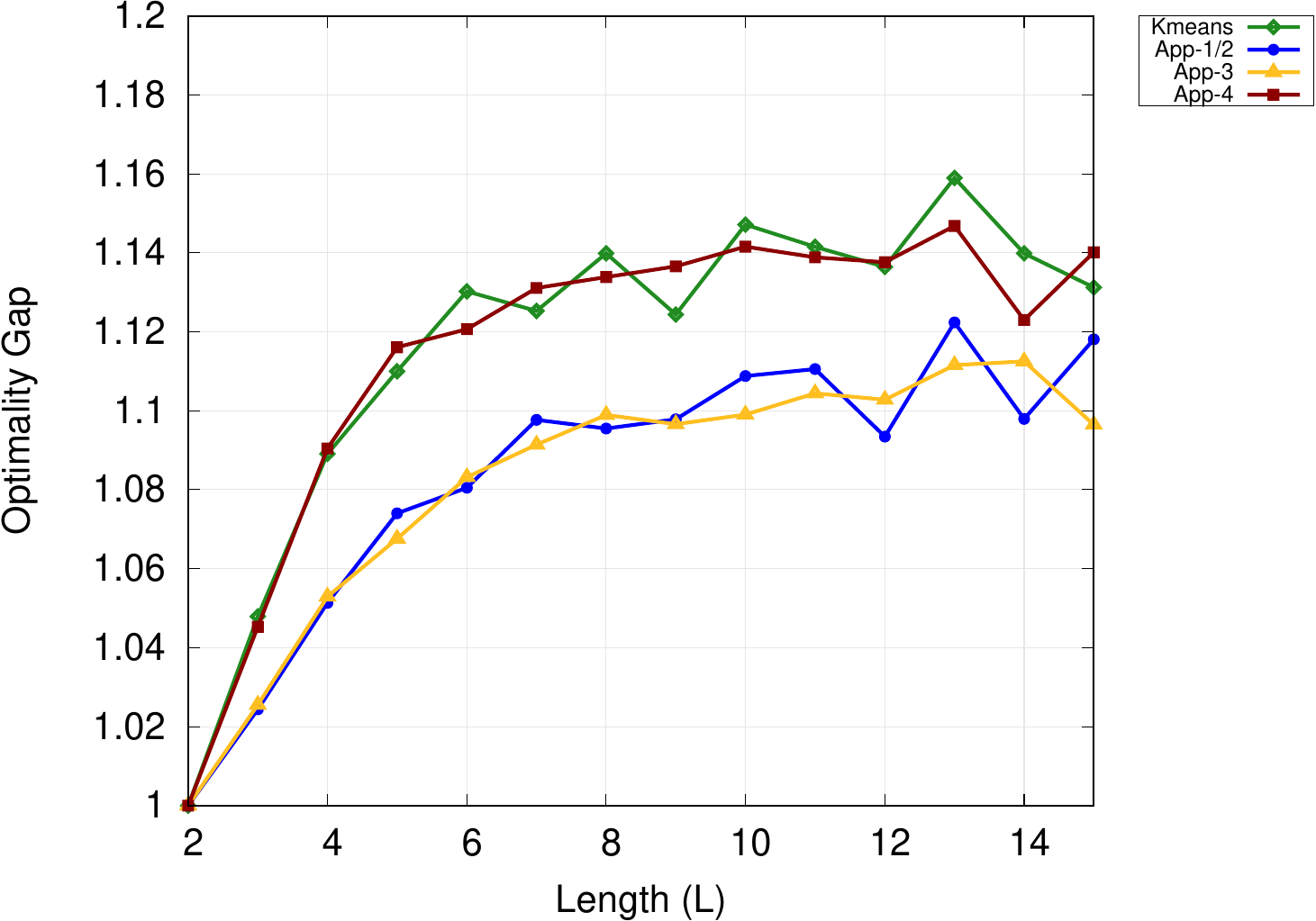}
         		\caption{UB Ratio}\label{fig:app-shp-ubratio-l}
     		\end{subfigure}
        		\caption{Shortest path - performance with fixed width ($W=3$)}\label{fig:app-shp-performance-l}
\end{figure}

\begin{figure}[htbp]
	\centering
		\begin{subfigure}{0.32\textwidth}
         		\centering
         			\includegraphics[width=\textwidth]{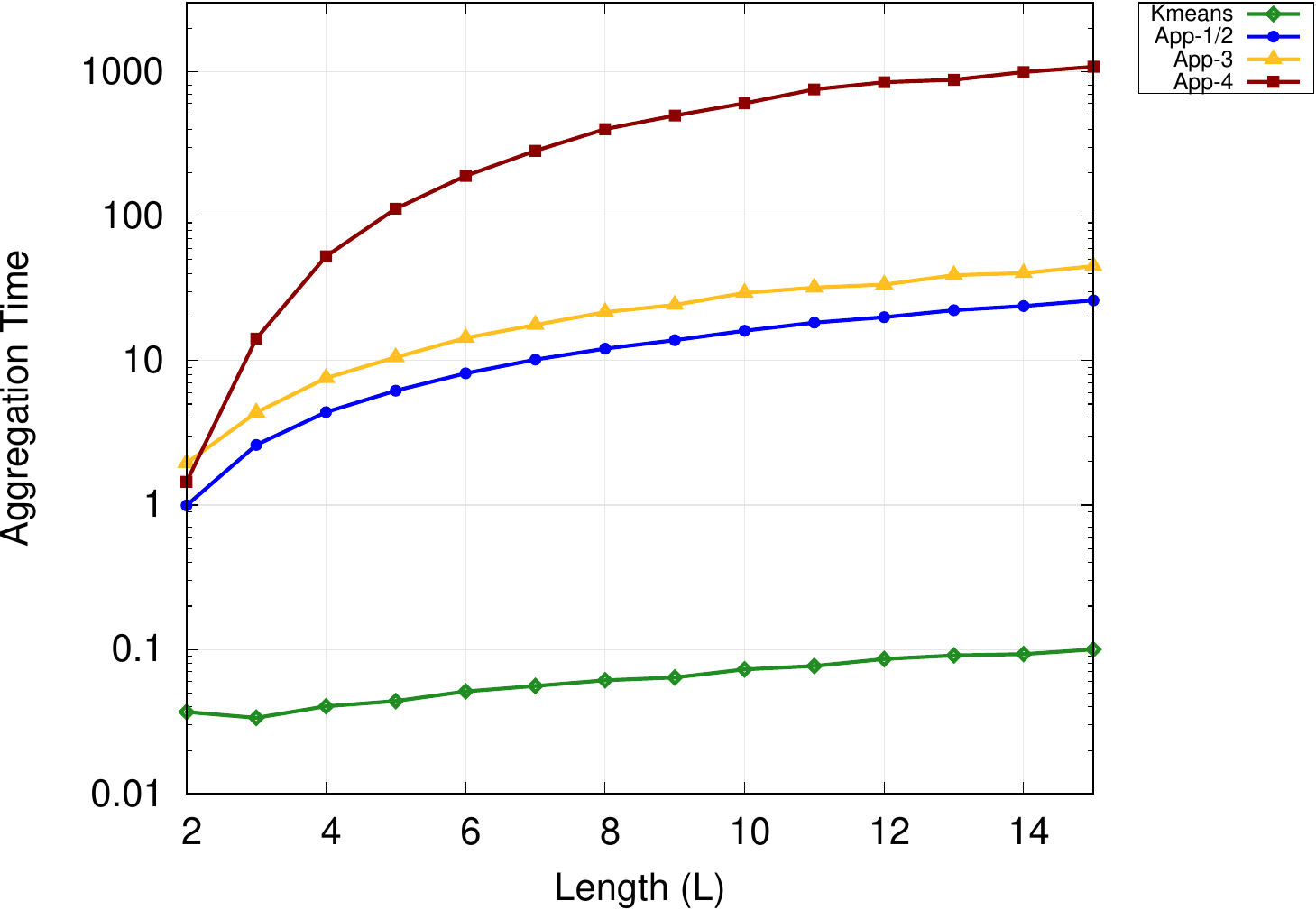}
         		\caption{Aggregation time}\label{fig:app-shp-aggtime-l}
     		\end{subfigure}
     		\begin{subfigure}{0.32\textwidth}
         		\centering
         			\includegraphics[width=\textwidth]{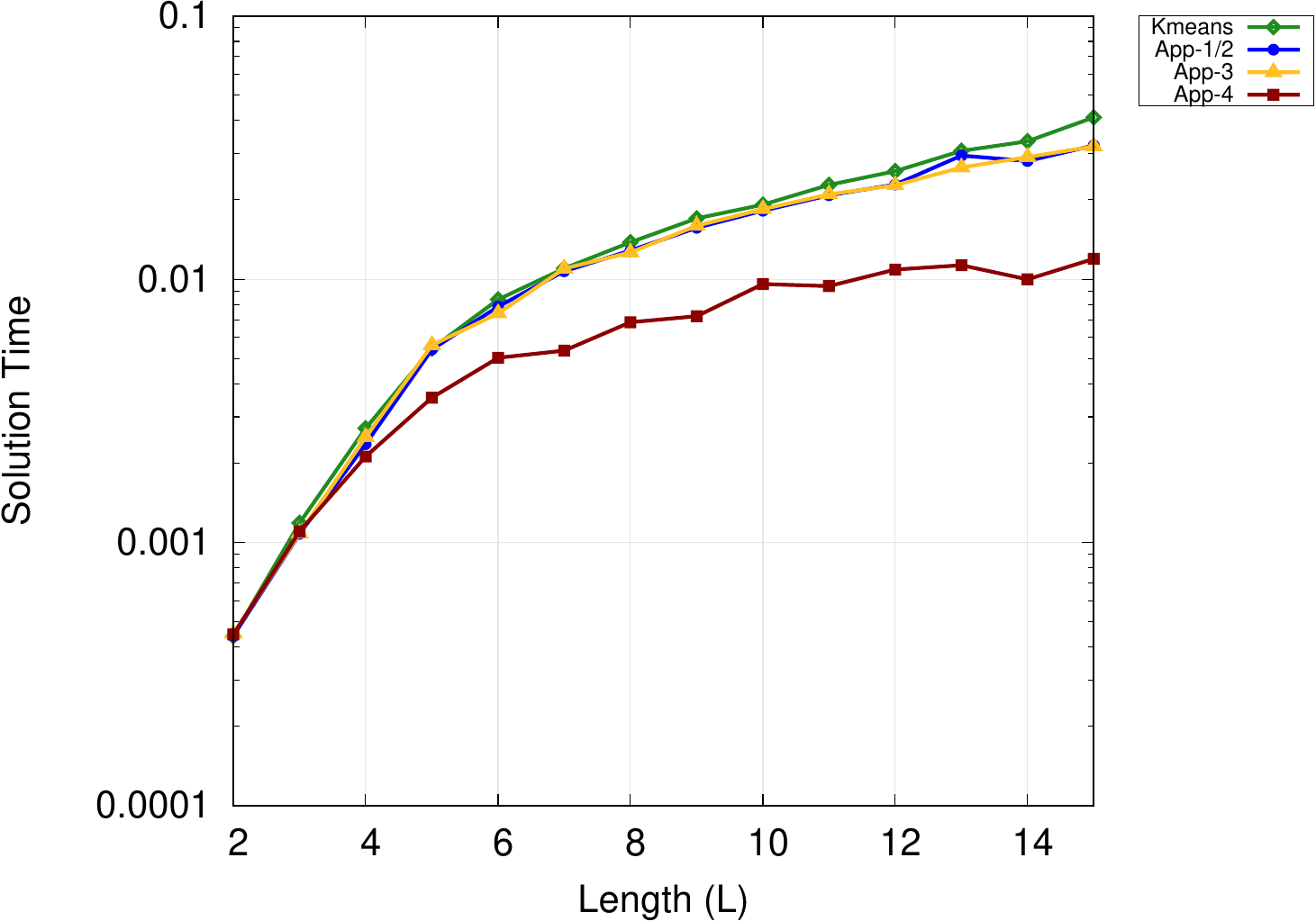}
         		\caption{Solution time}\label{ffig:app-shp-soltime-l}
     		\end{subfigure}
		\begin{subfigure}{0.32\textwidth}
         		\centering
         			\includegraphics[width=\textwidth]{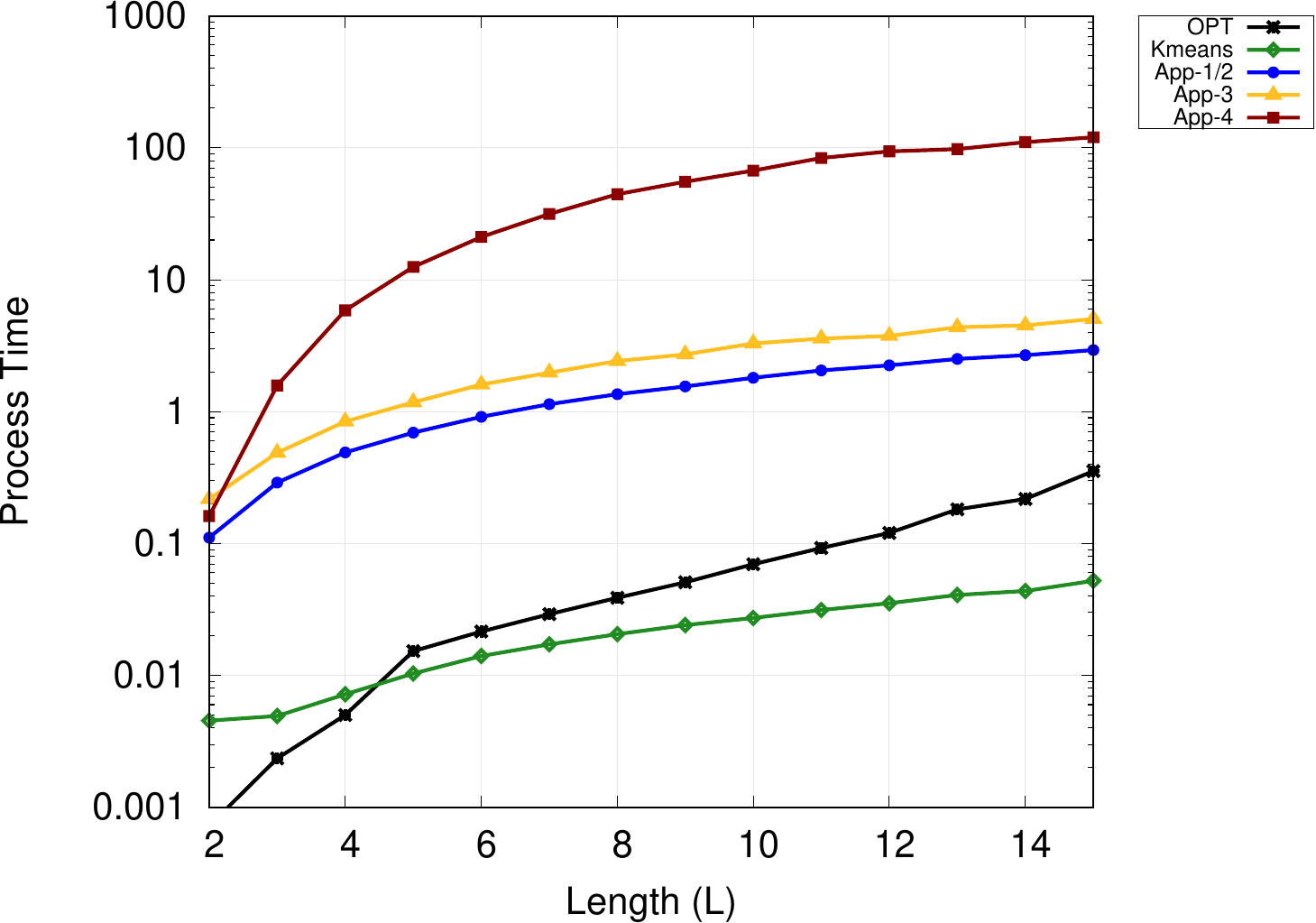}
         		\caption{Process time}\label{fig:app-shp-protime-l}
     		\end{subfigure}
        		\caption{Shortest path - time performance with fixed width ($W=3$)}\label{fig:app-shp-time-performance-l}
\end{figure}

Regarding time performance, all methods perform similarly for both settings. The iterative approximation (iv) has the highest aggregation time but the lowest solution time. Additionally, in Figures~\ref{fig:app-shp-protime-l} and \ref{fig:app-shp-protime-w}, the process time is calculated as the sum of both aggregation and solution times for all corresponding $s$-$t$ paths, divided by the number of $s$-$t$ paths. It also includes the time to solve the robust optimization model directly, using the original scenario set. We note that our approaches require more time when applied to a single instance than solving the problem directly. However, the motivation for our method is that it needs to be done only once, and can then be applied multiple times to a set of instances.

\begin{figure}
	\centering
		\begin{subfigure}{0.32\textwidth}
         		\centering
         			\includegraphics[width=\textwidth]{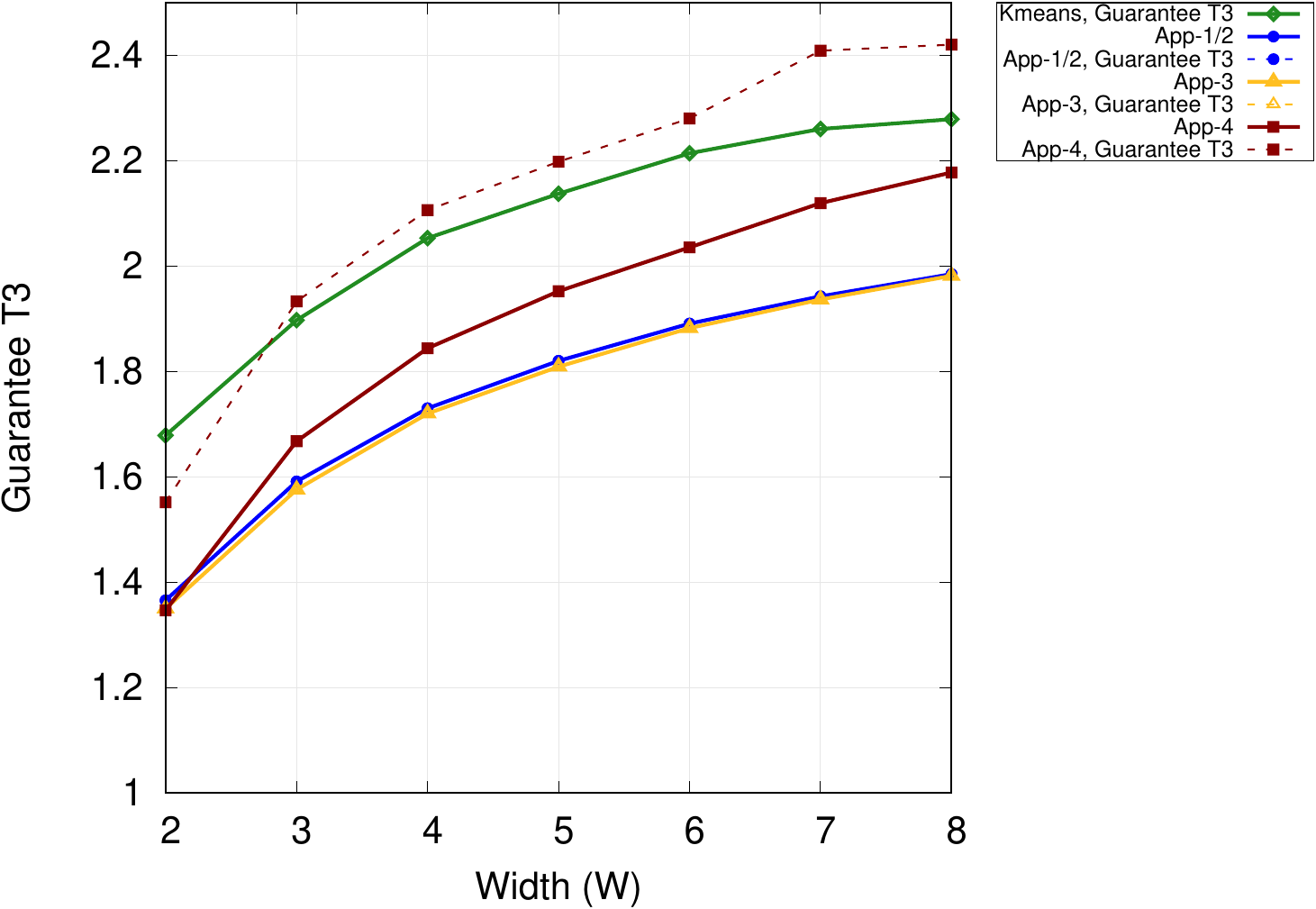}
         		\caption{Guarantee T3}\label{fig:app-shp-gua3-w}
     		\end{subfigure}
     		\begin{subfigure}{0.32\textwidth}
         		\centering
         			\includegraphics[width=\textwidth]{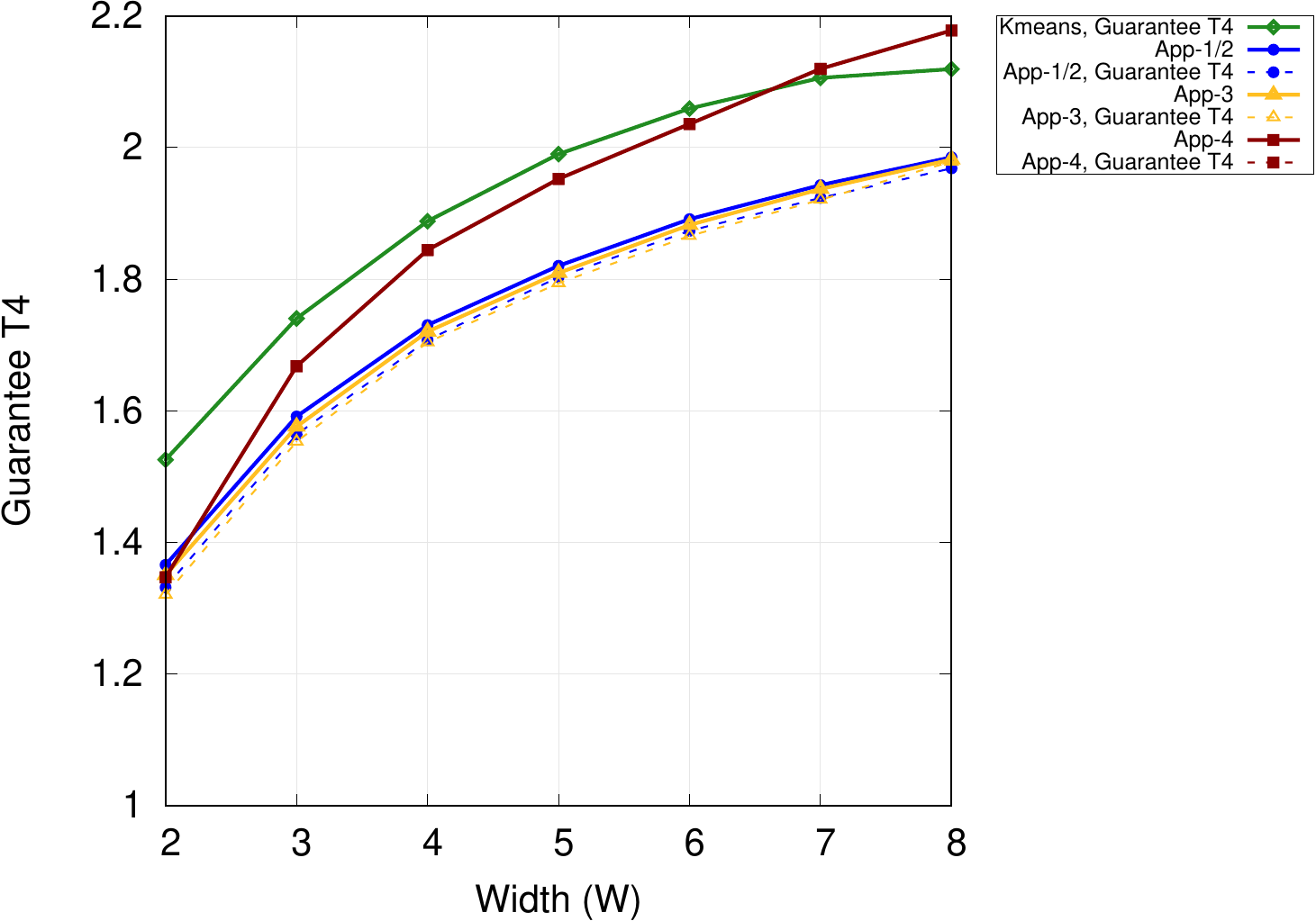}
         		\caption{Guarantee T4}\label{fig:app-shp-gua4-w}
     		\end{subfigure}
		\begin{subfigure}{0.32\textwidth}
         		\centering
         			\includegraphics[width=\textwidth]{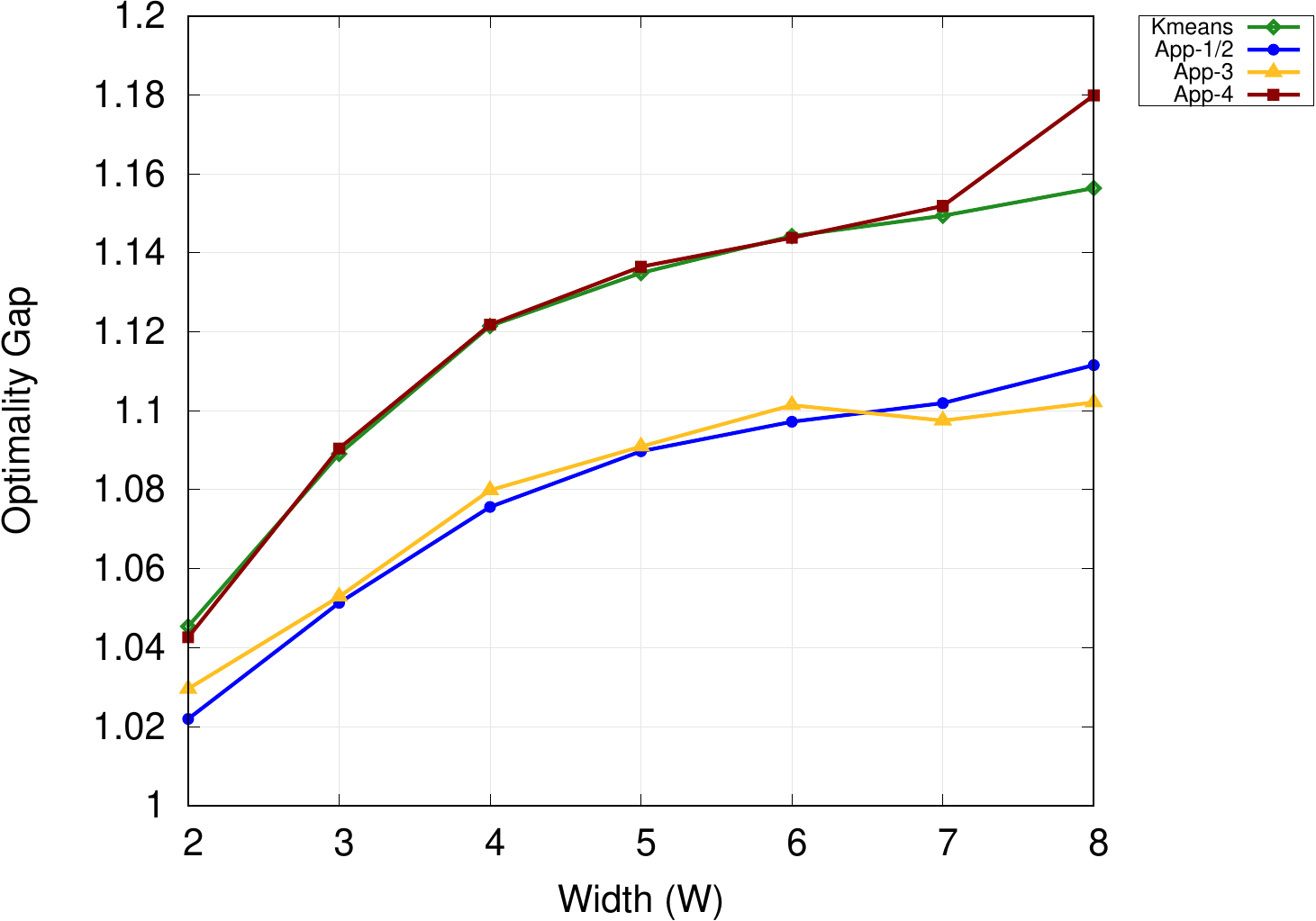}
         		\caption{UB Ratio}\label{fig:app-shp-ubratio-w}
     		\end{subfigure}
        		\caption{Shortest path - performance with fixed length ($L=4$)}\label{fig:app-shp-performance-w}
\end{figure}

 \begin{figure}
	\centering
		\begin{subfigure}{0.32\textwidth}
         		\centering
         			\includegraphics[width=\textwidth]{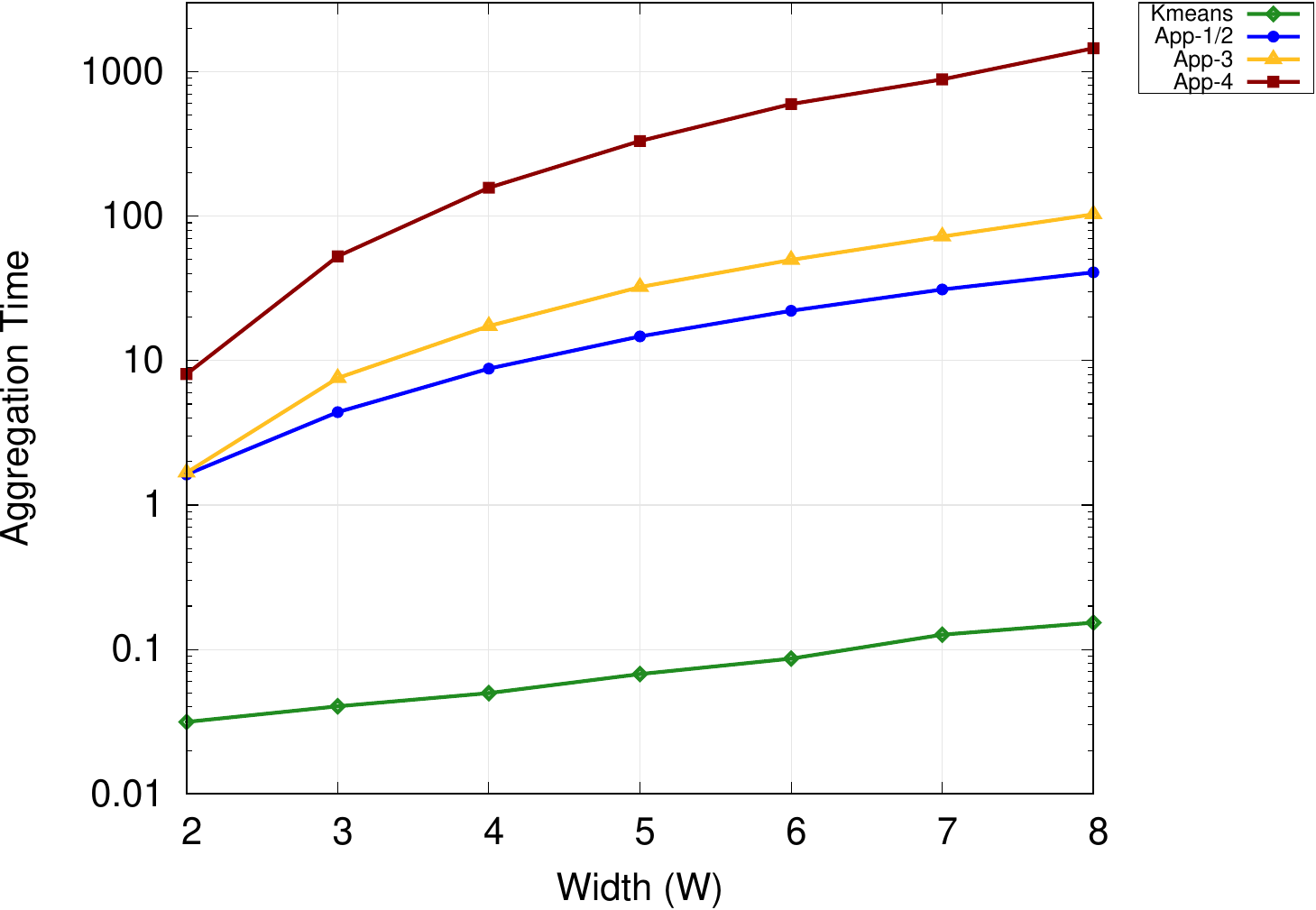}
         		\caption{Aggregation time}\label{fig:app-shp-aggtime-w}
     		\end{subfigure}
     		\begin{subfigure}{0.32\textwidth}
         		\centering
         			\includegraphics[width=\textwidth]{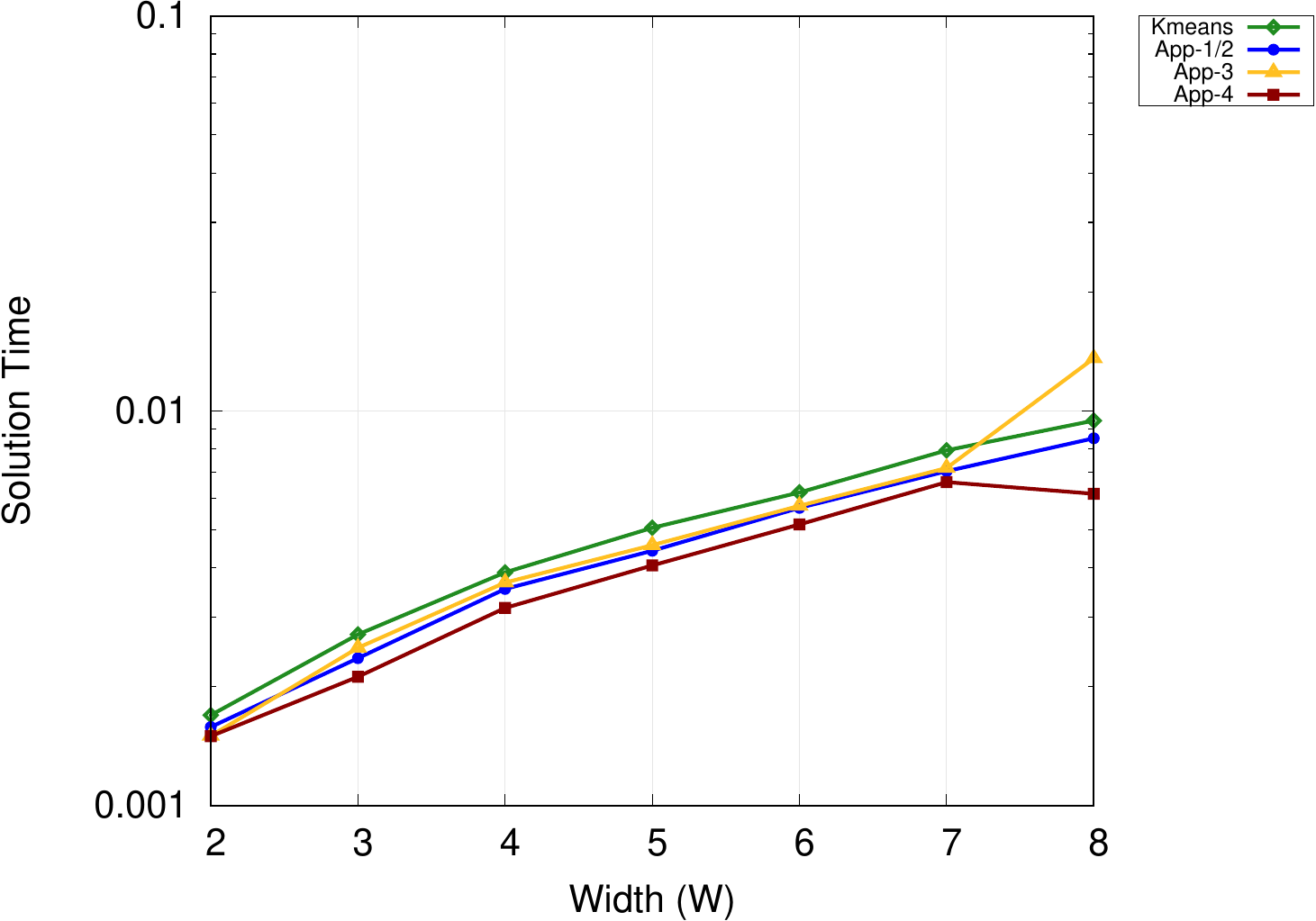}
         		\caption{Solution time}\label{ffig:app-shp-soltime-w}
     		\end{subfigure}
		\begin{subfigure}{0.32\textwidth}
         		\centering
         			\includegraphics[width=\textwidth]{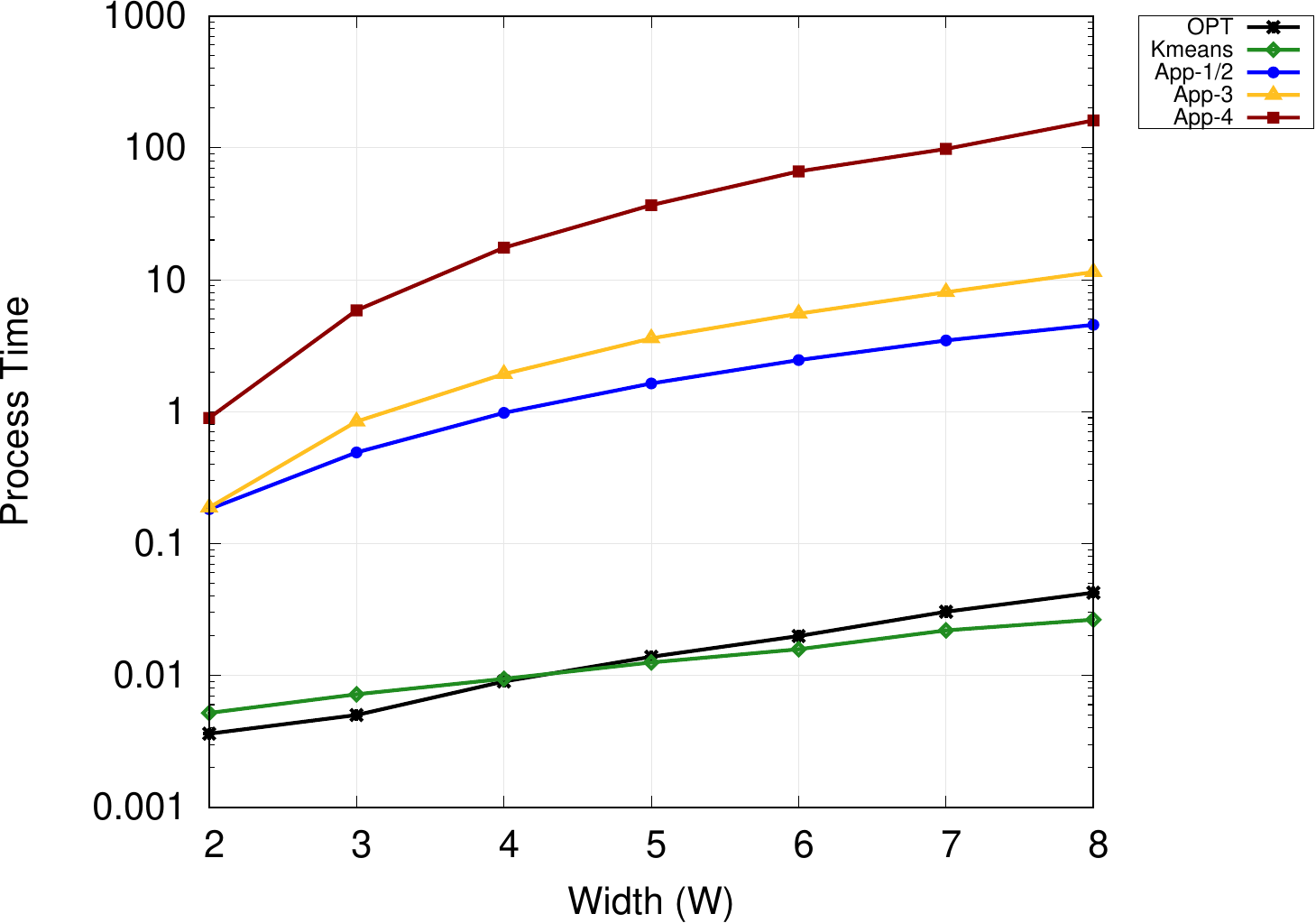}
         		\caption{Process time}\label{fig:app-shp-protime-w}
     		\end{subfigure}
        		\caption{Shortest path - time performance with fixed length ($L=4$)}\label{fig:app-shp-time-performance-w}
\end{figure}

\section{Conclusions}
\label{sec:conclusions}

Robust optimization problems tend to become harder with an increasing number of scenarios, both from a theoretical and practical point of view. Thus we have an interest in constructing uncertainty sets from discrete data observations with as few scenarios as possible. In this paper we studied the scenario reduction problem (which aims at removing unnecessary scenarios from the problem) and the scenario approximation problem (which aims at finding a representative set that results in good robust solutions). Our framework takes the set of feasible solutions into account, which guides the process of removing or representing scenarios. We presented a four types of sufficiency and approximation criteria and compared them theoretically. We discussed connections to conic orderings and introduced optimization models to find reduced or approximate uncertainty sets and guarantees. In our computational experiments using selection and shortest path instances, we showed the potential of our approach in reducing the size of uncertainty sets, in particular in comparison to a simple clustering method or when we ignore the set of feasible solutions in the process.

The framework presented in this paper takes a discrete uncertainty and represents it by another discrete uncertainty set. In further research we will consider more general representations that make use of different types of uncertainty sets. For example, using the resulting approximation guarantee as a guiding principle, we will study how to represent discrete observations using budgeted or ellipsoidal uncertainty sets.

\newcommand{\etalchar}[1]{$^{#1}$}

\newpage

\appendix

\section{Extensions of Domination Types}\label{sec:appendix}

We show how the dominance criteria discussed in Section~\ref{sec:reduction} extend to non-linear objective functions $f$. For uncertainty sets $\cU' \subseteq \cU = \{\pmb{c}^1,\ldots,\pmb{c}^N\}\subseteq \mathbb{R}^n_+$ and a set of feasible solutions $\X\subseteq\mathbb{R}^n_+$, we define:
\begin{itemize}
\item sufficiency type (i):
\[ \forall \pmb{c}\in \cU \ \exists \pmb{c}'\in\cU' : \pmb{c}' \ge \pmb{c} \]

\item sufficiency type (ii):
\[ \forall \pmb{c}\in\cU, \pmb{\alpha}\in\mathbb{R}^n_+\ \exists \pmb{c}'\in\cU' : \pmb{\alpha}^\top \pmb{c}' \ge \pmb{\alpha}^\top \pmb{c} \]

\item generalized sufficiency type (iii):
\[ \forall \pmb{c}\in\cU\ \exists \pmb{c}'\in\cU' \ \forall \pmb{x}\in\X : f(\pmb{x},\pmb{c}') \ge f(\pmb{x},\pmb{c}) \]

\item generalized sufficiency type (iv):
\[ \forall \pmb{c}\in\cU, \pmb{x}\in\X\ \exists \pmb{c}'\in\cU' : f(\pmb{x},\pmb{c}') \ge f(\pmb{x},\pmb{c})  \]

\end{itemize}

\begin{theorem}
Let $\cU'\subseteq \cU$. Set $\cU'$ is a sufficient set if one of the following conditions holds:
\begin{itemize}
\item set $\cU'$ fulfills sufficiency type (i), and $f(\pmb{x},\pmb{c})$ is monotonically increasing in $\pmb{c}$ for all $\pmb{x}$,
\item set $\cU'$ fulfills sufficiency type (ii), and $f(\pmb{x},\pmb{c})$ is monotonically increasing and concave in $\pmb{c}$ for all $\pmb{x}$,
\item set $\cU'$ fulfills one of the generalized sufficiency types (iii) or (iv).
\end{itemize}
\end{theorem}
\begin{proof}
Let some $\pmb{x}\in\X$ be given, and let $\pmb{c}(\pmb{x})$ be some maximizer of $\max_{\pmb{c}\in\cU} f(\pmb{x},\pmb{c})$.
\begin{itemize}
\item Let $\cU'$ fulfill sufficiency type (i). By definition, there is some $\pmb{c}'\in\cU'$ such that $\pmb{c}' \ge \pmb{c}(\pmb{x})$. By monotonicity of $f$, we have
\[ \max_{\pmb{c}\in\cU'} f(\pmb{x},\pmb{c}) \ge f(\pmb{x},\pmb{c}') \ge f(\pmb{x},\pmb{c}(\pmb{x})) = \max_{\pmb{c}\in\cU} f(\pmb{x},\pmb{c}). \]
\item Let $\cU'$ fulfill sufficiency type (ii). As $f$ is concave in $\pmb{c}$, we can assume that $\pmb{c}(\pmb{x})$ is an extreme point of the polyhedron $\text{conv}(\cU)$. Due to the monotonicity, $\pmb{c}(\pmb{x})$ is also efficient. The efficient extreme points of $\text{conv}(\cU)$ are exactly the supported efficient points of $\cU$. As $\cU'$ contains all supported efficient solutions, we therefore have $\pmb{c}(\pmb{x})\in\cU'$.

\item The proofs for generalized sufficiency of types (iii) and (iv) work in the same way as for the proofs of Theorem~\ref{th:sufficient}.
\end{itemize}
\end{proof}

\section{Additional Experimental Results}\label{sec:expappendix}

In this section we collect the remaining approximation experiments for the selection problem. In Figure~\ref{appfig:sel-performance-small}, performance (guarantee and actual ratio) of small instances with $n=8$ and $n=15$ are presented and Figure~\ref{appfig:sel-time-performance-small} shows their respective time performance. Moreover, the performance and time evaluation of the remaining large-scale experiments with $n\in\{20,30,100\}$ for the selection problem are provided in Figure~\ref{appfig:sel-performance-larg} and Figure~\ref{appfig:sel-time-performance-large}, respectively.

\begin{figure}
	\centering
		\begin{subfigure}{0.45\textwidth}
         		\centering
         			\includegraphics[width=\textwidth]{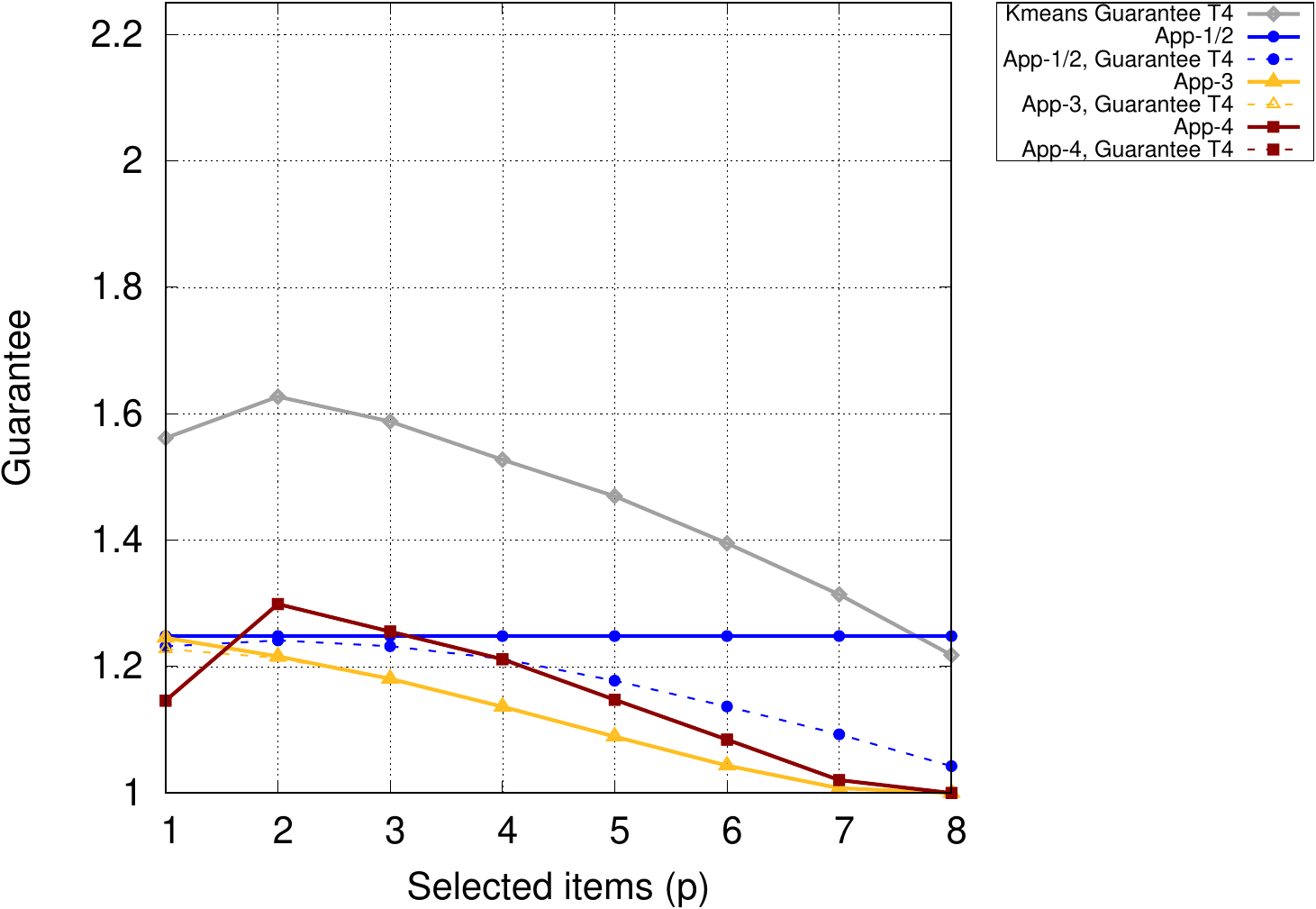}
         		\caption{Guarantee ($n=8$)}\label{fig:guarantee-n8}
		\end{subfigure}
     		\begin{subfigure}{0.45\textwidth}
         		\centering
         			\includegraphics[width=\textwidth]{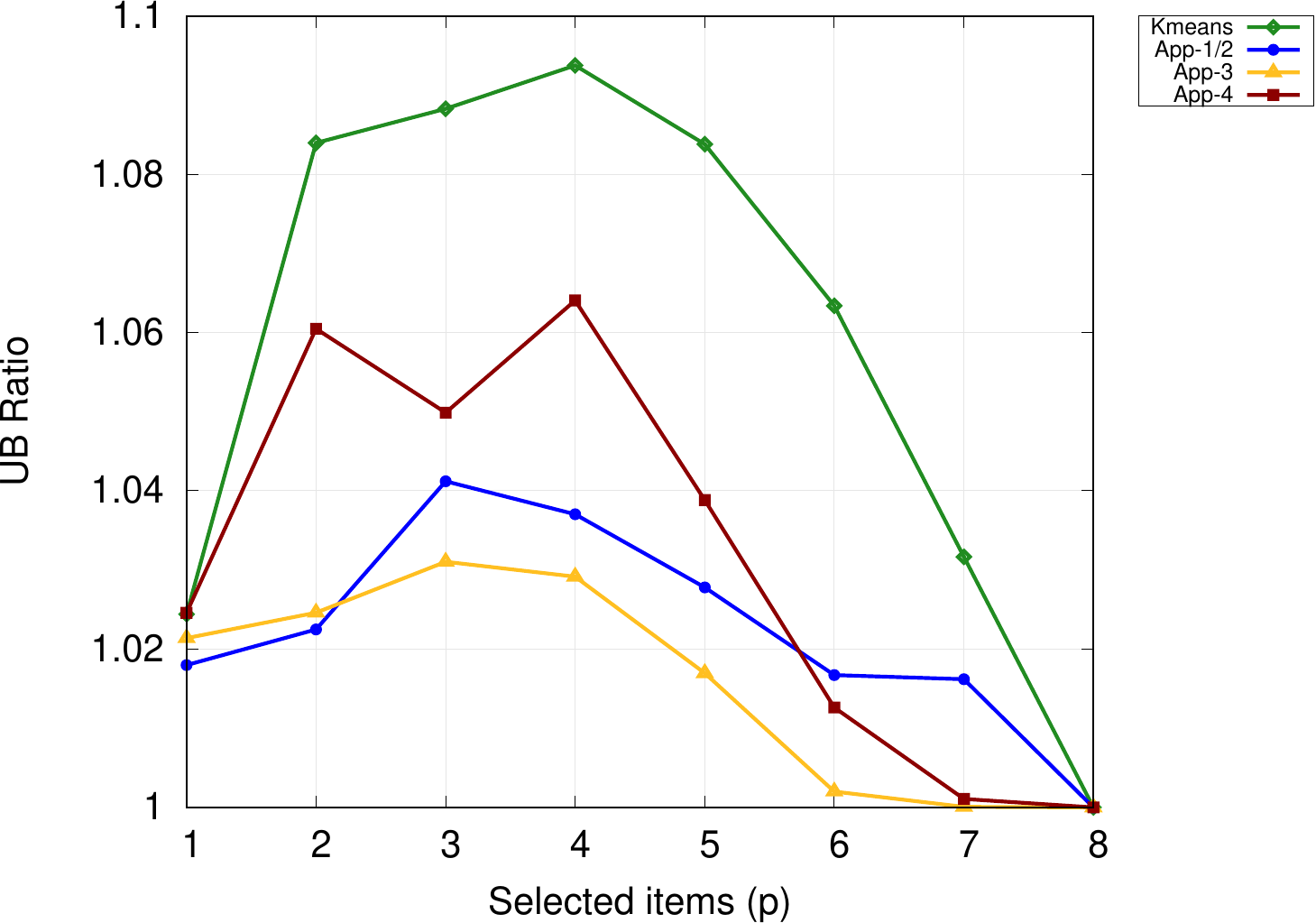}
         		\caption{UB Ratio ($n=8$)}\label{fig:optgap-n8}
     		\end{subfigure}\\
		\begin{subfigure}{0.45\textwidth}
         		\centering
         			\includegraphics[width=\textwidth]{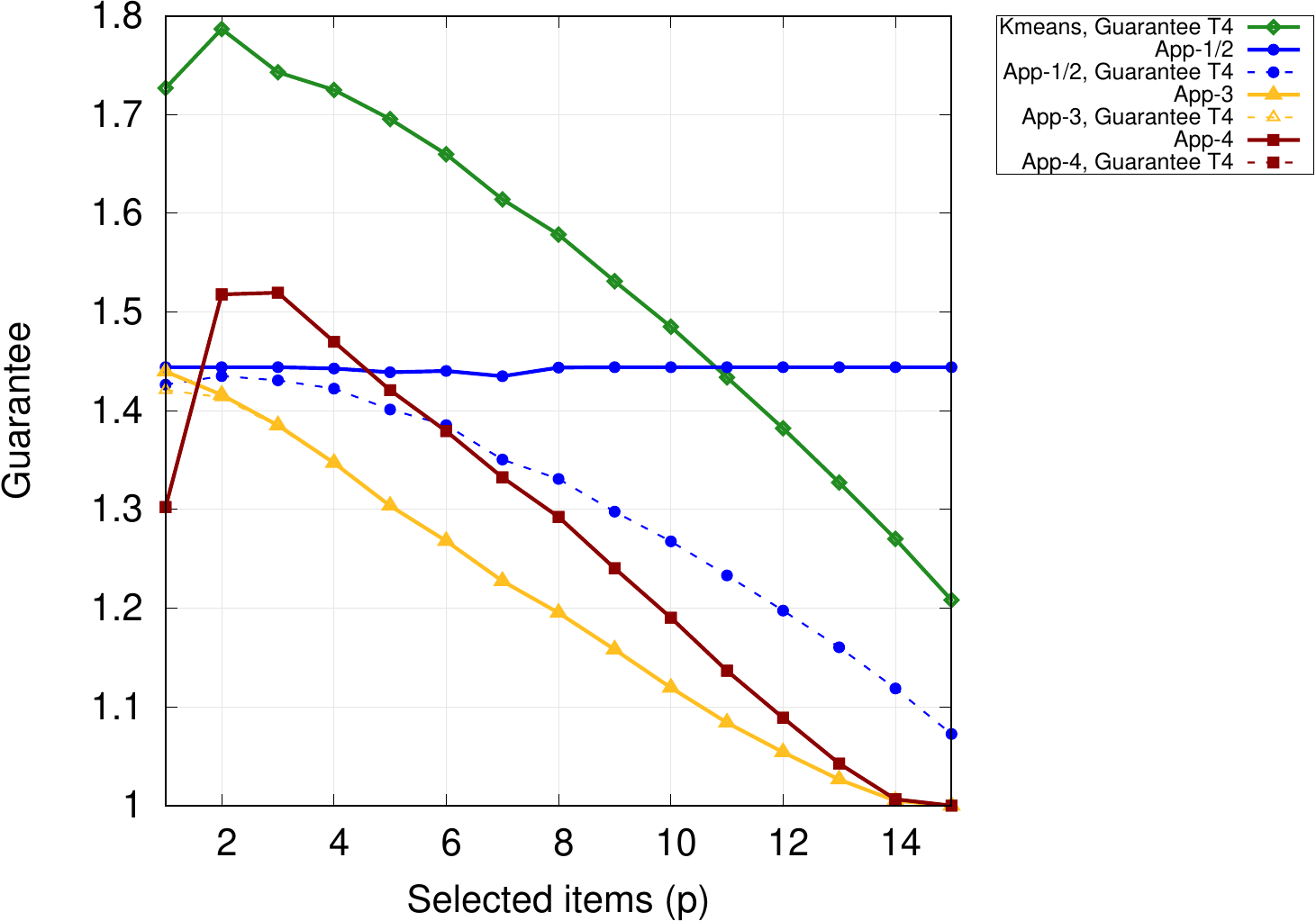}
         		\caption{Guarantee ($n=15$)}\label{fig:guarantee-n15}
		\end{subfigure}
     		\begin{subfigure}{0.45\textwidth}
         		\centering
         			\includegraphics[width=\textwidth]{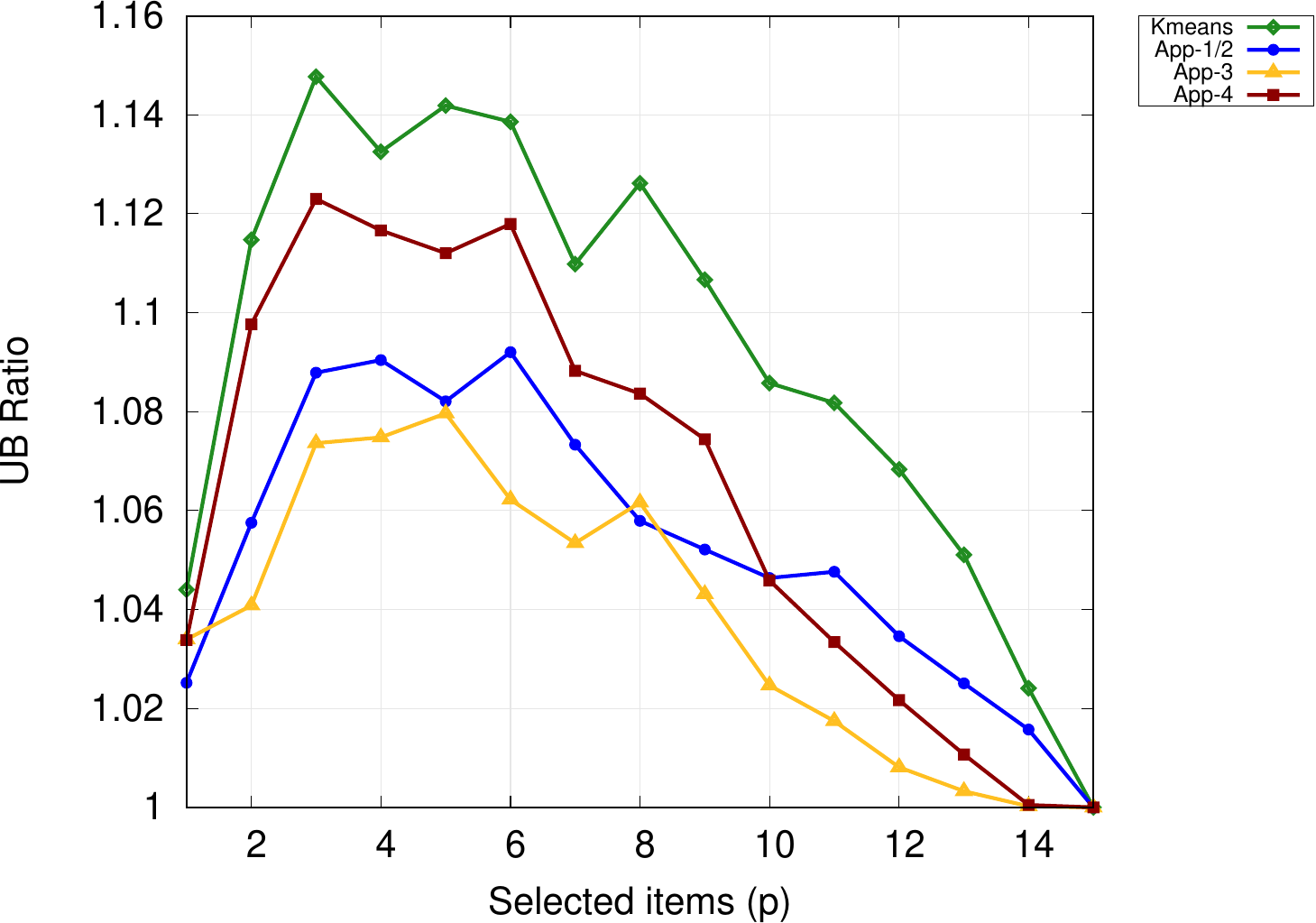}
         		\caption{UB Ratio ($n=15$)}\label{fig:optgap-n15}
     		\end{subfigure}
        		\caption{Selection - performance of small instances}\label{appfig:sel-performance-small}
\end{figure}

\begin{figure}
	\centering
		\begin{subfigure}{0.32\textwidth}
         		\centering
         			\includegraphics[width=\textwidth]{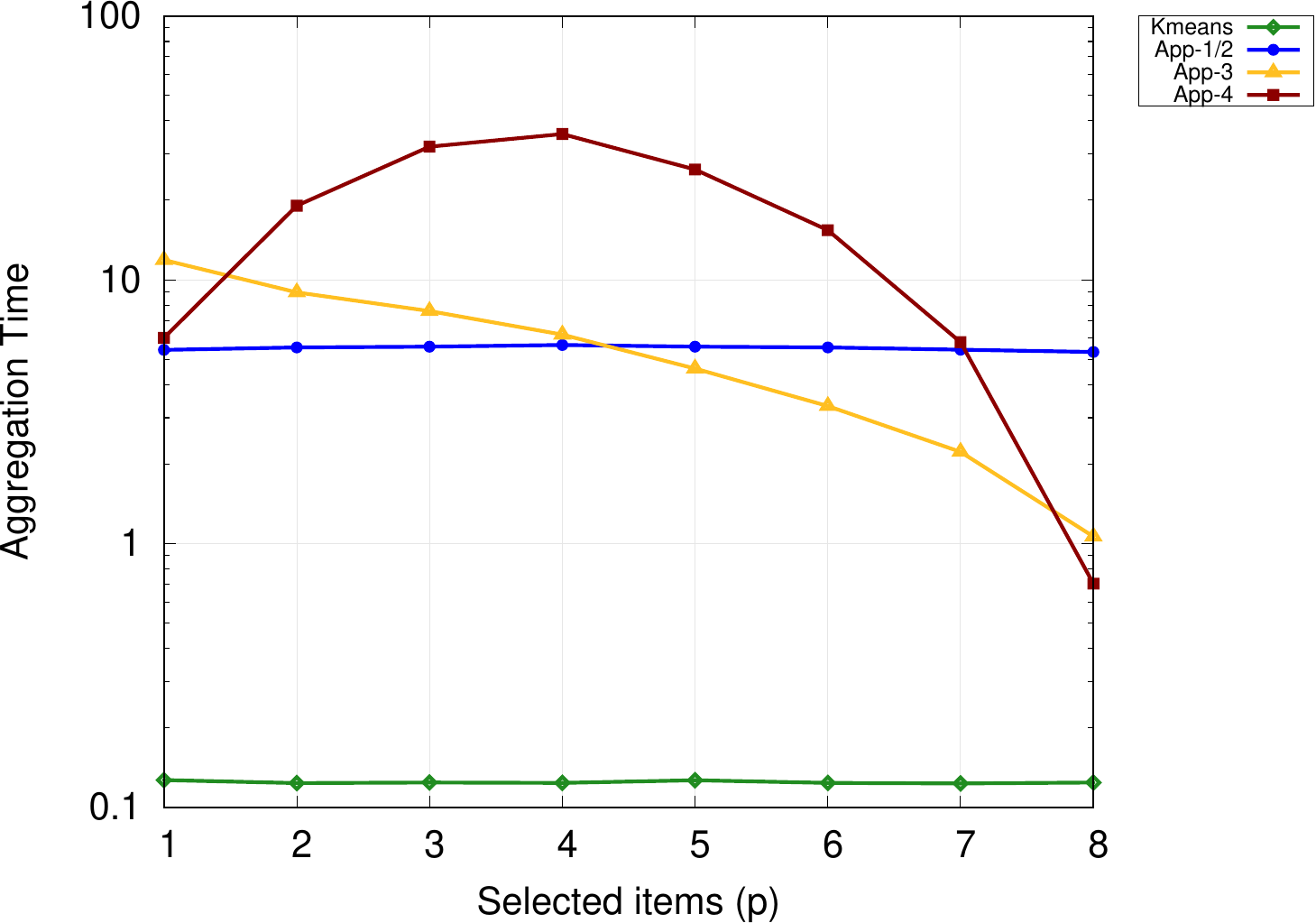}
         		\caption{Aggregation}\label{fig:aggtime,n8}
     		\end{subfigure}
     		\begin{subfigure}{0.32\textwidth}
         		\centering
         			\includegraphics[width=\textwidth]{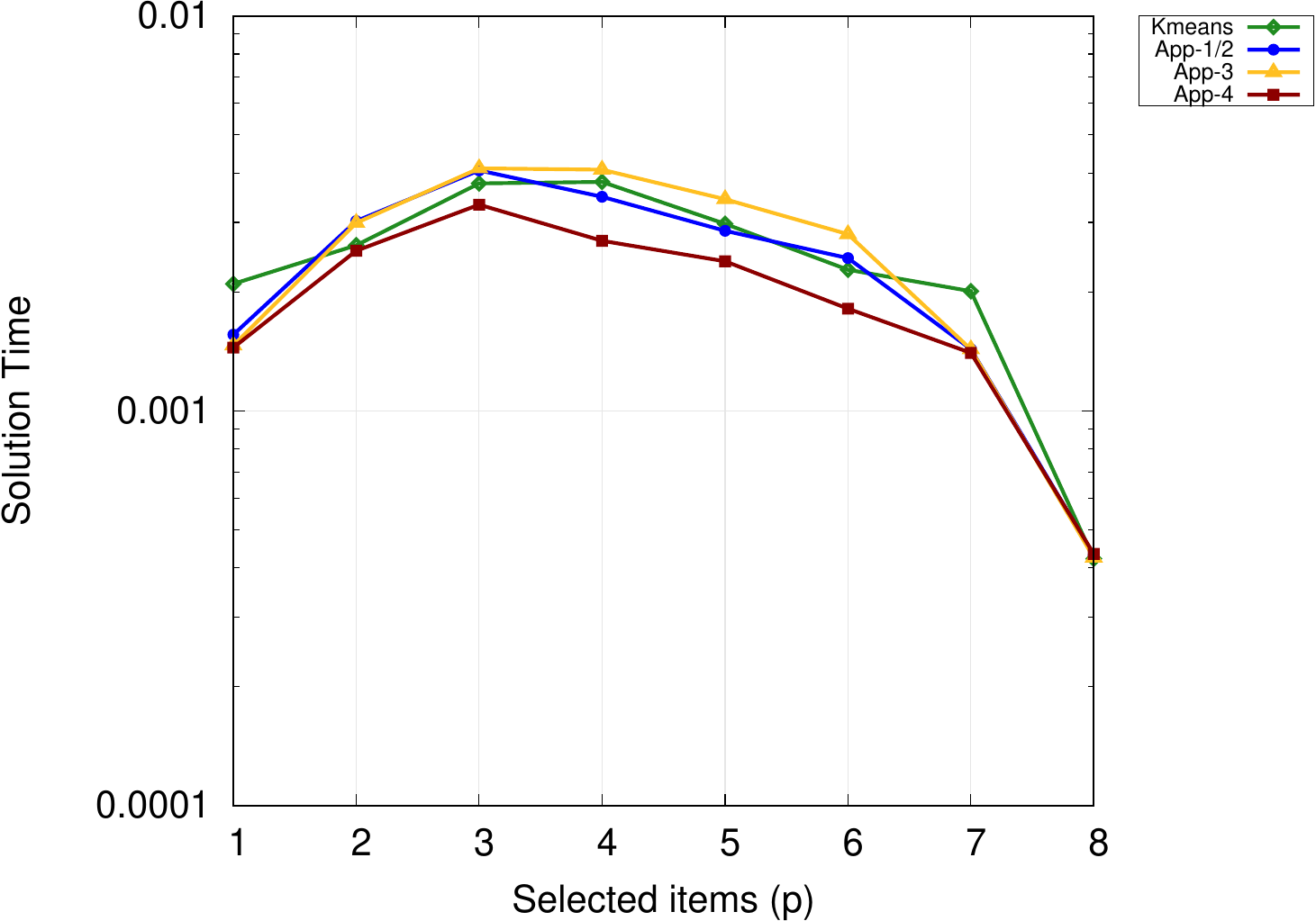}
         		\caption{Solution}\label{fig:soltime,n8}
     		\end{subfigure}
		\begin{subfigure}{0.32\textwidth}
         		\centering
         			\includegraphics[width=\textwidth]{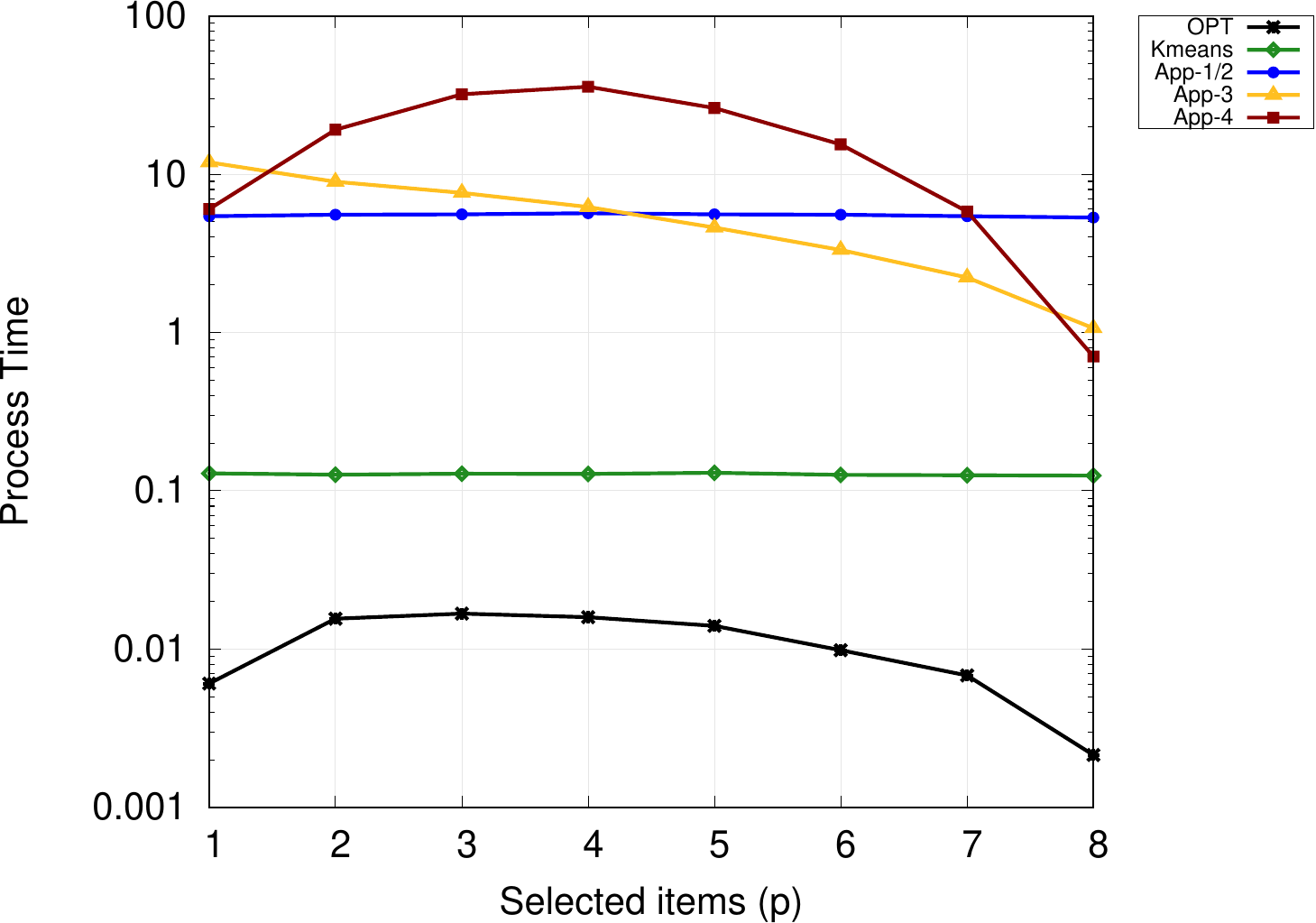}
         		\caption{Process}\label{fig:protime,n8}
     		\end{subfigure}\\
		\begin{subfigure}{0.32\textwidth}
         		\centering
         			\includegraphics[width=\textwidth]{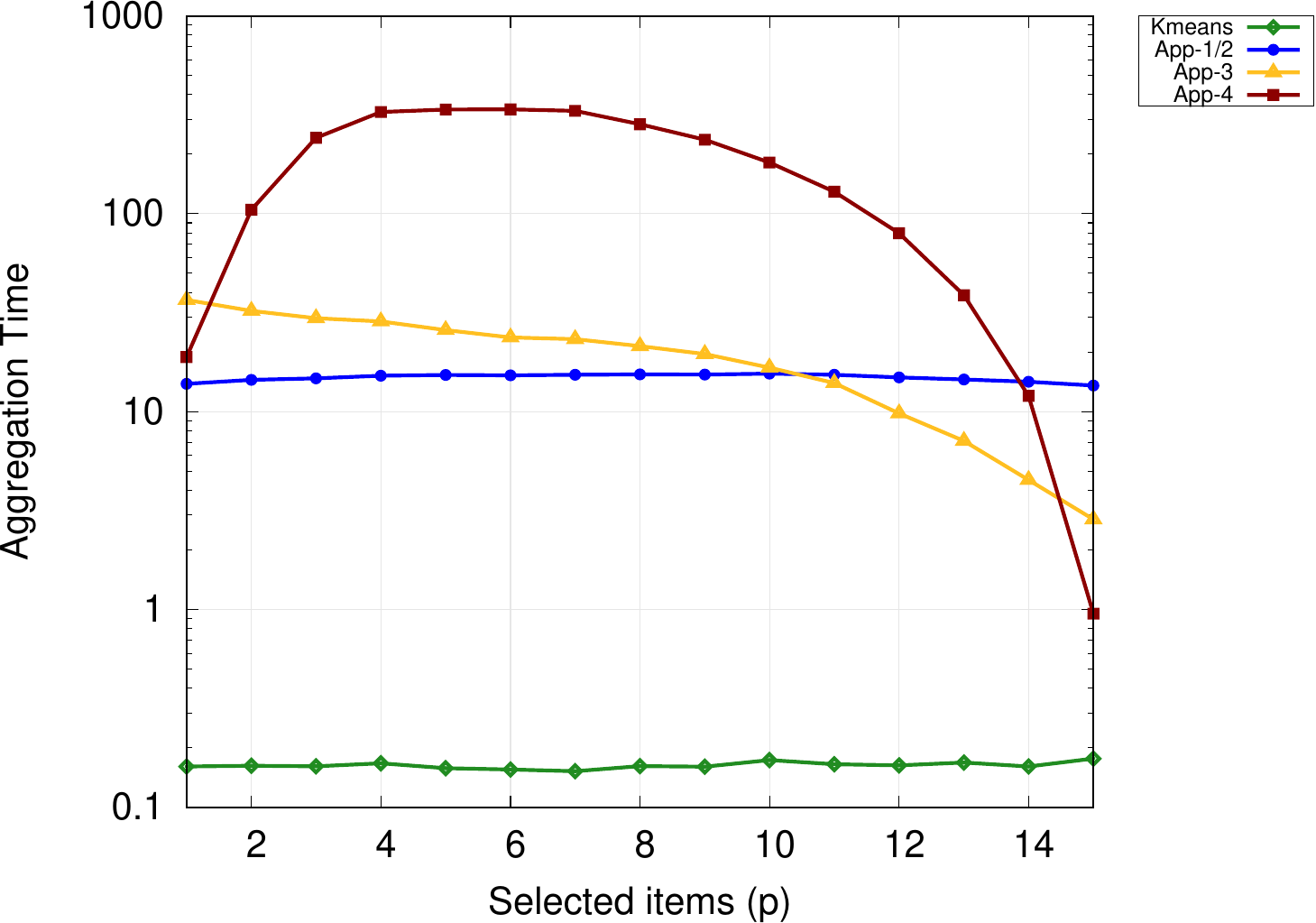}
         		\caption{Aggregation Time}\label{fig:aggtime-n15}
     		\end{subfigure}
     		\begin{subfigure}{0.32\textwidth}
         		\centering
         			\includegraphics[width=\textwidth]{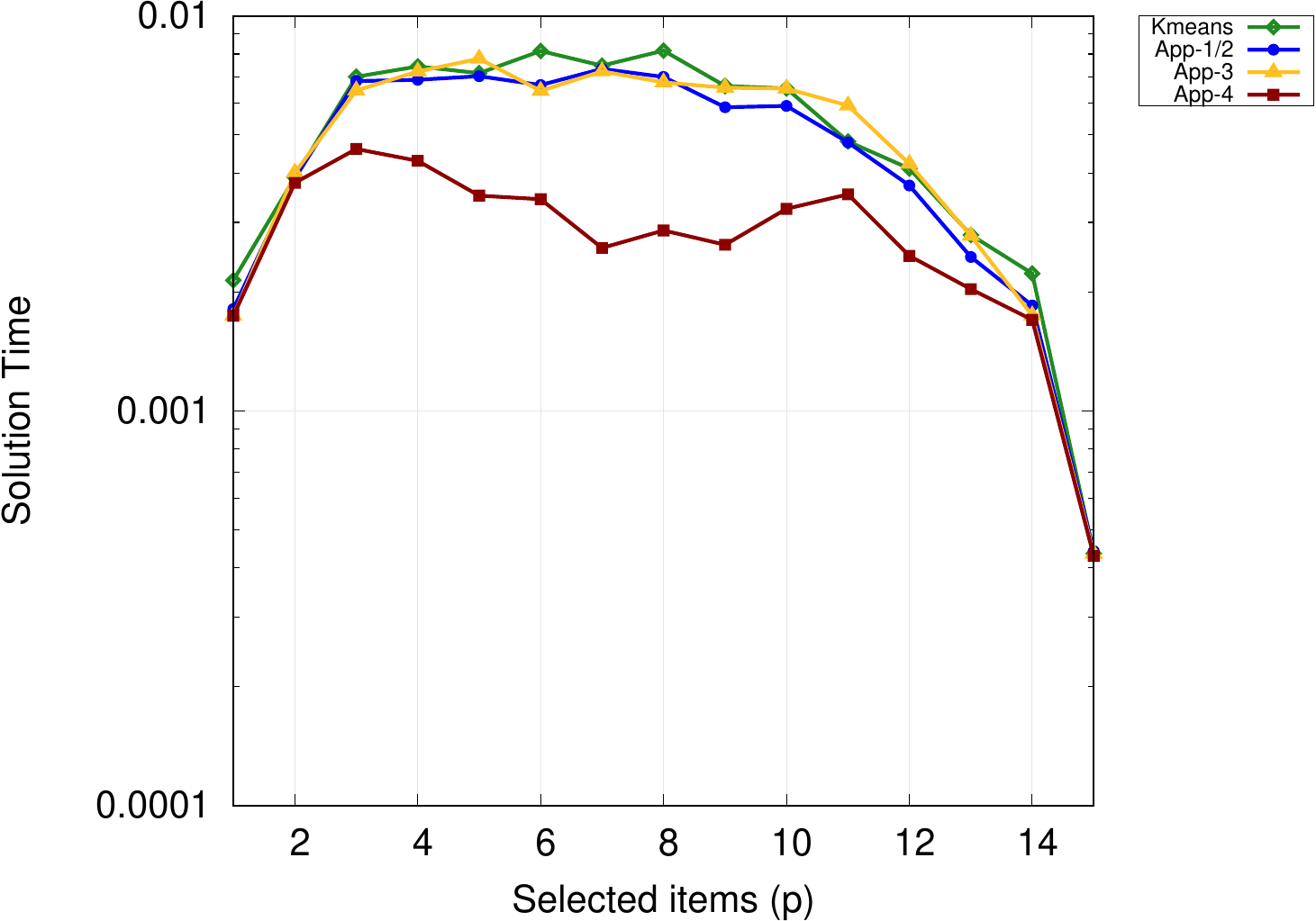}
         		\caption{Solution Time}\label{fig:soltime-n15}
     		\end{subfigure}
		\begin{subfigure}{0.32\textwidth}
         		\centering
         			\includegraphics[width=\textwidth]{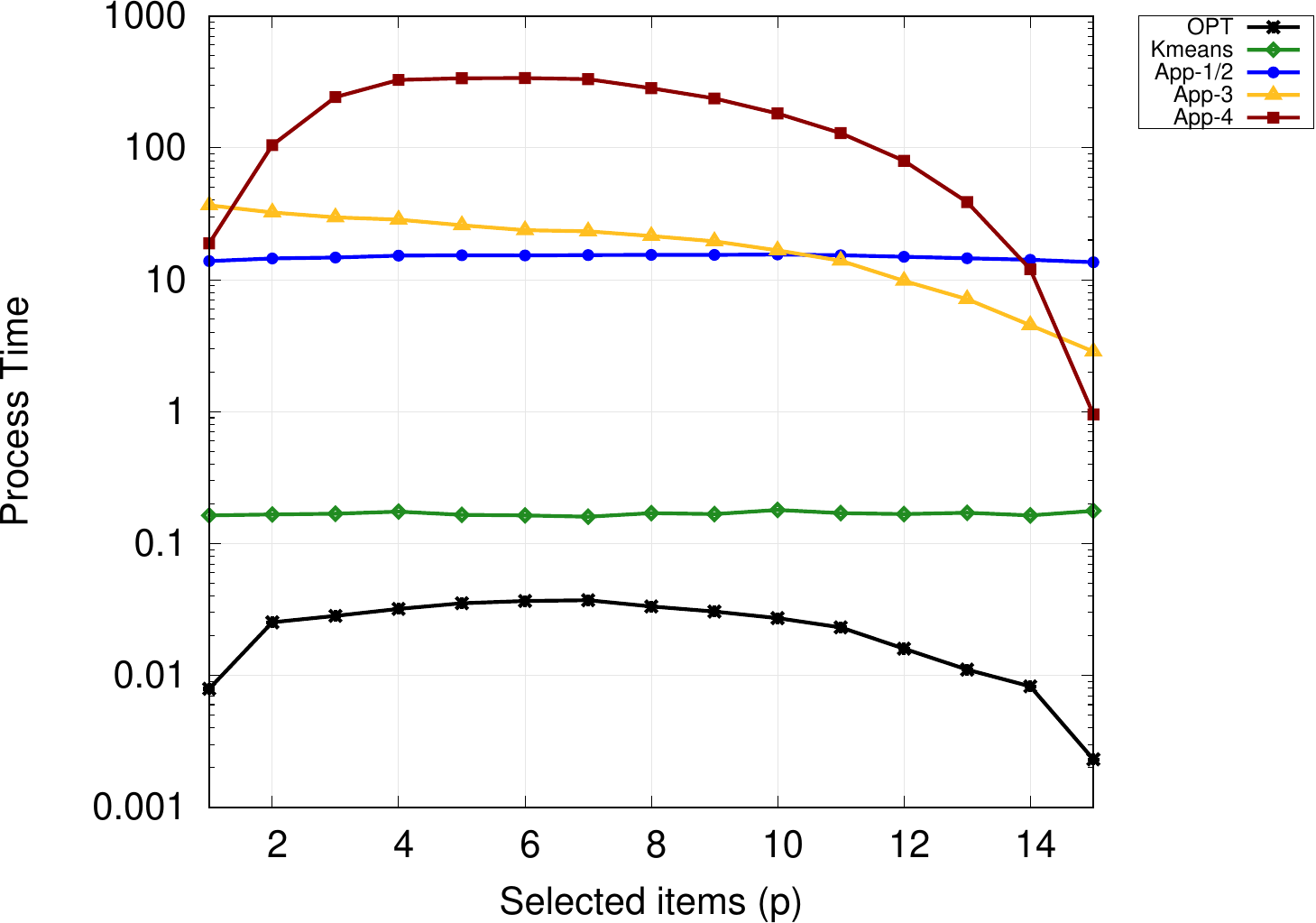}
         		\caption{Process}\label{fig:protime,n15}
     		\end{subfigure}
        		\caption{Selection - time performance of small instances}\label{appfig:sel-time-performance-small}
\end{figure}

\begin{figure}
	\centering
		\begin{subfigure}{0.45\textwidth}
         		\centering
         			\includegraphics[width=\textwidth]{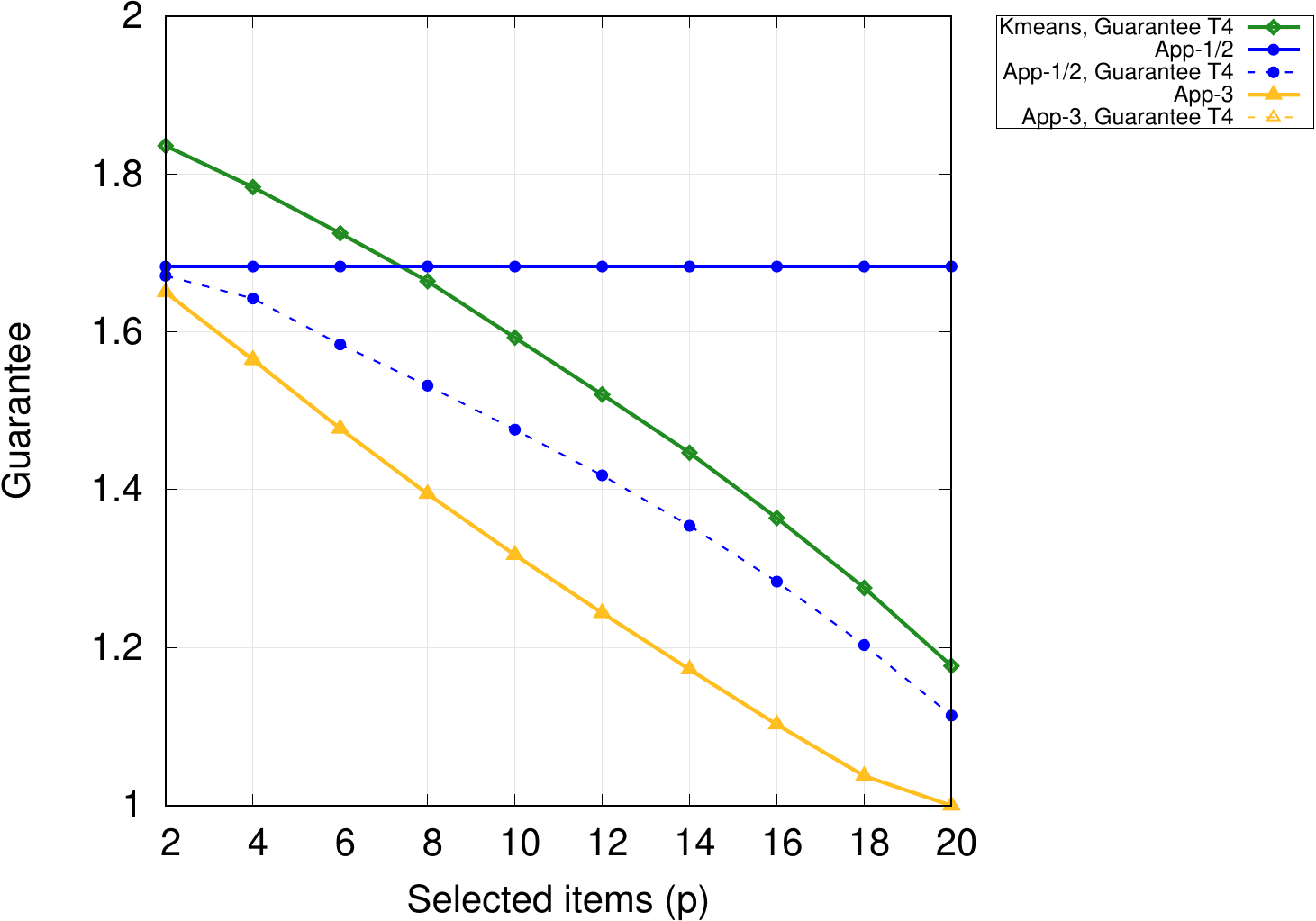}
         		\caption{Guarantee ($n=20$)}\label{fig:guarantee-t4-n20}
		\end{subfigure}
     		\begin{subfigure}{0.45\textwidth}
         		\centering
         			\includegraphics[width=\textwidth]{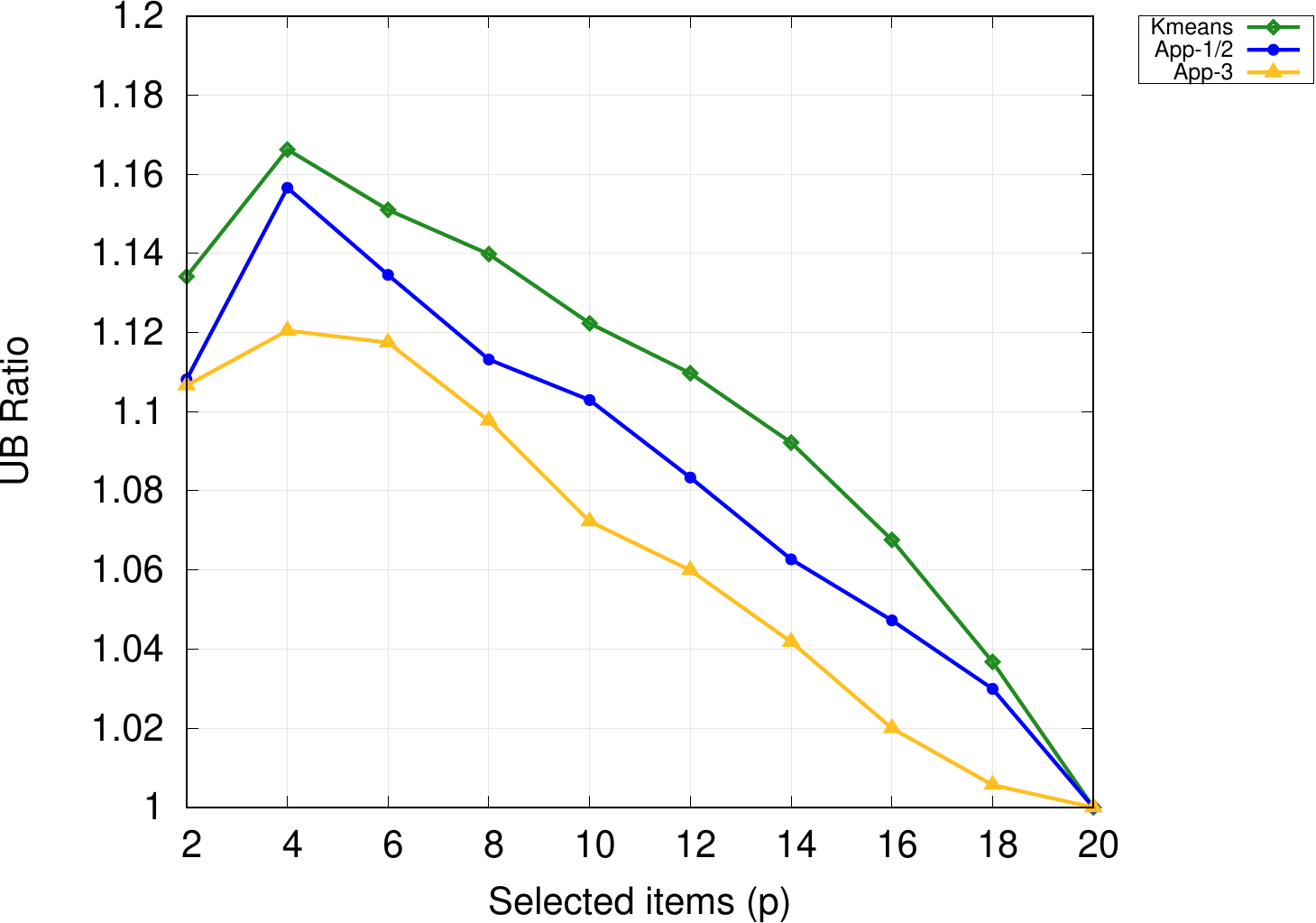}
         		\caption{UB Ratio ($n=20$)}\label{fig:ubratio-n20}
     		\end{subfigure}\\
		\begin{subfigure}{0.45\textwidth}
         		\centering
         			\includegraphics[width=\textwidth]{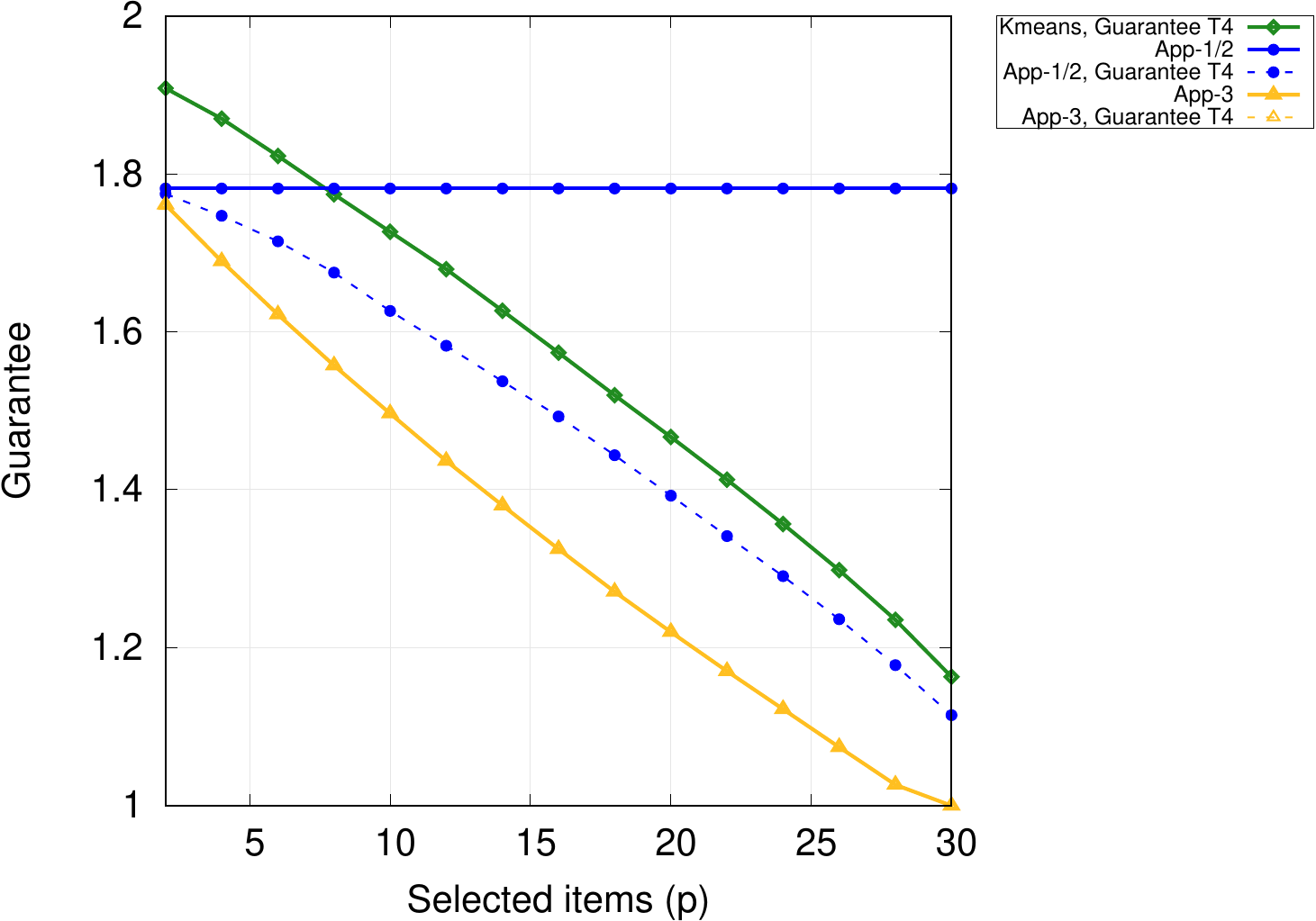}
         		\caption{Guarantee ($n=30$)}\label{fig:guarantee-t4-30}
		\end{subfigure}
     		\begin{subfigure}{0.45\textwidth}
         		\centering
         			\includegraphics[width=\textwidth]{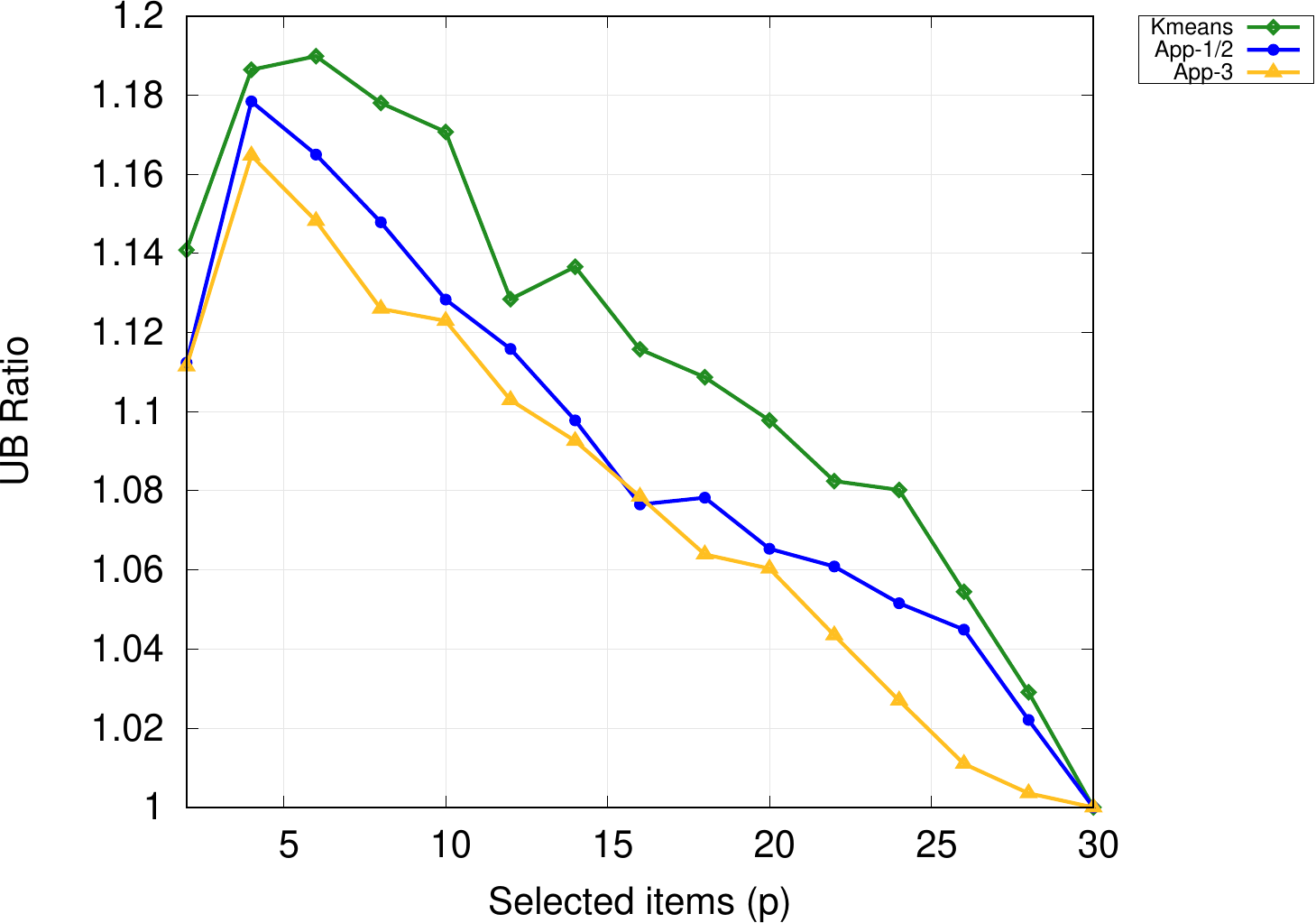}
         		\caption{UB Ratio ($n=30$)}\label{fig:ubratio-30}
     		\end{subfigure}\\
		\begin{subfigure}{0.45\textwidth}
         		\centering
         			\includegraphics[width=\textwidth]{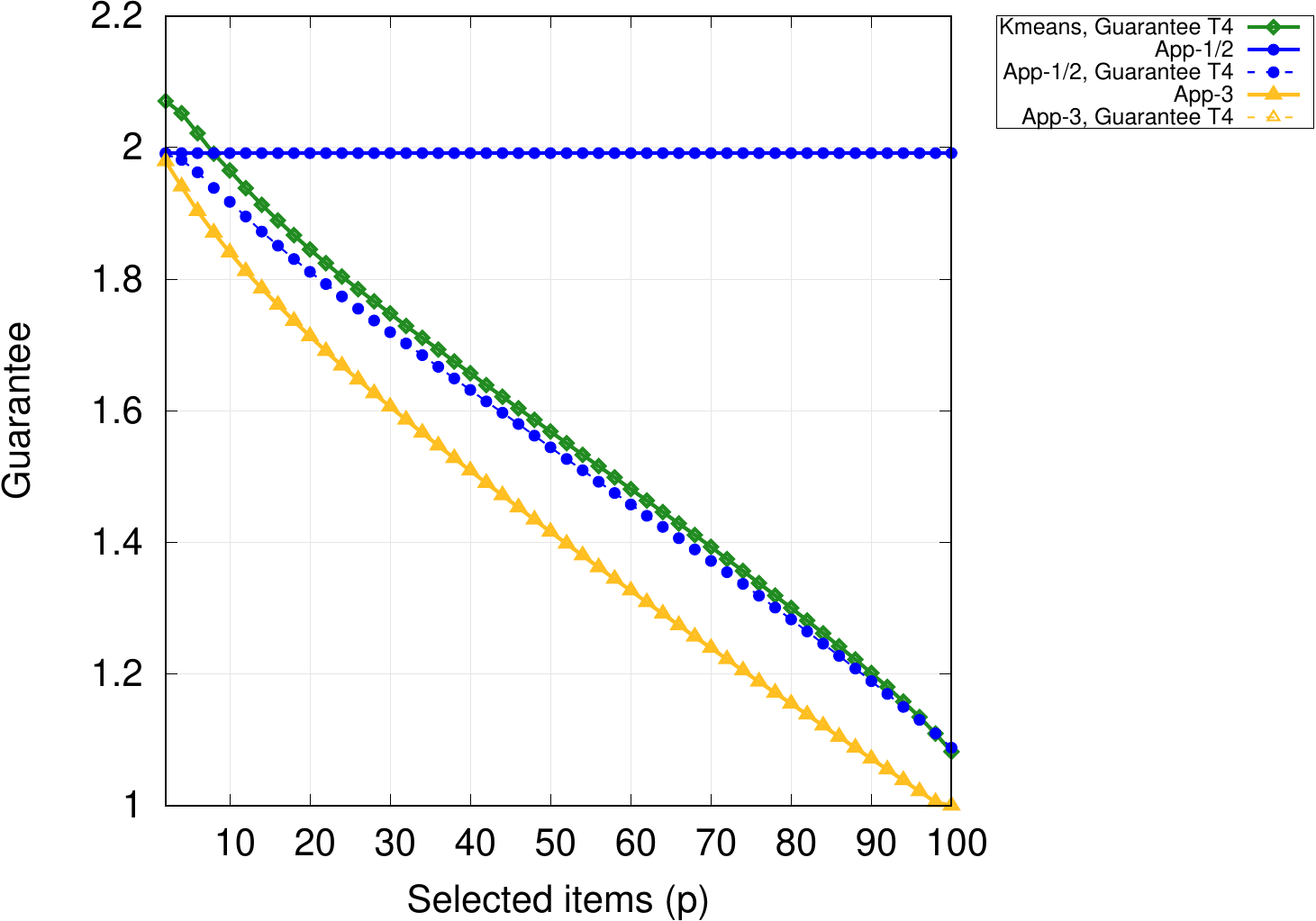}
         		\caption{Guarantee ($n=100$)}\label{fig:guarantee-t4-100}
		\end{subfigure}
     		\begin{subfigure}{0.45\textwidth}
         		\centering
         			\includegraphics[width=\textwidth]{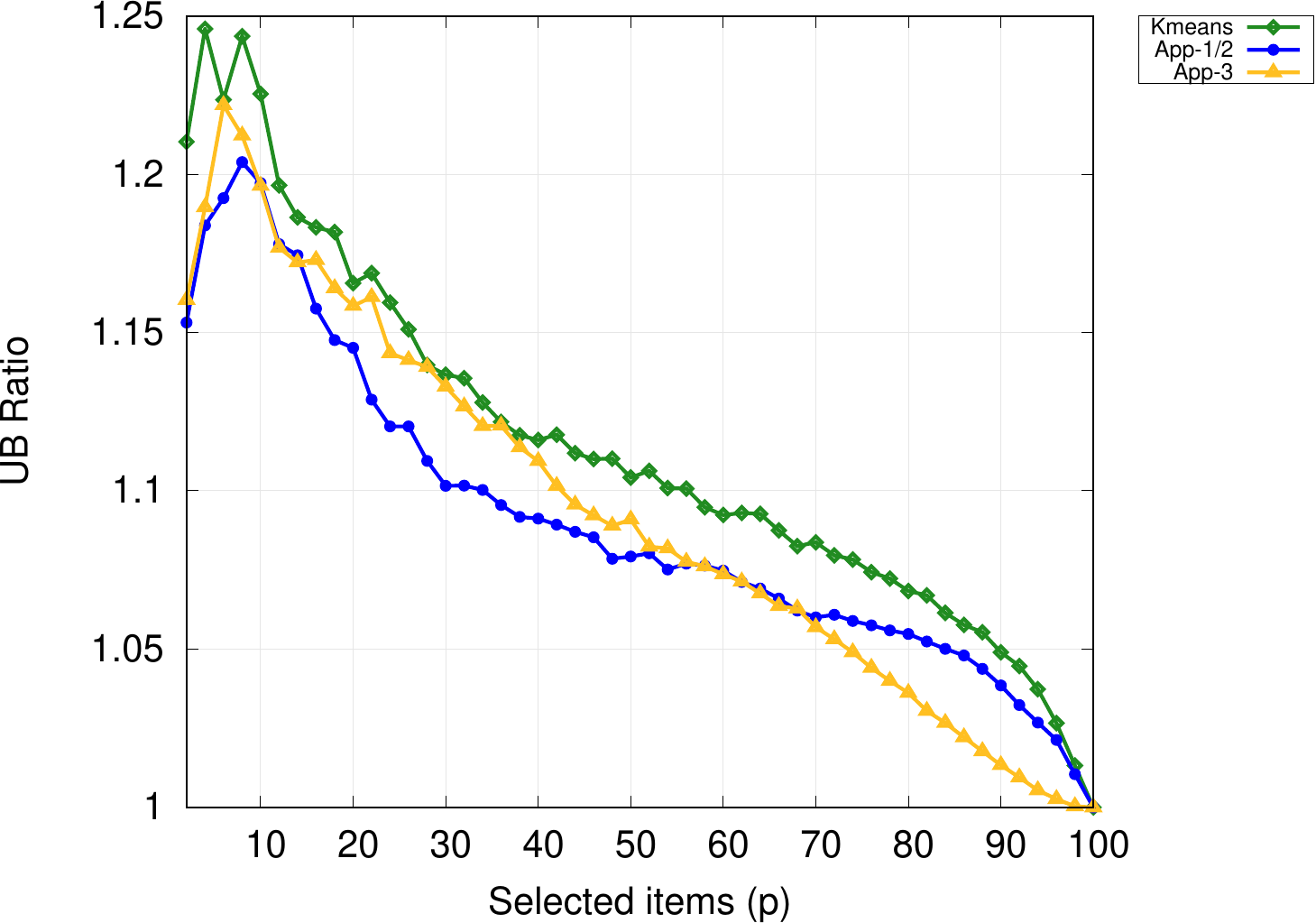}
         		\caption{UB Ratio ($n=100$)}\label{fig:ubratio-100}
     		\end{subfigure}
        		\caption{Selection - performance of large instances}\label{appfig:sel-performance-larg}
\end{figure}

\begin{figure}
	\centering
		\begin{subfigure}{0.32\textwidth}
         		\centering
         			\includegraphics[width=\textwidth]{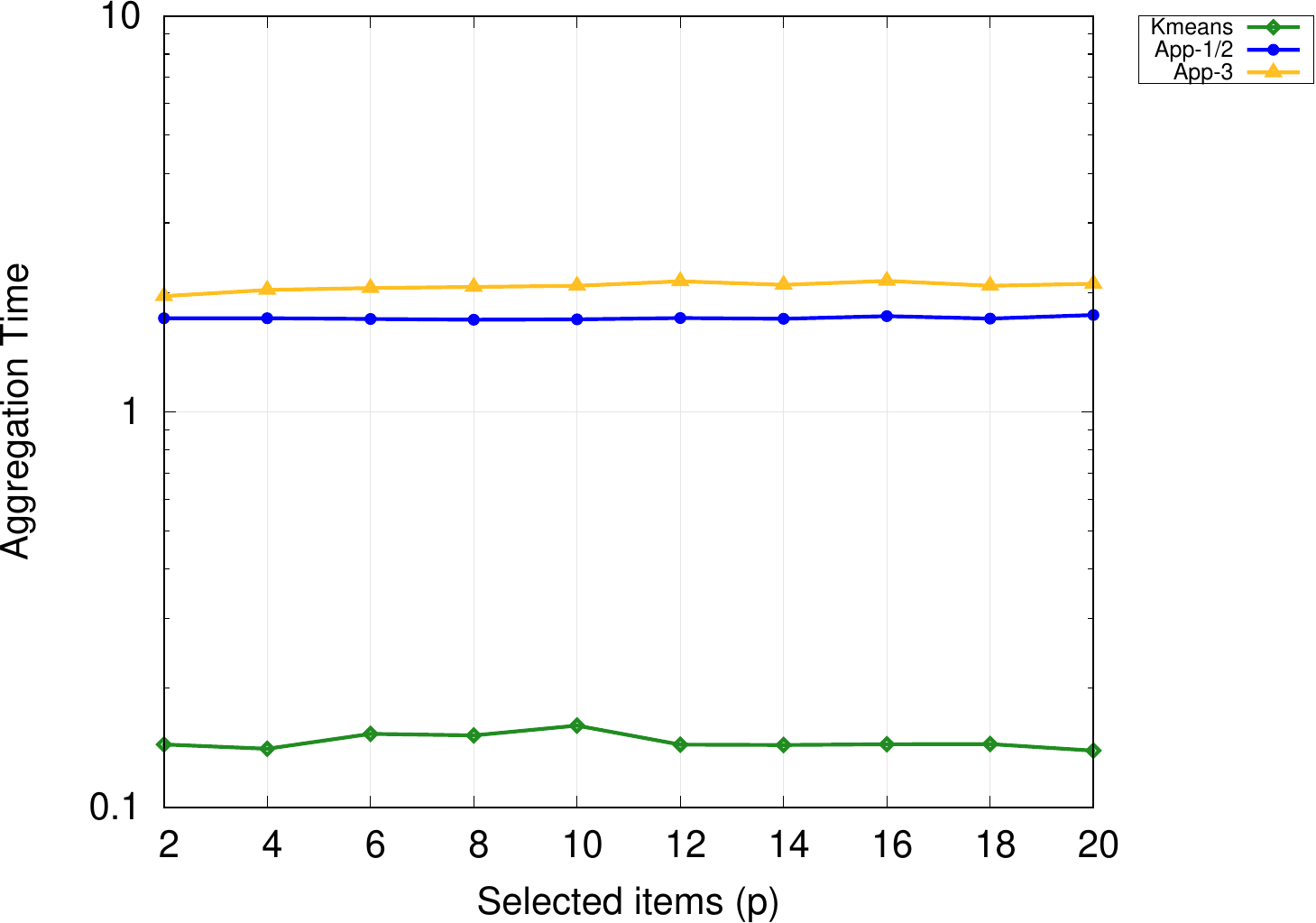}
         		\caption{Aggregation ($n=20$)}\label{fig:aggtime,n20}
     		\end{subfigure}
     		\begin{subfigure}{0.32\textwidth}
         		\centering
         			\includegraphics[width=\textwidth]{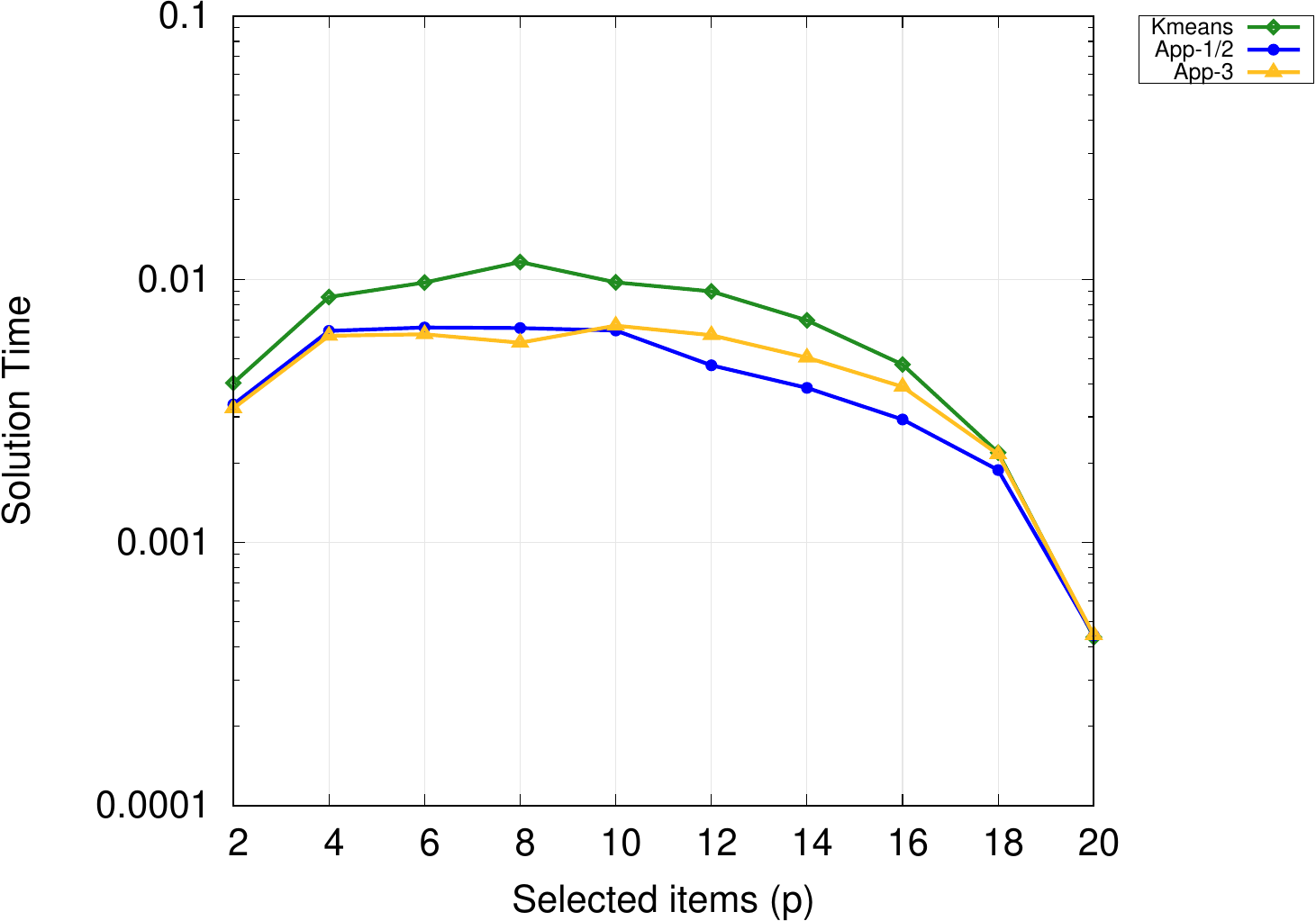}
         		\caption{Solution ($n=20$)}\label{fig:soltime,n20}
     		\end{subfigure}
     		\begin{subfigure}{0.32\textwidth}
         		\centering
         			\includegraphics[width=\textwidth]{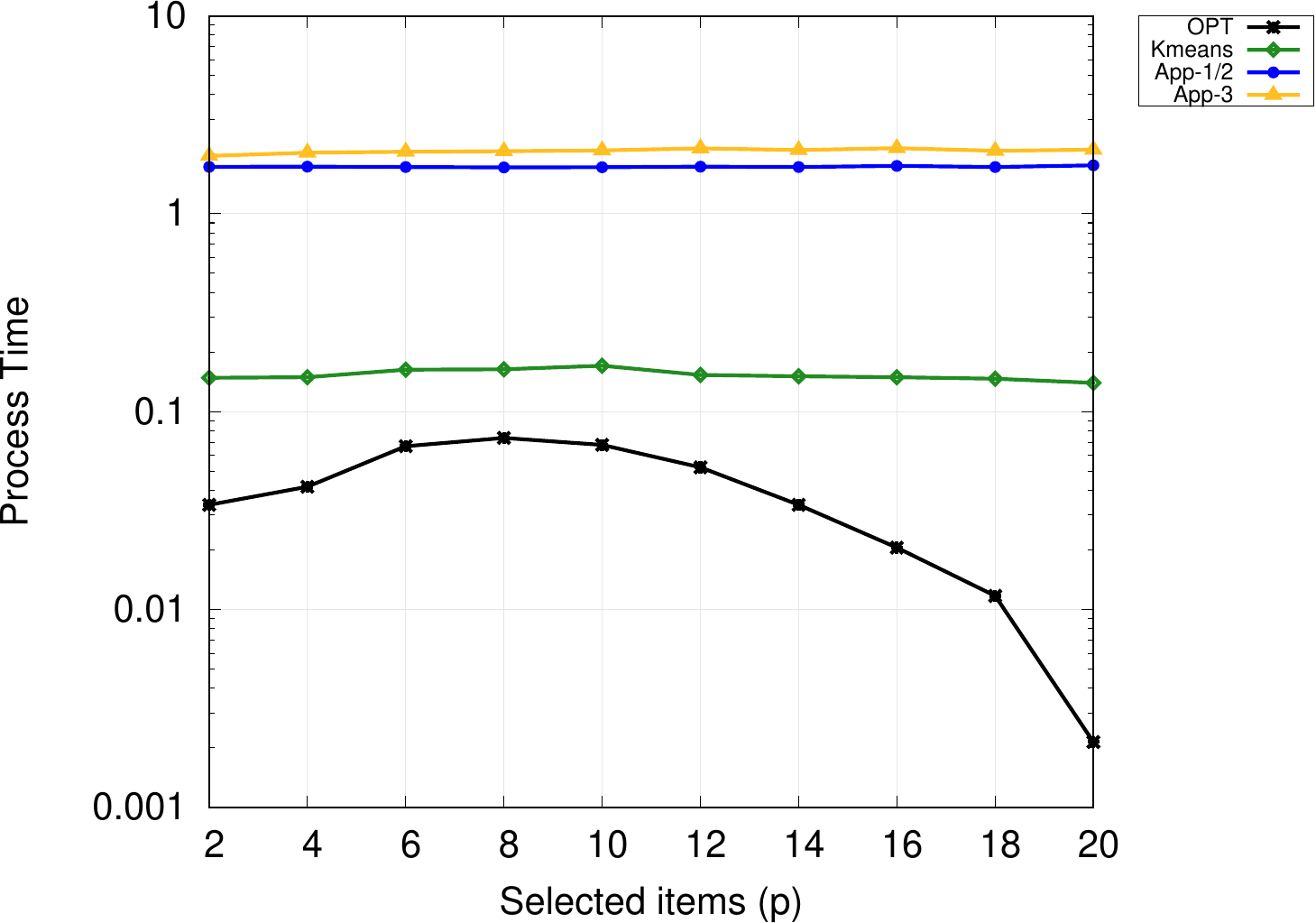}
         		\caption{Process ($n=20$)}\label{fig:protime,n20}
     		\end{subfigure} \\
		\begin{subfigure}{0.32\textwidth}
         		\centering
         			\includegraphics[width=\textwidth]{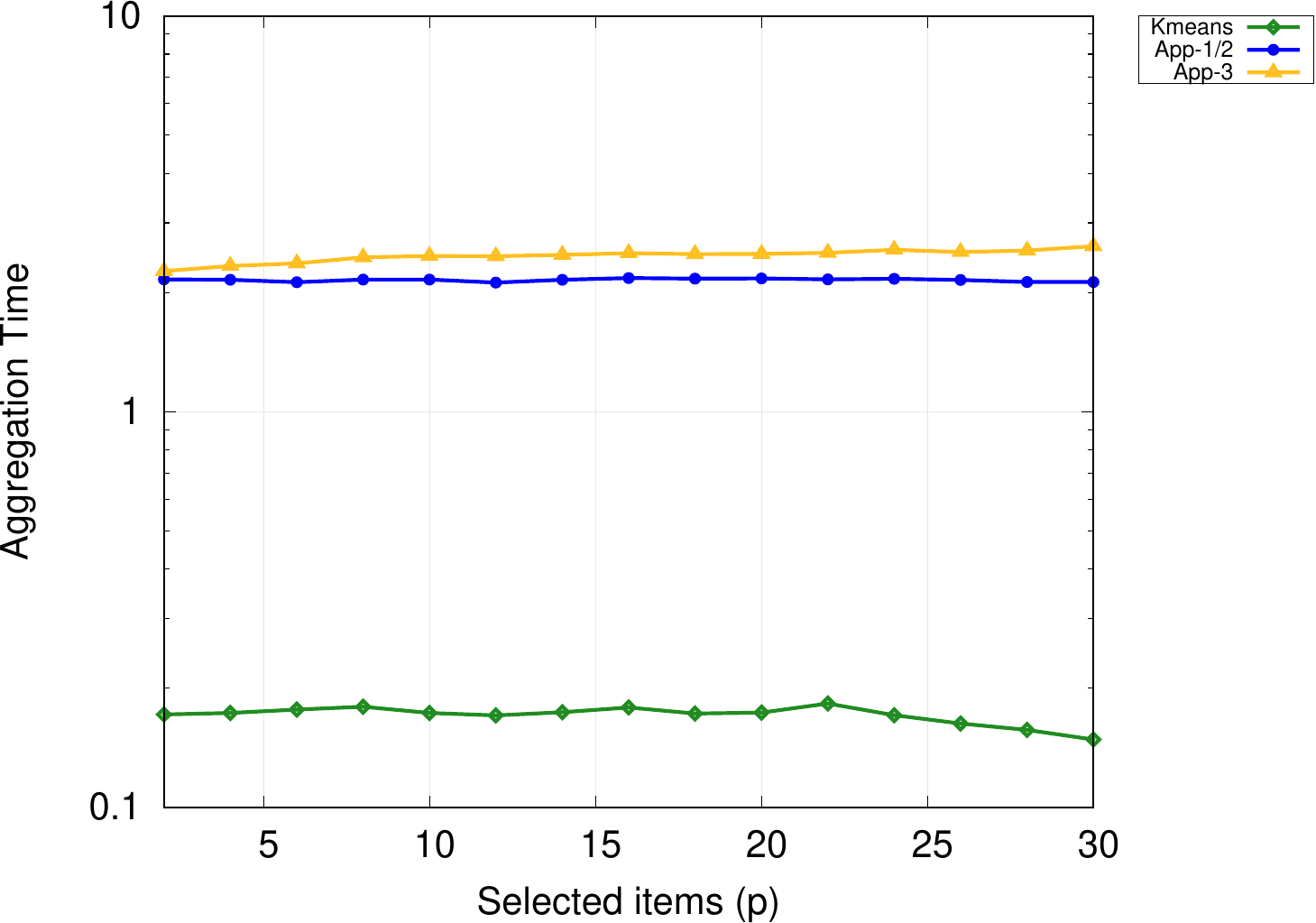}
         		\caption{Aggregation ($n=30$)}\label{fig:aggtime-n30}
     		\end{subfigure}
     		\begin{subfigure}{0.32\textwidth}
         		\centering
         			\includegraphics[width=\textwidth]{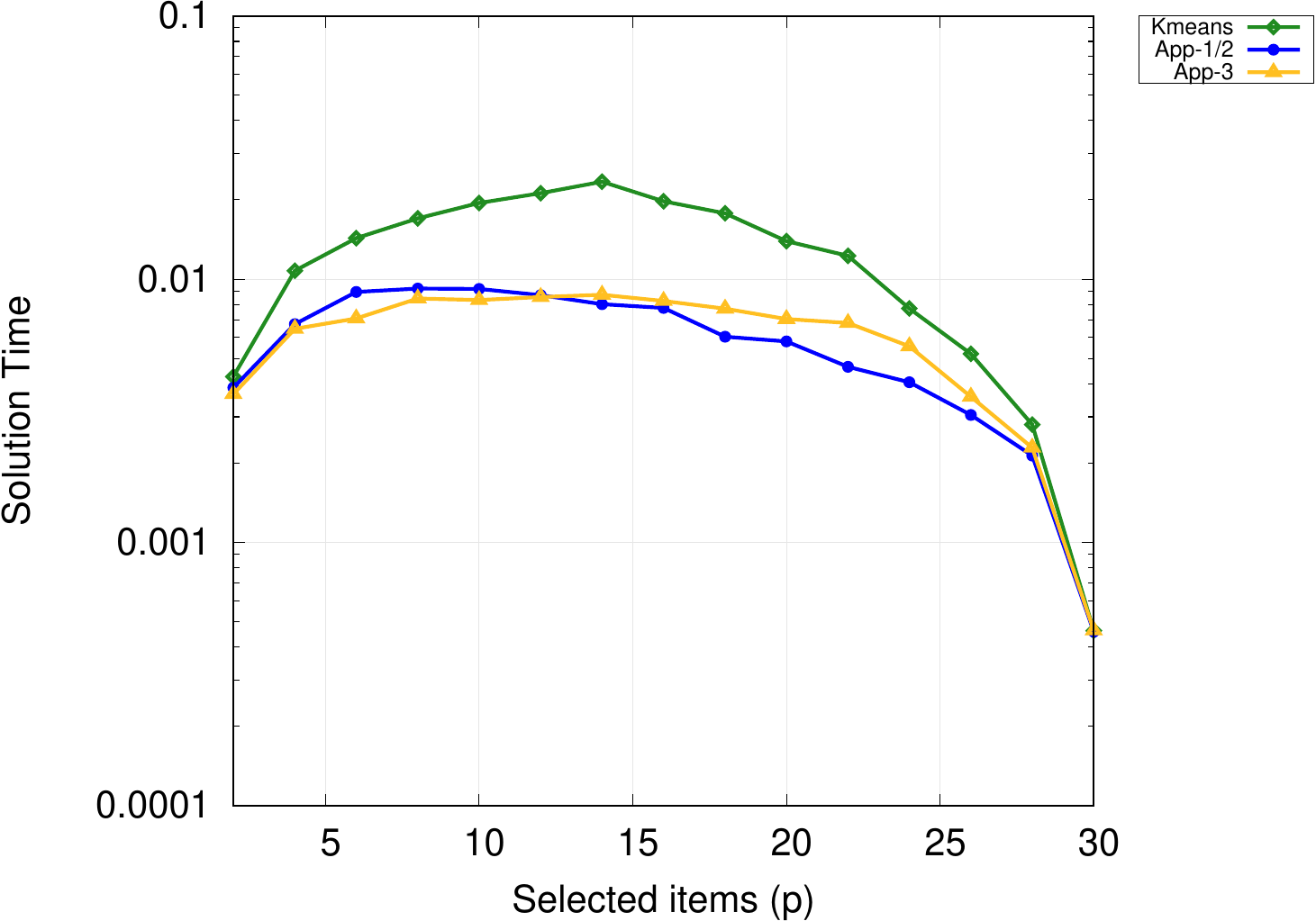}
         		\caption{Solution ($n=30$)}\label{fig:soltime-n30}
     		\end{subfigure}
     		\begin{subfigure}{0.32\textwidth}
         		\centering
         			\includegraphics[width=\textwidth]{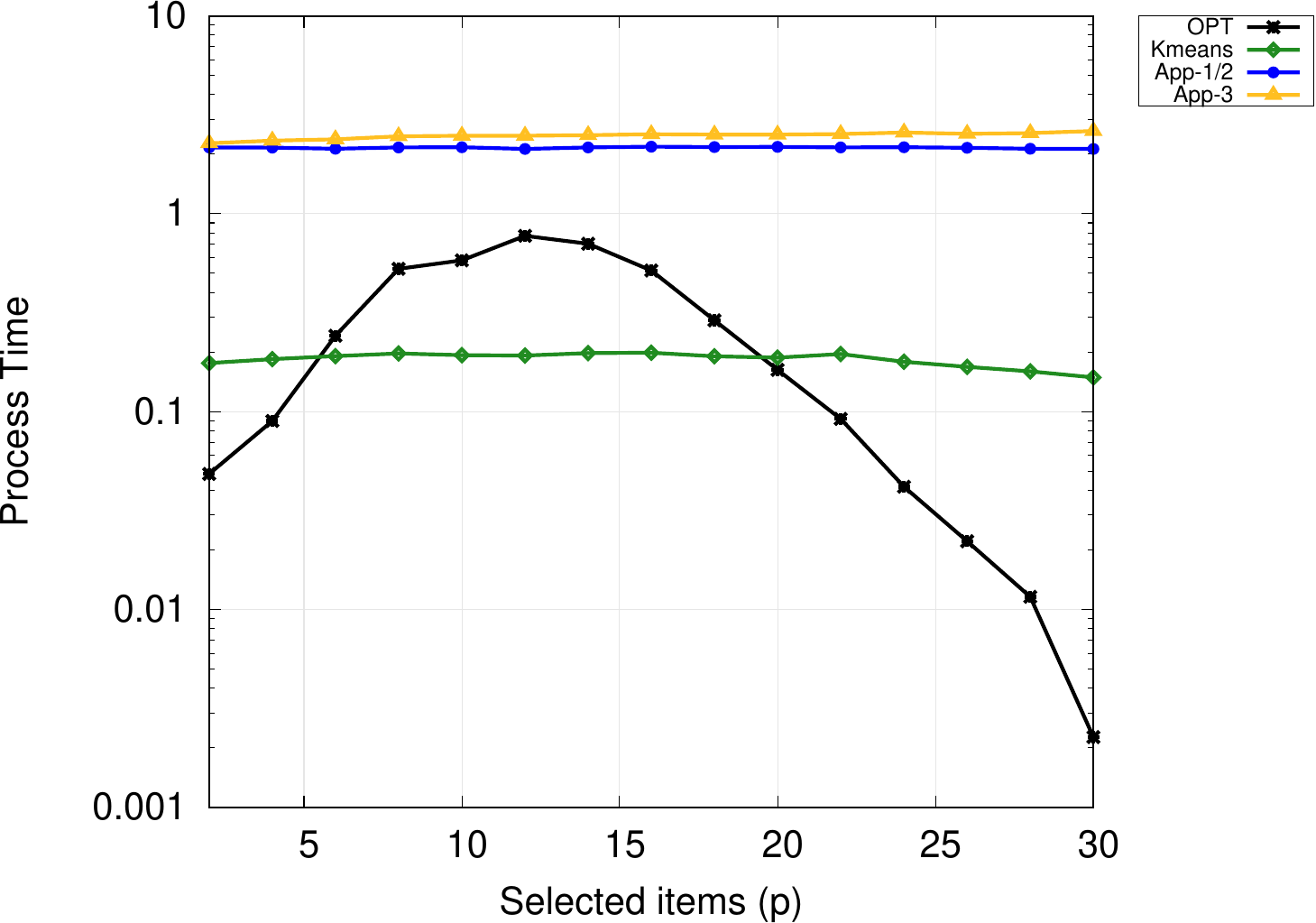}
         		\caption{Process ($n=30$)}\label{fig:protime,n30}
     		\end{subfigure} \\
		\begin{subfigure}{0.32\textwidth}
         		\centering
         			\includegraphics[width=\textwidth]{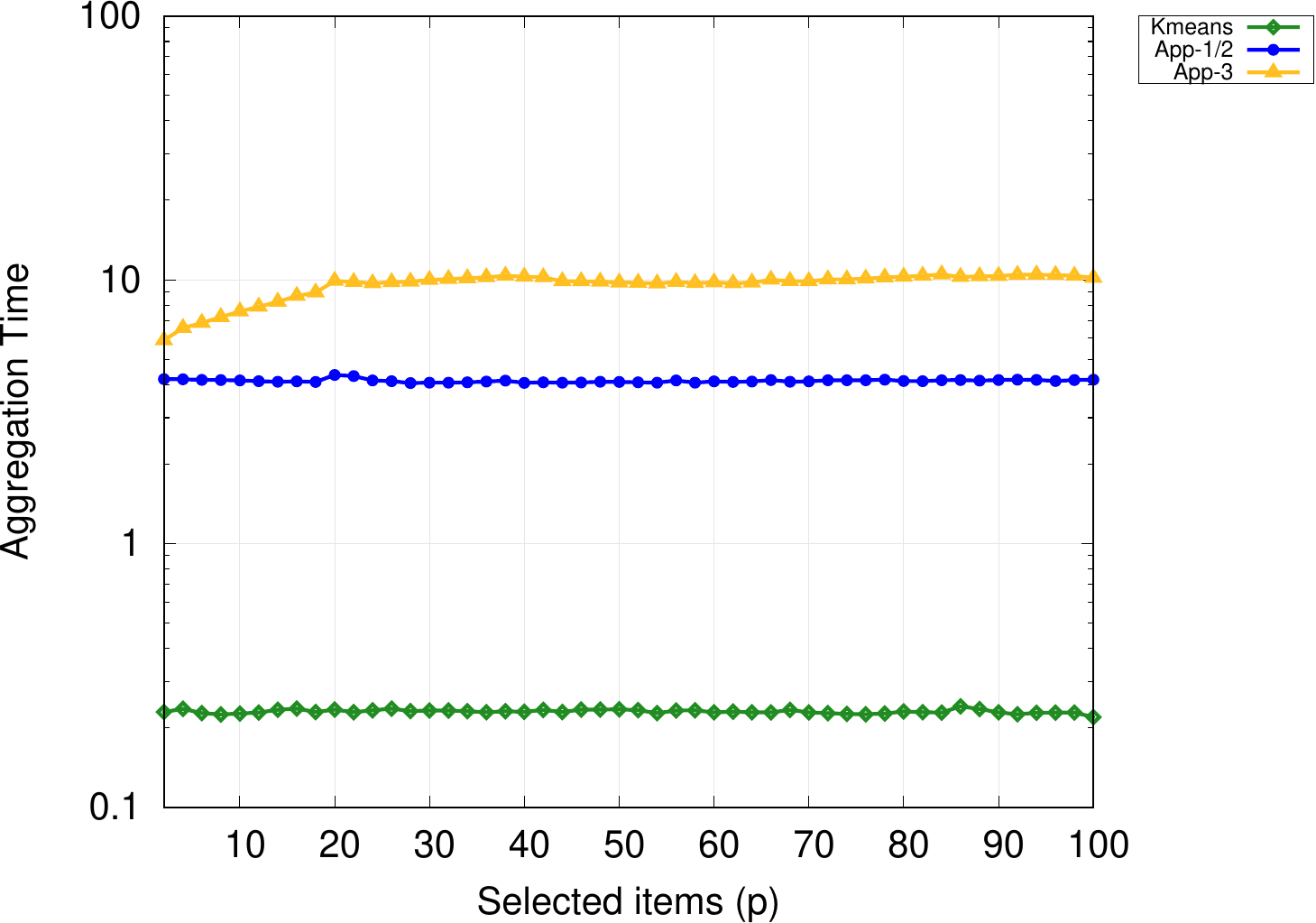}
         		\caption{Aggregation ($n=100$)}\label{fig:aggtime-n100}
     		\end{subfigure}
     		\begin{subfigure}{0.32\textwidth}
         		\centering
         			\includegraphics[width=\textwidth]{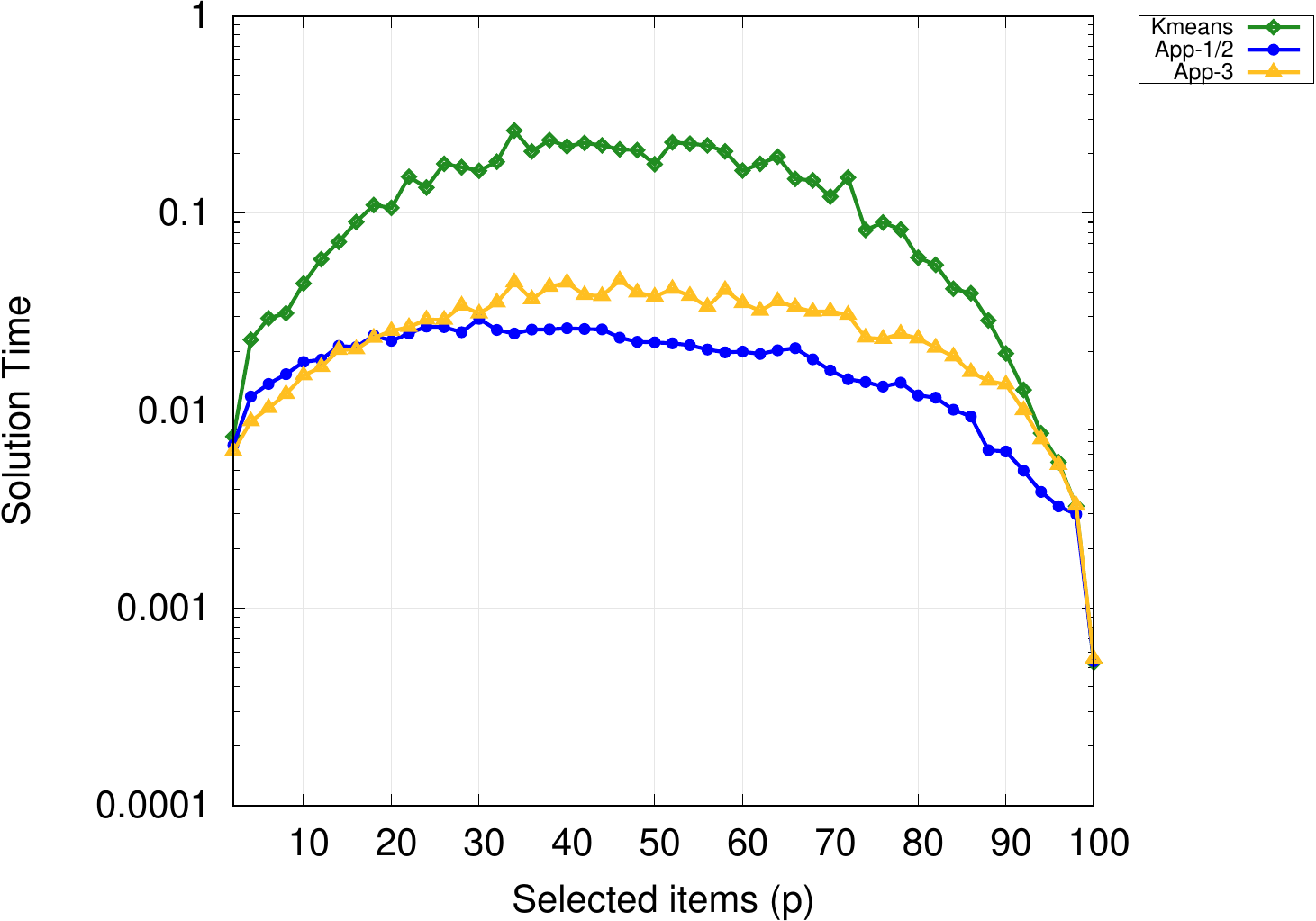}
         		\caption{Solution ($n=100$)}\label{fig:soltime-n100}
     		\end{subfigure}
     		\begin{subfigure}{0.32\textwidth}
         		\centering
         			\includegraphics[width=\textwidth]{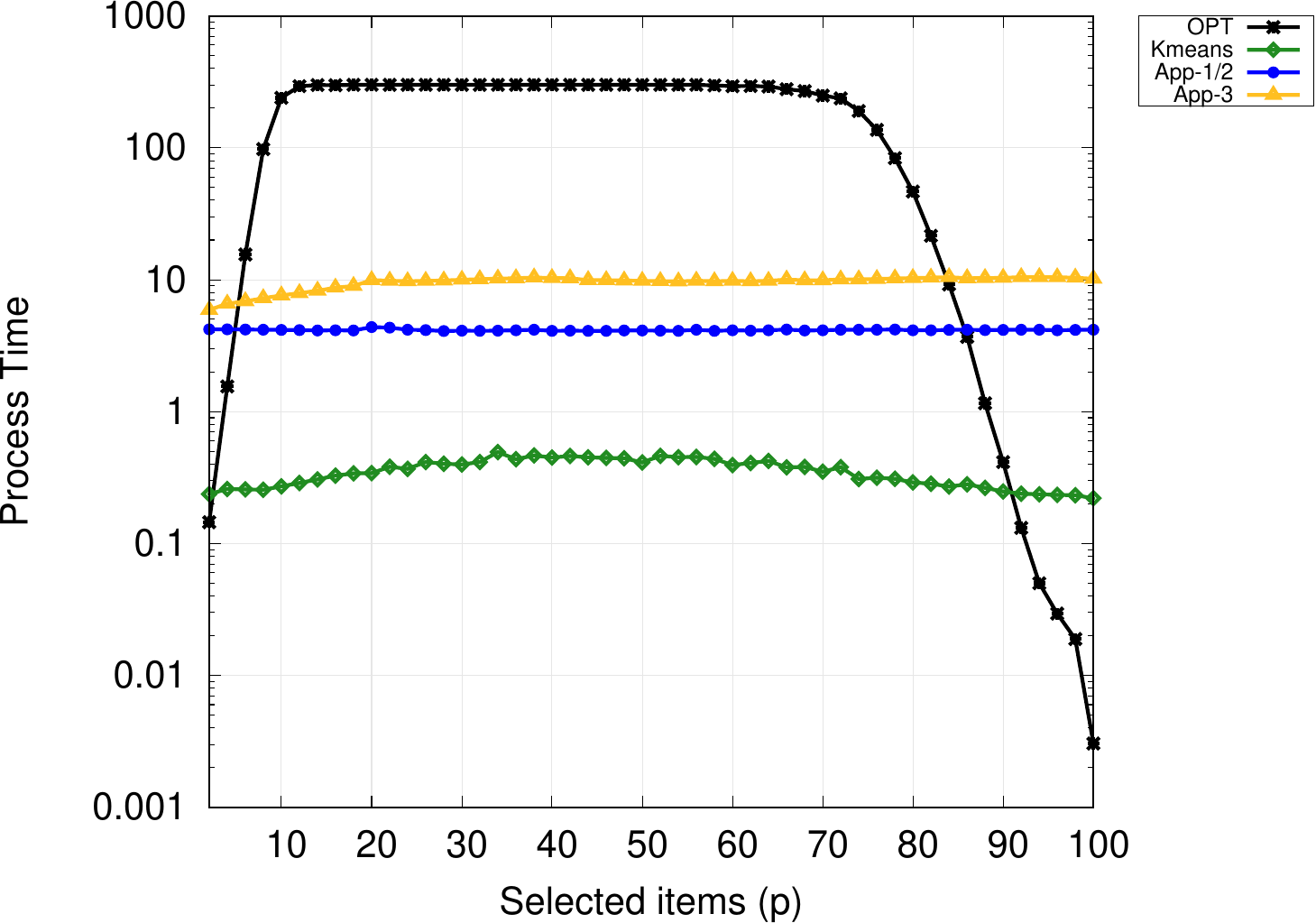}
         		\caption{Process ($n=100$)}\label{fig:protime,n100}
     		\end{subfigure}
        		\caption{Selection - time performance of large instances}\label{appfig:sel-time-performance-large}
\end{figure}


\begin{thebibliography}{WBVPS24}

\bibitem[ABV09]{aissi2009min}
Hassene Aissi, Cristina Bazgan, and Daniel Vanderpooten.
\newblock Min--max and min--max regret versions of combinatorial optimization
  problems: A survey.
\newblock {\em European Journal of Operational Research}, 197(2):427--438,
  2009.

\bibitem[Bd22]{bookbertsimasdenHertog2022}
Dimitris Bertsimas and Dick {den Hertog}.
\newblock {\em Robust and Adaptive Optimization}.
\newblock Dynamic Ideas LLC, 2022.

\bibitem[BGK18]{bertsimas2018data}
Dimitris Bertsimas, Vishal Gupta, and Nathan Kallus.
\newblock Data-driven robust optimization.
\newblock {\em Mathematical Programming}, 167(2):235--292, 2018.

\bibitem[BS03]{bertsimas2003robust}
Dimitris Bertsimas and Melvyn Sim.
\newblock Robust discrete optimization and network flows.
\newblock {\em Mathematical Programming}, 98(1-3):49--71, 2003.

\bibitem[BTEGN09]{ben2009robust}
Aharon Ben-Tal, Laurent El~Ghaoui, and Arkadi Nemirovski.
\newblock {\em Robust optimization}, volume~28.
\newblock Princeton University Press, 2009.

\bibitem[CG15]{chassein2015new}
Andr{\'e}~B Chassein and Marc Goerigk.
\newblock A new bound for the midpoint solution in minmax regret optimization
  with an application to the robust shortest path problem.
\newblock {\em European Journal of Operational Research}, 244(3):739--747,
  2015.

\bibitem[CG18]{chassein2018scenario}
Andr{\'e} Chassein and Marc Goerigk.
\newblock On scenario aggregation to approximate robust combinatorial
  optimization problems.
\newblock {\em Optimization Letters}, 12(7):1523--1533, 2018.

\bibitem[CGKZ20]{chassein2020approximating}
Andr{\'e} Chassein, Marc Goerigk, Adam Kasperski, and Pawe{\l} Zieli{\'n}ski.
\newblock Approximating combinatorial optimization problems with the ordered
  weighted averaging criterion.
\newblock {\em European Journal of Operational Research}, 286(3):828--838,
  2020.

\bibitem[Con12]{conde2012constant}
Eduardo Conde.
\newblock On a constant factor approximation for minmax regret problems using a
  symmetry point scenario.
\newblock {\em European Journal of Operational Research}, 219(2):452--457,
  2012.

\bibitem[DGKR03]{DupacovaEA03}
J.~Dupa\v{c}ov\'{a}, N.~Gr{\"o}we-Kuska, and W.~R{\"o}misch.
\newblock Scenario reduction in stochastic programming: An approach using
  probability metrics.
\newblock {\em Mathematical Programming}, 95(3):493--511, 2003.

\bibitem[Ehr05]{ehrgott2005multicriteria}
Matthias Ehrgott.
\newblock {\em Multicriteria optimization}, volume 491.
\newblock Springer Science \& Business Media, 2005.

\bibitem[FTW22]{fairbrother2019problem}
Jamie Fairbrother, Amanda Turner, and Stein~W Wallace.
\newblock Problem-driven scenario generation: an analytical approach for
  stochastic programs with tail risk measure.
\newblock {\em Mathematical Programming}, 191(1):141--182, 2022.

\bibitem[GH19]{goerigk2019representative}
Marc Goerigk and Martin Hughes.
\newblock Representative scenario construction and preprocessing for robust
  combinatorial optimization problems.
\newblock {\em Optimization Letters}, 13(6):1417--1431, 2019.

\bibitem[GH24]{robook}
Marc Goerigk and Michael Hartisch.
\newblock {\em An Introduction to Robust Combinatorial Optimization}, volume
  361 of {\em International Series in Operations Research \& Management
  Science}.
\newblock Springer, 2024.
\newblock Forthcoming.

\bibitem[GK22]{goerigk2022data}
Marc Goerigk and Jannis Kurtz.
\newblock Data-driven prediction of relevant scenarios for robust combinatorial
  optimization.
\newblock {\em arXiv preprint arXiv:2203.16642}, 2022.

\bibitem[GK23a]{goerigk2023optimal}
Marc Goerigk and Mohammad Khosravi.
\newblock Optimal scenario reduction for one-and two-stage robust optimization
  with discrete uncertainty in the objective.
\newblock {\em European Journal of Operational Research}, 310(2):529--551,
  2023.

\bibitem[GK23b]{goerigk2023data}
Marc Goerigk and Jannis Kurtz.
\newblock Data-driven robust optimization using deep neural networks.
\newblock {\em Computers \& Operations Research}, 151:106087, 2023.

\bibitem[{Gur}24]{gurobi}
{Gurobi Optimization, LLC}.
\newblock {Gurobi Optimizer Reference Manual}, 2024.

\bibitem[HR22]{henrion2022problem}
R{\'e}ne Henrion and Werner R{\"o}misch.
\newblock Problem-based optimal scenario generation and reduction in stochastic
  programming.
\newblock {\em Mathematical Programming}, 191(1):183--205, 2022.

\bibitem[KV18]{korte2011combinatorial}
Bernhard~H Korte and Jens Vygen.
\newblock {\em Combinatorial optimization}, volume~21 of {\em Algorithms and
  Combinatorics}.
\newblock Springer, 2018.

\bibitem[KW12]{king2012modeling}
Alan~J King and Stein~W Wallace.
\newblock {\em Modeling with stochastic programming}.
\newblock Springer, 2012.

\bibitem[KZ06]{kasperski2006approximation}
Adam Kasperski and Pawe{\l} Zieli{\'n}ski.
\newblock An approximation algorithm for interval data minmax regret
  combinatorial optimization problems.
\newblock {\em Information Processing Letters}, 97(5):177--180, 2006.

\bibitem[KZ16]{kasperski2016robust}
Adam Kasperski and Pawe{\l} Zieli{\'n}ski.
\newblock Robust discrete optimization under discrete and interval uncertainty:
  A survey.
\newblock In {\em Robustness analysis in decision aiding, optimization, and
  analytics}, pages 113--143. Springer, 2016.

\bibitem[LDD{\etalchar{+}}23]{Lubin2023}
Miles Lubin, Oscar Dowson, Joaquim {Dias Garcia}, Joey Huchette, Beno{\^i}t
  Legat, and Juan~Pablo Vielma.
\newblock {JuMP} 1.0: {R}ecent improvements to a modeling language for
  mathematical optimization.
\newblock {\em Mathematical Programming Computation}, 15:581–589, 2023.

\bibitem[Pow19]{powell2019unified}
Warren~B Powell.
\newblock A unified framework for stochastic optimization.
\newblock {\em European Journal of Operational Research}, 275(3):795--821,
  2019.

\bibitem[SHY17]{shang2017data}
Chao Shang, Xiaolin Huang, and Fengqi You.
\newblock Data-driven robust optimization based on kernel learning.
\newblock {\em Computers \& Chemical Engineering}, 106:464--479, 2017.

\bibitem[WBVPS24]{wang2023learning}
Irina Wang, Cole Becker, Bart Van~Parys, and Bartolomeo Stellato.
\newblock Learning decision-focused uncertainty sets in robust optimization.
\newblock {\em arXiv preprint arXiv:2305.19225}, 2024.

\bibitem[ZZ13]{zeng2013solving}
Bo~Zeng and Long Zhao.
\newblock Solving two-stage robust optimization problems using a
  column-and-constraint generation method.
\newblock {\em Operations Research Letters}, 41(5):457--461, 2013.

\end{thebibliography}
\end{document}